\documentclass[10pt, leqno]{article}
\usepackage{amsmath,amsfonts,amsthm,amscd,amssymb}
\numberwithin{equation}{section}
\theoremstyle{plain}
\newtheorem{theo}{Theorem}[section]
\newtheorem{prop}[theo]{Proposition}
\newtheorem{coro}[theo]{Corollary} 
\newtheorem{lemm}[theo]{Lemma}
\theoremstyle{definition}
\newtheorem{defi}[theo]{Definition}
\newtheorem{rema}[theo]{Remark}
\newtheorem{theo-defi}[theo]{Theorem-Definition}
\newtheorem{prop-defi}[theo]{Proposition-Definition}
\newtheorem{rema-defi}[theo]{Remark-Definition}
\newtheorem{exem-defi} [theo]{Example-Definiton}

\newtheorem{exem}[theo]{Example}

\def \al{\alpha}

\def \bul{\bullet}
\def \col{\colon}
\def \Del{\Delta}

\def \eps{\epsilon}
\def \Gam{\Gamma}

\def \inf{\infty}

\def \kap{\kappa}
\def \Lam{\Lambda}
\def \lam{\lambda}

\def \Lo{\Longrightarrow}
\def \lo{\longrightarrow}
\def \lom{\longmapsto}
\def \mab{\mathbb}

\def \Om{\Omega}
\def \om{\omega}
\def \ol{\overline}
\def \os{\overset}
\def \parno{\par\noindent}

\def \sig{\sigma}

\def \sus{\subset}

\def \ul{\underline}
\def \us{\underset}

\def \vp{\varpi}

\def \vpl{\varprojlim}

\def \wh{\widehat}
\def \wt{\widetilde}
\bigskip

\newcommand{\getsfrom}{\ensuremath{
\longleftarrow\kern-.50em\lower.0ex\hbox%
{$\shortmid\,$}}}

\begin{document}
\title{Steenbrink isomorphism and 
crystals on tubular neighbourhoods}
\author{Yukiyoshi Nakkajima
\date{}\thanks{\endgraf}}
\maketitle

\bigskip
\parno
{\bf Abstract.---}
For a locally nilpotent integrable connection on 
a proper (strict) semistable family over a small polydisc 
with a relative horizontal simple normal crossing divisor,  
we construct a canonical section in derived categories 
inducing an isomorphism from the log de Rham cohomology of it on the log special fiber 
of this family to the stalk of the higher direct image of it at the origin 
modulo the maximal ideal of the localization of the structure sheaf at the origin. 
As an application of the existence of this section, we prove the following: 
(1): the locally freeness of the higher direct image by a purely algebraic method; 
(2): the existence of a canonical isomorphism between the base change  
of the log de Rham cohomology of the log special fiber to the small polydisc 
and the higher direct image above. 
The result (2) tells us that the higher direct image has the crystalline nature: 
the invariance of the higher direct image. 
%We also prove analogous results for 
%a proper generalized strict semistable scheme   
%with a horizontal simple normal crossing divisor over a smooth scheme 
%over a field of characteristic $0$ by a purely algebraic method. 

\section{Introduction}
This article is a sequel of my article \cite{nf}. 
The most important aim of this article is to show that 
the higher direct image of 
the log de Rham complex of a locally nilpotent integrable connection on 
a proper strict semistable analytic family over a small unit polydisc 
is the resulting crystal of the log de Rham cohomology 
of the connection on the log special fiber of this family. 
To prove this, we construct an $\infty$-adic analogue of the  
Hyodo-Kato section in the $p$-adic case (\cite{hk}, \cite{ey}). 
%and to give an affirmative answer for a classical problem about  
%the construction of an isomorphism between 
%the (log) de Rham cohomologies of fibers of 
%a proper (generalized strict) semistable analytic space over a unit (poly) disc
%which is {\it independent} of the choice of a parameter of the unit (poly) disc 
%in a certain sense stated below. 
%More strongly we construct a canonical isomorphism 
%between the (log) de Rham cohomological complexes of the fibers 
%in the derived category of bounded below complexes of ${\mab C}$-vector spaces. 
%This is the best answer for the problem above. 
And then we prove the locally freeness of the higher direct image 
by a {\it purely algebraic} method. 
%except the proper base change theorem for a Hausdorff topological space. 
Lastly we construct a canonical isomorphism between the base change  
of the log de Rham cohomology to the small polydisc and the higher direct image, 
which implies the existence of the crystal above.   
This isomorphism induces an $\infty$-adic analogue of 
the Hyodo-Kato isomorphism in the $p$-adic case (\cite{bei}, \cite{nb}, \cite{ey}). 
We call this $\infty$-adic analogue the {\it Steenbrink isomorphism}.  
%The generalization tells us that the higher direct image has 
%the crystalline nature on a small polydisc.  
The existence of the crystal should have applications 
for certain classes of concrete degenerate log analytic spaces 
and (local) moduli problems in algebraic and analytic geometry. 
%As an example, we prove the injectivity result of the automorphism group 
%of a log K\"{a}hler $K3$ surface and we prove that 
%the necessary and sufficient condition for the extendability of 
%a morphism of the degenerations of abelian varieties. 
\par 
First let us recall Steenbrink's result.
\par 
Let $\os{\circ}{\Del}:=\{z\in {\mab C}~\vert~\vert z\vert <1\}$ 
be the unit disc and let $f\col \os{\circ}{\cal X}\lo \os{\circ}{\Del}$ be a proper 
(not necessarily strict) semistable family of 
analytic spaces over ${\mab C}$. 
Endow $\os{\circ}{\cal X}$ and 
$\os{\circ}{\Del}$ with the canonical log structures, respectively.  
Let ${\cal X}$ and $\Del$ be the resulting log analytic spaces, respectively. 
Let $s$ be the log point of ${\mab C}$: 
the log analytic space whose underlying analytic space is 
the origin $\{O\}$ and whose log structure is the inverse image of the log structure 
of $\Del$ by the inclusion $\{O\}\os{\subset}{\lo} \os{\circ}{\Del}$. 
Let $X/s$ be the log special fiber of ${\cal X}/\Del$. 
Let $\Om^{\bul}_{{\cal X}/\Del}$ and 
$\Om^{\bul}_{X/s}$ be the relative log de Rham complexes of 
${\cal X}/\Del$ and $X/s$, respectively. 
Set $H^q_{\rm dR}(X/s)=H^q(X,\Om^{\bul}_{X/s})$ $(q\in {\mab N})$. 
Then Steenbrink has proved the following theorem: 

\begin{theo}[{\rm \cite{sti}}]\label{theo:xds}
The higher direct image $R^qf_*(\Om^{\bul}_{{\cal X}/\Del})$ 
$(q\in {\mab N})$ is a coherent locally free ${\cal O}_{\Del}$-module 
$($and commutes with base change of a strict morphism of log analytic spaces$)$. 
\end{theo}

\parno
Set ${\cal X}^*:={\cal X}\setminus X$ 
and $\Del^*:=\Del \setminus \{O\}$. 
Let ${\mathfrak H}\lo \Del^*$ be the universal cover of $\Del^*$.  
Set $\ol{{\cal X}^*}:={\cal X}^*\times_{\Del^*}{\mathfrak H}$ 
and ${\cal X}_u:={\cal X}\times_{\Del}u$ for a point $u\in \Del^*$. 
Set $H^q_{\rm dR}({\cal X}_u/{\mab C}):=
H^q({\cal X}_u,\Om^{\bul}_{{\cal X}_u/{\mab C}})$ $(q\in {\mab N})$.  
The key point of the proof of (\ref{theo:xds}) is to prove that 
there exists the following isomorphism 
\begin{align*} 
H^q(\ol{{\cal X}^*},{\mab C})\os{\sim}{\lo} 
H^q_{\rm dR}(X/s)
\tag{1.1.1}\label{ali:xxc}
\end{align*} 
depending on a parameter $z$ of $\Del$. 
By using this isomorphism, we have an isomorphism 
\begin{align*} 
\rho_z \col H^q_{\rm dR}(X/s)
\os{\sim}{\lo} 
H^q({\cal X}_u,{\mab C})=H^q_{\rm dR}({\cal X}_u/{\mab C})
\tag{1.1.2}\label{ali:xtc}
\end{align*} 
depending on a parameter $z$ of $\Del$ for a point $u\in \Del^*$ 
and we obtain (\ref{theo:xds}) by Grauert's theorem (\cite{grh}) because 
$\dim_{\mab C}H^q_{\rm dR}(X/s)$ does not jump. 
We call the isomorphism (\ref{ali:xtc}) {\it Steenbrink's isomorphism}. 
%, 
%where $\ol{j} \col \wt{U}\lo {\cal X}$ is the composite morphism in the diagram above. 
%Then $H^q(\wt{\mathfrak X}},{\mab C})=H^q(X,R\Psi({\mab C})$ is isomorphic to 
Let $\ol{j}\col \ol{{\cal X}^*}\lo {\cal X}$ and $i\col X\os{\sus}{\lo} {\cal X}$ 
be the natural morphisms. Let 
$R\Psi({\mab C})=i^{-1}R\ol{j}_*({\mab C})$ be the nearby cycle sheaf
on $X$. To prove (\ref{ali:xxc}), Steenbrink has proved 
that there exists the following isomorphisms
\begin{align*}
R\Psi({\mab C})\os{\sim}{\lo} 
i^{-1}(\Om^{\bul}_{{\cal X}/{\mab C}}[\log z])\os{\sim}{\lo} \Om^{\bul}_{X/s}
\tag{1.1.3}\label{ali:xantc}
\end{align*}
 \cite{sti} (cf.~\cite{sga72})
depending on a parameter $z$ of $\Del$. 
Here $\Om^{\bul}_{{\cal X}/{\mab C}}[\log z]:={\mab C}[\log z]
\otimes_{\mab C}\Om^{\bul}_{{\cal X}/{\mab C}}$
is considered 
as a subdga of $\ol{j}_*(\Om^{\bul}_{\ol{{\cal X}^*}/{\mab C}})$.
The complex  
$\Om^{\bul}_{{\cal X}/{\mab C}}[\log z]$ 
is the Hirsch extension of 
$\Om^{\bul}_{{\cal X}/{\mab C}}$ 
by a vector space ${\mab C}\log z$ over ${\mab C}$ of rank 1 
with respect to a morphism 
$\varphi \col {\mab C}\log z \owns \log z \lom d\log z=z^{-1}dz
\in \Om^1_{{\cal X}/{\mab C}}=\Om^1_{\os{\circ}{\cal X}/{\mab C}}(\log \os{\circ}{X})$.  
The notion of the Hirsch extension by a vector space appears in the definition of 
the minimal model of a dga in Sullivan's theory (\cite{su}). 
%(See \cite{gm} and \cite{nhi} for the notion of Hirsch extension.)  
Conversely, if one knew (\ref{theo:xds}) a priori, one could construct 
a non-canonical isomorphism 
\begin{align*} 
H^q_{\rm dR}(X/s)
\os{\sim}{\lo} 
H^q_{\rm dR}({\cal X}_u/{\mab C})
\tag{1.1.4}\label{ali:xntc}
\end{align*} 
for a close point $u\in \Del^*$ from $O$ and any point $u\in \Del^*$. 
%which is independent of the choice of the parameter $z$. 
\par 
Let $f\col X\lo Y$ be a proper log smooth morphism of 
fs$($=fine and saturated$)$ log analytic spaces over ${\mab C}$. 
Assume that $Y$ is log smooth and that $f$ is exact. 
In \cite{kf} F.~Kato has proved that there exists the 
following canonical isomorphism 
\begin{align*} 
{\cal O}_{Y^{\log}}\otimes_{\mab C}Rf^{\log}_*({\mab C})
\os{\sim}{\lo} {\cal O}_{Y^{\log}}\otimes_{\eps^{-1}_Y({\cal O}_Y)}
\eps^{-1}_YRf_*(\Om^{\bul}_{X/Y}). 
\tag{1.1.5}\label{ali:osxky}
\end{align*} 
Here $X^{\log}$ and $Y^{\log}$ are 
the Kato-Nakayama spaces of $X$ and $Y$, respectively, defined in \cite{kn}, 
${\cal O}_{Y^{\log}}:={\cal O}_Y^{\log}$ in [loc.~cit.], 
$\eps_Y \col Y^{\log} \lo Y$ is the natural morphism of topological spaces 
defined in \cite{kn}, which has been denoted by $\tau$ in [loc.~cit.] 
and $f^{\log}\col X^{\log}\lo Y^{\log}$ is the induced morphism by $f$.  
Using the isomorphism (\ref{ali:osxky}) in the case where 
$Y$ is the log point of ${\mab C}$, 
F.~Kato has reproved the existence of the isomorphism 
(\ref{ali:xxc}) ([loc.~cit.]). 
\par 
In \cite{ikn} L.~Illusie, K.~Kato and C.~Nakayama
have obtained a more general theorem than the theorem (\ref{theo:xds}) 
by using the Kato-Nakayama space: 

\begin{theo}[{\rm \cite[(6.4)]{ikn}}]\label{theo:kikn}
Let $f\col X\lo Y$ be a proper log smooth morphism of 
fs$($=fine and saturated$)$ log analytic spaces over ${\mab C}$. 
Assume that $Y$ is log smooth or $f$ is exact. 
Assume also that ${\rm Coker}(\ol{M}_{f(x)}^{\rm gp}\lo \ol{M}_x^{\rm gp})$ 
is torsion free for any exact point $x\in X$. 
Let $\nabla \col {\cal F}\lo {\cal F}\otimes_{{\cal O}_X}\Om^{\bul}_{X/{\mab C}}$ 
be a locally nilpotent integrable connection on $X$ in the sense of {\rm \cite{kn}}. 
Let ${\cal F}\otimes_{{\cal O}_X}\Om^{\bul}_{X/Y}$ be the log de Rham complex 
obtained by $({\cal F},\nabla)$. 
Assume that any stalk of $\ol{M}_Y:=M_Y/{\cal O}_Y^*$ is isomorphic to 
the finitely many direct sum of ${\mab N}$. Then 
$R^qf_*({\cal F}\otimes_{{\cal O}_X}\Om^{\bul}_{X/Y})$ $(q\in {\mab N})$
is a locally free coherent ${\cal O}_Y$-module and commutes with base change. 
\end{theo}

\par
Let $V$ be a locally unipotent local system on $X^{\log}$ in the sense of 
{\rm \cite{kn}} corresponding to $({\cal F},\nabla)$ by the log Riemann-Hilbert correspondence in [loc.~cit.]: $V={\rm Ker}(\nabla \col  {\cal F}^{\log}\lo 
{\cal F}^{\log}\otimes_{{\cal O}_{X^{\log}}}\Om^{1,\log}_{X/{\mab C}})$. 
%Here $\eps_X \col X^{\log} \lo X$ is the natural morphism of topological spaces, 
%which has been denoted by $\tau$ in [loc.~cit.]. 
In [loc.~cit.] Illusie, K.~Kato and Nakayama 
have also proved that  
there exists a canonical isomorphism 
\begin{align*} 
{\cal O}_{Y^{\log}}\otimes^L_{\mab C}Rf^{\log}_*(V)
\os{\sim}{\lo} {\cal O}_{Y^{\log}}\otimes^L_{\eps^{-1}_Y({\cal O}_Y)}
\eps^{-1}_YRf_*({\cal F}\otimes_{{\cal O}_X}\Om^{\bul}_{X/Y}).  
\tag{1.2.1}\label{ali:osxy}
\end{align*}  
\par 
By using (\ref{theo:kikn}),  we see that, for each exact point 
$y\in Y$,  
%there exists a non-canonical section 
%$$\iota_{y*}R^qf_*({\cal F}\otimes_{{\cal O}_X}\Om^{\bul}_{X_y/y})\lo 
%R^qf_*({\cal F}\otimes_{{\cal O}_X}\Om^{\bul}_{X/Y})$$ 
%of the projection 
%$$R^qf_*({\cal F}\otimes_{{\cal O}_X}\Om^{\bul}_{X/Y})
%\lo \iota_{y*}R^qf_*({\cal F}\otimes_{{\cal O}_X}\Om^{\bul}_{X_y/y}),$$
%where $\iota_{y} \col y\os{\sus}{\lo} Y$ is the inclusion of $y$ into $Y$.  
%Furthermore, 
we have a non-canonical isomorphism 
$$H^q(X_y,{\cal F}\otimes_{{\cal O}_X}\Om^{\bul}_{X_y/y})
\os{\sim}{\lo} 
H^q(X_u,{\cal F}\otimes_{{\cal O}_X}\Om^{\bul}_{X_u/u})$$
for a sufficiently close exact point $u\in Y$ from $y$. 

\par 
On the other hand, the following $p$-adic case can be considered. 
%only in this introduction. 
\par 
Let ${\cal V}$ be a complete discrete valuation ring of 
mixed characteristics $(0,p)$ with perfect residue field $\kap$. 
Let $K$ be the fraction field of ${\cal V}$. 
Endow ${\rm Spec}({\cal V})$ 
with the log structure ${\cal V}\setminus \{0\}\os{\subset}{\lo}{\cal V}$
and let $S$ be the resulting log scheme. Let $s_{\kap}$ 
be the exact closed point of $S$.  
Let ${\cal X}$ be a proper log smooth scheme over $S$
whose log special fiber $X/s_{\kap}$ is of Cartier type. 
Let $\pi$ be a uniformizer of ${\cal V}$. 
Let ${\cal W}$ be the Witt ring of $\kap$ and $K_0$ the fraction field of ${\cal W}$. 
Let ${\cal W}(s_{\kap})$ be the canonical lift of $s_{\kap}$ 
over ${\rm Spf}({\cal W})$. 
Let ${\cal X}_K$ be the log generic fiber of ${\cal X}$ over $S$. 
In \cite{hk} Hyodo and K.~Kato have constructed the following isomorphism 
\begin{align*} 
\rho_{\pi} \col H^q_{\rm crys}(X/{\cal W}(s_{\kap}))\otimes_{\cal W}K
\os{\sim}{\lo} 
H^q_{\rm dR}({\cal X}_K/K)\quad (q\in {\mab N})
\tag{1.2.2}\label{ali:crxw}
\end{align*}
depending on the uniformizer $\pi$. (Because 
the proof of the existence of this isomorphism in [loc.~cit.] is not complete, 
see also \cite{ndw} and \cite{nb}.)
The key point of the proof of (\ref{ali:crxw}) is to prove that there exists 
a unique section 
\begin{align*} 
K_0\otimes^L_{\cal W}R\Gam_{\rm crys}(X/{\cal W})
\lo
K_0\otimes^L_{\cal W}R\Gam_{\rm crys}(X/\wh{{\cal W}\langle t\rangle})
\tag{1.2.3}\label{ali:canc}
\end{align*}
of the projection 
\begin{align*} 
K_0\otimes^L_{\cal W}R\Gam_{\rm crys}(X/\wh{{\cal W}\langle t\rangle})
\lo
K_0\otimes^L_{\cal W}R\Gam_{\rm crys}(X/{\cal W})
\end{align*} 
which induces the following isomorphism: 
\begin{align*} 
K_0\otimes^L_{\cal W}
\wh{{\cal W}\langle t\rangle} \otimes^L_{\cal W}R\Gam_{\rm crys}(X/{\cal W})
\os{\sim}{\lo}
K_0\otimes_{\cal W}^LR\Gam_{\rm crys}(X/\wh{{\cal W}\langle t\rangle}). 
\tag{1.2.4}\label{ali:cdrxw}
\end{align*}
Here ${\cal W}\langle t\rangle$ is the PD-polynomial algebra in one variable 
over ${\cal W}$ and $\wh{\quad}$ means the $p$-adic completion. 
The $n$-th power of the Frobenius endomorphism on 
$${\rm ``}{\rm Ker}(R\Gam_{\rm crys}(X/\wh{{\cal W}\langle t\rangle})\lo 
R\Gam_{\rm crys}(X/\wh{\cal W})){\textrm "}$$ tends to a zero morphism for $n\to \infty$, 
while 
the $n$-th power of the Frobenius endomorphism on 
$K_0\otimes^L_{\cal W}R\Gam_{\rm crys}(X/{\cal W})$ 
is an isomorphism for all $n\in {\mab N}$. 
By this fact, we obtain the canonical section 
(\ref{ali:canc}) (as in the case of the Teichm\"{u}ller lift) with complicated calculations. 
\par 
By modifying $\rho_{\pi}$, 
we have constructed a canonical contravariantly functorial isomorphism 
\begin{align*} 
\rho_p \col H^q_{\rm crys}(X/{\cal W}(s_{\kap}))\otimes_{\cal W}K
\os{\sim}{\lo} 
H^q_{\rm dR}({\cal X}_K/K) \quad (q\in {\mab N})
\tag{1.2.5}\label{ali:ckrxw}
\end{align*}
which is independent of the choice of $\pi$ (\cite{nb}). 
We use the prime number $p\in {\cal V}$ 
and the isomorphism ``$\rho_p$'' in spite that 
$p$ is not necessarily a uniformizer of ${\cal V}$. 
The isomorphism $\rho_p$ is compatible with any algebraic extension of $K$ ([loc.~cit.]). 
We call this isomorphism the {\it Hyodo-Kato isomorphism}. 
In [loc.~cit.] we have conjectured that
%However we do not know whether 
there exists a canonical isomorphism 
\begin{align*} 
\rho_p \col R\Gam_{\rm crys}(X/{\cal W}(s_{\kap}))\otimes^L_{\cal W}K
\os{\sim}{\lo} 
R\Gam_{\rm dR}({\cal X}_K/K)
\tag{1.2.6}\label{ali:ckraxw}
\end{align*}
in the derived category of bounded complexes of $K$-vector spaces, which we call  
the {\it derived} Hyodo-Kato isomorphism (cf.~\cite{bei}). 
%though we have conjectured this existence (\cite[Conjecture 0.1.3]{nb}). 
Recently Ertl and Yamada have proved this conjecture in an ingenious method 
when ${\cal X}$ is a proper strictly semistable scheme over $S$ 
(\cite{ey}). They have proved the existence of the derived Hyodo-Kato isomorphism 
(\ref{ali:ckraxw}) 
in this case by using the Hirsch extension.  
%and the completed Hirsch extension. 
%They have noticed the big difference of the Hirsch extension 
%and the completed one in general, though they are equal in a certain case. 
\par 
Now let us come back to the case where the base field is ${\mab C}$. 
It is a natural problem to ask whether there exists a canonical isomorphism 
between $R\Gam_{\rm dR}(X/s)$ and $R\Gam({\cal X}_u,{\mab C})
=R\Gam_{\rm dR}({\cal X}_u/{\mab C})$ for a sufficiently 
close point $u\in \Del^*$ from $O$. 
Let $\iota_O \col \{O\} \os{\subset}{\lo} \Del$ and $\iota \col X\os{\sus}{\lo} {\cal X}$ 
be the closed immersions. 
%Let $p\col \Del \lo {\rm Spec}({\mab C})^{\rm an}$ be the projection. 
Though we do not know whether there exists a canonical contravariantly functorial 
section 
\begin{align*}
R\Gam_{\rm dR}(X/s)\lo R\Gam({\cal X},\Om^{\bul}_{{\cal X}/\Del})
\tag{1.2.7}\label{ali:dndrxa}
\end{align*} 
of the natural morphism 
$R\Gam({\cal X},\Om^{\bul}_{{\cal X}/\Del})\lo R\Gam_{\rm dR}(X/s)$ 
(this seems us impossible), 
we construct a canonical contravariantly functorial section 
\begin{align*} 
\iota_{O*}R\Gam_{\rm dR}(X/s)\lo 
Rf_*\iota_*\iota^{-1}(\Om^{\bul}_{{\cal X}/\Del})
\tag{1.2.8}\label{ali:dxda}
\end{align*} 
of the natural morphism 
$Rf_*\iota_*\iota^{-1}(\Om^{\bul}_{{\cal X}/\Del})\lo 
\iota_{O*}R\Gam_{\rm dR}(X/s)$ induced by 
the natural morphism 
\begin{align*} 
\iota^{-1}(\Om^{\bul}_{{\cal X}/\Del})\lo 
\iota^{!}(\Om^{\bul}_{{\cal X}/\Del})=\Om^{\bul}_{X/s}
\tag{1.2.9}\label{ali:dxsda}
\end{align*} 
as an example of a more general section without using the topological spaces 
$\ol{{\cal X}^*}$ nor $X^{\log}$. 
(In fact, the morphism (\ref{ali:dxsda}) is a quasi-isomorphism by 
\cite{sti}.)
%However 
Moreover, if $u\in \Del^*$ is sufficiently close from $O$, 
then we prove that there exists a canonical isomorphism 
\begin{align*} 
\rho \col H^q_{\rm dR}(X/s)\os{\sim}{\lo} H^q_{\rm dR}({\cal X}_u/{\mab C})
%R\Gam_{\rm dR}(X/s)\os{\sim}{\lo}
%R\Gam_{\rm dR}({\cal X}_u/\{u\}) 
\quad (q\in {\mab N})
\tag{1.2.10}\label{ali:dxa}
\end{align*} 
without using $\ol{{\cal X}^{\log}}$ nor ${\cal X}^{\log}$.  
%by using ${\cal X}^{\log}$ at only two points:  
%we use Usui's main result in \cite{ur1} which tells us that 
%${\cal X}^{\log} \lo \Del^{\log}$ is a fiber bundle and we use 
%the log Riemann-Hilbert correspondence in \cite{kn} and \cite{ikn}. 
%By using the isomorphism (\ref{ali:dxa}) 
%we can construct a canonical isomorphism 
%\begin{align*} 
%R\Gam_{\rm dR}(X/s)\os{\sim}{\lo}
%R\Gam_{\rm dR}({\cal X}_u/\{u\}). 
%\tag{1.2.9}\label{ali:dxda}
%\end{align*} 
%(We can generalize the isomorphisms (\ref{ali:dxa}) and 
%(\ref{ali:dxda}) for any point $u\in \Del^*$ by connecting 
%any $u$ with a close point of $\Del^*$ by a segment.)
%in a certain sense (see the text for this in detail). 
Let $M_s$ be the log structure of $s$. 
The section (\ref{ali:dxda}) is independent of 
the choice of a parameter $z$ of $\Del$; 
the isomorphism (\ref{ali:dxa}) is also independent of 
the choice of a parameter $z$ of $\Del$ in our sense:  
we prepare a parameter $z$ of $\Del$ according to 
a global section of the log structure of $s$ whose image in 
$\ol{M}_s:=M_s/{\cal O}_s^*={\mab N}$ is the generator. 
The choice of $z$ is determined up to ${\mab C}^*$ not ${\cal O}_{\Del}^*$.
If we replace $z$ by $cz$ for $c\in {\mab C}^*$, 
then the isomorphism (\ref{ali:dxa}) does {\it not} change. 
One can consider a similar isomorphism to (\ref{ali:dxa}) by using 
a convergent power series $\al(z)$ around $\{O\}$ instead of $z$, 
then the similar isomorphism changes by ${\rm exp}(-\al(0)N)$, where $N\col 
H^q_{\rm dR}(X/s)\lo H^q_{\rm dR}(X/s)(-1)$ is the monodromy operator. 
The existence of the canonical section (\ref{ali:dxda}) 
and the canonical isomorphism (\ref{ali:dxa}) have not 
been conjectured. 
The hope in \cite[0.2.23]{ku} that the theory of log geometry 
in the sense of Fontaine-Illusie-Kato is useful 
for the case where the base field is of characteristic $0$ 
as in the case of characteristic $p>0$ or mixed characteristics 
comes true in our setting. 
%(See the text for this in detail.)
%Let $M_s$ be the log structure of $s$. 
%The isomorphism (\ref{ali:dxa}) is independent of 
%the choice of a parameter $z$ of $\Del$ in our sense.
%That is,  we prepare a parameter $z$ of $\Del$ according to 
%a global section of the log structure of $s$ whose image in 
%$\ol{M}_s:=M_s/{\cal O}_s^*={\mab N}$ is the generator; 
%the choice of $z$ is determined up to ${\mab C}^*$ not on ${\cal O}_{\Del}^*$.
To construct the section (\ref{ali:dxda}), 
we do {\it not} calculate the nearby cycle sheaf 
$R\Psi({\mab C})$ (especially we do not use the space $\ol{{\cal X}^*}$). 
We use only the Hirsch extension of a dga 
as in \cite{ey} (and \cite{nhi}) by believing in a philosophy 
that dga's have informations of geometry (modulo torsion). 
Though $\iota_{O*}R\Gam_{\rm dR}(X/s)$ has only the information of the log special fiber, 
$Rf_*\iota_*\iota^{-1}(\Om^{\bul}_{{\cal X}/\Del})$ has an information of 
the general fiber of ${\cal X}/\Del$ near $X/s$.  
We have constructed a canonical lift to  
$Rf_*\iota_*\iota^{-1}(\Om^{\bul}_{{\cal X}/\Del})$ from $\iota_{O*}R_{\rm dR}(X/s)$.   
Since 
the topology of a (log) analytic space 
is finer than the Zariski topology of an algebraic scheme and since 
the different points in the analytic setting are more far than the different points 
in the algebraic case, it remains a little surprising for us.
\par 
Our method is nothing but a log analytic version of 
Ertl-Yamada's one in the log rigid analytic case (\cite{ey}). 
Though ${\cal O}_{\Del}$ is not complete unlike in the $p$-adic case, 
we heavily depend on their idea in this article. 
We show that their method is also applicable 
for the log analytic case over ${\mab C}$ by paying careful attentions to 
the proof of the existence of the canonical section (\ref{ali:dxda}) and 
the canonical isomorphism (\ref{ali:dxa}). 
(These attentions are not necessary in the $p$-adic case because the underlying scheme 
${\rm Spec}({\cal V})$ of the base log scheme is complete and we have only to 
consider (log) formal schemes by Grothendieck's formal function theorem.)  
Though we do not use any result in  \cite{ey} and \cite{nb} in this article, 
we are psychologically supported by the existence of the 
Hyodo-Kato section (\ref{ali:canc}).  
However, even if $u\in \Del^*$ is sufficiently close from $O$, 
we do not know whether there exists 
a canonical isomorphism 
\begin{align*} 
R\Gam_{\rm dR}(X/s)\os{\sim}{\lo}
R\Gam_{\rm dR}({\cal X}_u/\{u\}).  
\tag{1.2.11}\label{ali:dnnxa}
\end{align*}

We prove the following as a special case of our result: 

\begin{theo}\label{theo:rpi}
Let $r_1$ be a positive integer and let $r_2$ be a nonnegative integer. 
Set $B:=\Del^{r_1}\times \os{\circ}{\Del}{}^{r_2}$.  
%$($the $r$-times direct product of $\Del)$. 
Let $f\col {\cal X}\lo B$ be a $($not necessarily proper$)$ 
$($not necessarily strictly$)$-semistable analytic space  
with horizontal NCD$($=normal crossing divisor$)$ ${\cal D}$. 
$($See the text for the notion of the semistable analytic space.  
%$($here  stands for ``locally product''$)$.$)$
%Let $V$ be a local system on $\os{\circ}{\cal X}$ of 
%finite dimensional vector spaces over ${\mab C}$. 
%Set ${\cal E}:=V\otimes_{\mab C}{\cal O}_{\cal X}$ and 
%$\nabla:={\rm id}_V\otimes_{\mab C}d \col 
Let $({\cal E},\nabla)$ be a locally nilpotent integrable connection 
on $({\cal X},{\cal D})$ in the sense of {\rm \cite{kn}:} $\nabla \col {\cal E}\lo  
{\cal E}\otimes_{{\cal O}_{\cal X}}\Om^1_{{\cal X}/{\mab C}}(\log {\cal D})$. 
Let ${\cal E}\otimes_{{\cal O}_{\cal X}}\Om^{\bul}_{{\cal X}/B}(\log {\cal D})$ 
be the log de Rham complex obtained by 
$({\cal E},\nabla)$. Let $\{O\}$ be the origin of $B$. 
Endow $\{O\}$ with the pull-back of the log structure of $B$. 
Let $s$ be the resulting log analytic space. 
Set $S:=s\times \os{\circ}{\Del}{}^{r_2}$, 
$X:={\cal X}\times_BS$ and $D:={\cal D}\times_BS$. 
Let 
%$\iota_S \col S\os{\sus}{\lo} B$ and 
$\iota \col X\os{\sus}{\lo} {\cal X}$ be the natural exact closed immersion. 
Set $f_S:=f\circ \iota \col X\lo B$. 
%Let $f_S\col X\lo S$ be the structural morphism. 
%Let ${\cal E}$ be a locally free coherent ${\cal O}_{\cal X}$-module. 
%Let $\nabla \col {\cal E}\lo {\cal E}\otimes_{{\cal O}_{\cal X}}
%\Om^1_{{\cal X}/B}(\log {\cal D})$ 
%be an integrable connection which is a locally successive extension of 
%integrable connections without log poles along $S$. 
%$($See the text for the precise definition of this successive extension.$)$
Then there exists a canonical contravariantly functorial section
\begin{align*} 
{\cal E}\otimes_{{\cal O}_{\cal X}}\Om^{\bul}_{X/S}(\log D)
\lo 
\iota^{-1}({\cal E}\otimes_{{\cal O}_{\cal X}}\Om^{\bul}_{{\cal X}/B}(\log {\cal D}))
\tag{1.3.1}\label{ali:mocx}
\end{align*}
of the natural morphism 
\begin{align*}
\iota^{-1}({\cal E}\otimes_{{\cal O}_{\cal X}}\Om^{\bul}_{{\cal X}/B}(\log {\cal D}))
\lo 
{\cal E}\otimes_{{\cal O}_{\cal X}}\Om^{\bul}_{X/S}(\log D)
\tag{1.3.2}\label{ali:moomx}
\end{align*} 
in the derived category $D^+(\iota^{-1}f^{-1}({\cal O}_B))$ of 
bounded below complexes of $\iota^{-1}f^{-1}({\cal O}_B)$-modules. 
Consequently there exists a canonical contravariantly functorial section 
\begin{align*} 
%\iota_{S*}
Rf_{S*}({\cal E}\otimes_{{\cal O}_{\cal X}}\Om^{\bul}_{X/S}(\log D))
\lo 
Rf_{S*}\iota^{-1}({\cal E}\otimes_{{\cal O}_{\cal X}}\Om^{\bul}_{{\cal X}/B}(\log {\cal D}))
\tag{1.3.3}\label{ali:mox}
\end{align*}
of the natural morphism 
\begin{align*}
Rf_{S*}\iota^{-1}({\cal E}\otimes_{{\cal O}_{\cal X}}\Om^{\bul}_{{\cal X}/B}(\log {\cal D}))
\lo 
Rf_{S*}({\cal E}\otimes_{{\cal O}_{\cal X}}\Om^{\bul}_{X/S}(\log D)) 
\tag{1.3.4}\label{ali:moox}
\end{align*} 
%Consequently  
%$R^qf_*({\cal E}\otimes_{{\cal O}_{\cal X}}\Om^{\bul}_{{\cal X}/B}(\log {\cal D}))$
%$(q\in {\mab N})$ is a locally free coherent ${\cal O}_B$-module and 
%commutes with base change.
in the derived category $D^+({\cal O}_B)$ of the bounded below complexes
of ${\cal O}_B$-modules. 
%This section induces the following isomorphisms  
%\begin{align*} 
%{\cal O}_B\otimes^L_{{\cal O}_S}
%Rf_{S*}({\cal E}\otimes_{{\cal O}_{\cal X}}\Om^{\bul}_{X/S}(\log D))
%\os{\sim}{\lo} 
%Rf_*({\cal E}\otimes_{{\cal O}_{\cal X}}\Om^{\bul}_{{\cal X}/B}(\log {\cal D}))
%\tag{1.3.5}\label{ali:mxbox}
%\end{align*} 
%and 
%\begin{align*} 
%{\cal O}_B\otimes_{{\cal O}_S}
%R^qf_{S*}({\cal E}\otimes_{{\cal O}_{\cal X}}\Om^{\bul}_{X/S}(\log D))
%\os{\sim}{\lo} 
%R^qf_*({\cal E}\otimes_{{\cal O}_{\cal X}}\Om^{\bul}_{{\cal X}/B}(\log {\cal D}))
%\quad (q\in {\mab N})
%\tag{1.3.6}\label{ali:mxqbox}
%\end{align*} 
%around $S$ if $\os{\circ}{\cal X}$ is proper over $\os{\circ}{B}$. 
\end{theo}

\par 
%Consider the case $r_2=0$. 
%Assume that $\os{\circ}{\cal X}$ is proper over $\os{\circ}{B}$. 
%Set ${\cal U}:={\cal X}\setminus {\cal D}$. 
%Let $u$ be any exact point of $B$ whose log structure is trivial. 
%As a corollary of the isomorphism (\ref{ali:mxqbox}), we obtain 
%a canonical contravariantly functorial isomorphism 
%\begin{align*} 
%R\Gam(X,{\cal E}\otimes_{{\cal O}_{\cal X}}\Om^{\bul}_{X/s}(\log D))
%\os{\sim}{\lo} 
%R\Gam({\cal U}_u,{\cal E}\otimes_{{\cal O}_{\cal X}}\Om^{\bul}_{{\cal U}_u/u}). 
%\tag{1.3.7}\label{ali:matbox}
%\end{align*} 
%if $u$ is close from $O$.  
%If $({\cal E}\vert_{{\cal U}_u},\nabla)
%=({\cal O}_{{\cal U}_u}\otimes_{\mab C}V,d\otimes {\rm id}_V)$ 
%for a local system $V$ on ${\cal U}_u$ of 
%finite dimensional vector spaces over ${\mab C}$, 
%then we obtain a canonical contravariantly functorial isomorphism 
%\begin{align*} 
%R\Gam(X,{\cal E}\otimes_{{\cal O}_{\cal X}}\Om^{\bul}_{X/s}(\log D))
%\os{\sim}{\lo} R\Gam({\cal U}_u,V). 
%\tag{1.3.8}\label{ali:matox}
%\end{align*} 
%This is a generalization of the isomorphism (\ref{ali:dxa}). 
To prove the theorem (\ref{theo:rpi}), we use the theory of log geometry 
essentially and the Hirsch extension of a dga,  
more generally the Hirsch extension of a log integrable connection. 
As a corollary of (\ref{theo:rpi}), we prove the following corollary:  

\begin{coro}[{\rm cf.~\cite{ur2}}]\label{coro:rqe}
%Let the notations be as in {\rm (\ref{theo:rpi})}.
Let $\os{\circ}{\cal Y}$ be a smooth analytic space with an NCD $Y$. 
Let ${\cal Y}$ be the log analytic space obtained by a pair $(\os{\circ}{\cal Y},Y)$. 
Let $\os{\circ}{f}\col \os{\circ}{\cal X}\lo \os{\circ}{\cal Y}$ be 
a proper morphism of smooth schemes over $K$. 
Assume that $X:=f^*(Y)$ is a normal crossing divisor on $\os{\circ}{\cal X}$. 
Let ${\cal X}$ be the log analytic space obtained by a pair $(\os{\circ}{\cal X},X)$. 
Let $f\col {\cal X}\lo {\cal Y}$ be a proper strictly semistable morphism. 
Let ${\cal D}$ be a horizontal NCD on ${\cal X}/{\cal Y}$. 
%Assume that the structural morphism 
%$\os{\circ}{f}\col \os{\circ}{\cal X}\lo \os{\circ}{B}$ is proper. 
%Assume also that $f\col {\cal X}\lo B$ is a generalized strictly semistable morphism. 
Let $({\cal E},\nabla)$ be a locally nilpotent integrable connection 
on $({\cal X},{\cal D})$ in the sense of {\rm \cite{kn}:} $\nabla \col {\cal E}\lo  
{\cal E}\otimes_{{\cal O}_{\cal X}}\Om^1_{{\cal X}/{\mab C}}(\log {\cal D})$. 
Let ${\cal E}\otimes_{{\cal O}_{\cal X}}\Om^{\bul}_{{\cal X}/B}(\log {\cal D})$ 
be the log de Rham complex obtained by 
$({\cal E},\nabla)$. 
Then 
$R^qf_*({\cal E}\otimes_{{\cal O}_{\cal X}}\Om^{\bul}_{{\cal X}/B}(\log {\cal D}))$
$(q\in {\mab N})$ is a coherent locally free ${\cal O}_{\cal Y}$-module and 
commutes with base change.
Here, for a morphism $\os{\circ}{\cal Z}\lo \os{\circ}{\cal Y}$, we have endowed 
$\os{\circ}{\cal Z}$ with the inverse image of the log structure of ${\cal Y}$. 
\end{coro}

\par 
Consider the case ${\cal Y}=B$. 
Let $B^*$ be the maximal open log analytic space of $B$ whose log structure is trivial. 
By (\ref{coro:rqe}) we see that 
$R^qf_*({\cal E}\otimes_{{\cal O}_{\cal X}}\Om^{\bul}_{{\cal X}/B}(\log {\cal D}))$ is
the canonical extension of 
$R^qf_*({\cal E}\otimes_{{\cal O}_{\cal X}}\Om^{\bul}_{{\cal X}/B}(\log {\cal D}))\vert_{B^*}$
in the sense of \cite{de}. 
In the case where ${\cal E}$ is trivial, this has been already stated 
in \cite{ur2} without proof. 
In order to prove (\ref{coro:rqe}),  
we use only log de Rham complexes; 
we do not use the Kummer \'{e}tale topos 
nor ${\cal X}^{\log}$ unlike \cite{kf} and \cite{ikn}. Hence 
we can give a purely algebraic proof of (\ref{coro:rqe}). 
%Moreover, we do not use the system of parameters of 
%$\Del^r$ unlike \cite{sti} and \cite{fii}.  

Using a variant of the section (\ref{ali:dxda}) 
(we call this variant {\it the $\infty$-adic Yamada section}) 
and (\ref{theo:rpi}), we obtain the following most important result 
in this article (we call this theorem the invariance theorem): 

\begin{theo}[{\bf Invariance theorem}]\label{theo:rnpi}
Let the notations be as in {\rm (\ref{theo:rpi})}. 
%For an open log analytic space $U$ of $\Del^{r_1}$ including the origin of 
%$\Del^{r_1}$, set $B_U:=U\times \os{\circ}{\Del}{}^{r_1}$. 
%Let $\iota_{U}\col X\os{\sus}{\lo} {\cal X}\vert_{f^{-1}(B_U)}
%$ be the natural closed immersion 
%and let $f_U\col {\cal X}\vert_{f^{-1}(B_U)}\lo B_U$ be the structural morphism. 
%Set ${\cal D}_U:={\cal D}_{{\cal X}\vert_{f^{-1}(B_U)}}$. 
Assume that $f\col \os{\circ}{\cal X}\lo \os{\circ}{B}$ is proper. 
Let $p\col B\lo \os{\circ}{\Del}{}^{r_2}(\simeq \os{\circ}{S})$ be the projection. 
For a positive integer $\eps$ less than $1$, let $B(\eps)$ 
be an open log analytic space whose underlying space is 
$\{z\in \os{\circ}{\Del}~\vert~\vert z\vert <\eps\}^{r_1}\times 
\os{\circ}{\Del}{}^{r_2}$ and whose log structure is the inverse image of $B$ 
$($a tubular neighborhood of $S)$ in $B$ . 
Let $p(\eps) \col B(\eps)\lo \os{\circ}{\Del}{}^{r_2}$ be the projection. 
Then there exists a canonical contravariantly functorial isomorphism 
\begin{align*} 
\sig \col {\cal O}_{B(\eps)}\otimes_{p(\eps)^{-1}({\cal O}_S)}
p^{-1}R^qf_{S*}({\cal E}\otimes_{{\cal O}_{\cal X}}\Om^{\bul}_{X/S}(\log D))
\os{\sim}{\lo}  
R^qf(\eps)_*({\cal E}\otimes_{{\cal O}_{\cal X}}
\Om^{\bul}_{{\cal X}(\eps)/B(\eps)}(\log {\cal D}(\eps))). 
\tag{1.5.1}\label{ali:oebxas}
\end{align*} 
Here ${\cal X}(\eps):={\cal X}\times_BB(\eps)$,  
${\cal D}(\eps):={\cal D}\times_BB(\eps)$ and $f(\eps)\col {\cal X}(\eps)\lo B(\eps)$ 
is the structural morphism.  
%which induces an isomorphism 
%\begin{align*} 
%\Gam(W,{\cal O}_{B}\otimes_{p^{-1}({\cal O}_S)}
%p^{-1}R^qf_{S*}({\cal E}\otimes_{{\cal O}_{\cal X}}\Om^{\bul}_{X/S}(\log D)))
%\os{\sim}{\lo}  
%\Gam(W,R^qf_*({\cal E}\otimes_{{\cal O}_{\cal X}}\Om^{\bul}_{{\cal X}/B}(\log {\cal D})))
%\tag{1.5.2}\label{ali:mxqbox}
%\end{align*} 
%for a point $t\in S$ and a small neighborhood $W$ of $t$ in $B$. 
\end{theo}

The isomorphism (\ref{ali:oebxas}) tells us that 
$R^qf_*({\cal E}\otimes_{{\cal O}_{\cal X}}
\Om^{\bul}_{{\cal X}(\eps)/B(\eps)}(\log {\cal D}(\eps)))$ depends only on 
$R^qf_{S*}({\cal E}\otimes_{{\cal O}_{\cal X}}\Om^{\bul}_{X/S}(\log D))$ 
and $\os{\circ}{B}(\eps)$.  
%Especially it does not depend on ${\cal X}(\eps)/B(\eps)$ nor 
%the log structure of $B(\eps)$. 
Set $X_{\os{\circ}{B}(\eps)}:=X\times_{\os{\circ}{S},\os{\circ}{p}(\eps)}\os{\circ}{B}$, 
$D_{\os{\circ}{B}(\eps)}:=D\times_{\os{\circ}{S},\os{\circ}{p}(\eps)}\os{\circ}{B}$ 
and let $f_{\os{\circ}{B}(\eps)}\col X_{\os{\circ}{B}(\eps)}\lo \os{\circ}{B}(\eps)$ 
be the structural morphism 
(``the constant family of $X/S$ over $\os{\circ}{B}(\eps)$''). 
In other words, the isomorphism (\ref{ali:oebxas}) is the following isomorphism
\begin{align*} 
R^qf_{\os{\circ}{B}(\eps)*}({\cal E}\otimes_{{\cal O}_{\cal X}}\Om^{\bul}_{X/S}(\log D)
\otimes_{{\cal O}_S}{\cal O}_{B(\eps)})\os{\sim}{\lo} 
R^qf_*({\cal E}\otimes_{{\cal O}_{\cal X}}
\Om^{\bul}_{{\cal X}(\eps)/B(\eps)}(\log {\cal D}(\eps))). 
\tag{1.5.2}\label{ali:ocxas}
\end{align*} 
However the isomorphism (\ref{ali:ocxas}) does not 
induce a filtered isomorphism with respect to log Hodge filtrations in general. 
As usual, the variation of the filtrations depends on the family 
$({\cal X}(\eps),{\cal D}(\eps))$ 
over $B(\eps)$. 
Note also that the log structures of the base log analytic spaces of both hand sides in   
(\ref{ali:ocxas}) are different. 
%(except the constant family) in the smooth case. 
Though the following is included in the theorem above, 
we would like to state it because of the importance: 

\begin{coro}[{\bf Invariance theorem of the pull-backs of morphisms}]\label{coro:urm}
Let $B'$ and $S'$ be analogous log analytic spaces to $B$ and $S$, respectively. 
Let $({\cal X}',{\cal D}')/B'$ 
be an analogous log analytic space to $({\cal X},{\cal D})/B$ 
and let $(X',D')$ be an analogous log analytic space to 
$(X,D)$ for $({\cal X}',{\cal D}')/B'$. 
%For $i=1,2$, let 
%\begin{equation*} 
%\begin{CD}
%({\cal X},{\cal D})@>{g_i}>> ({\cal X}',{\cal D}')\\
%@V{f}VV @VV{f'}V \\
%B @>{u}>> B'
%\end{CD}
%\end{equation*}
%be two commutative diagrams of log analytic spaces. 
Assume that we are given the following commutative diagram 
\begin{equation*} 
\begin{CD}
(X,D)@>{g}>> (X',D')\\
@VVV @VVV \\
S@>{u\vert_S}>> S'\\
@V{\bigcap}VV @VV{\bigcap}V \\
B @>{u}>> B'. 
\end{CD}
\end{equation*}
%Let $t$ be a point of $S$. 
Let $({\cal E}',\nabla')$ be an analogous integrable connection to $({\cal E},\nabla)$ 
for ${\cal X}'/\os{\circ}{S}{}'$. 
Set $E^{\square}:={\cal E}^{\square}$, where $\square=$nothing or $'$. 
Let $\Phi \col g^*(E')\lo E$ 
be a morphism of ${\cal O}_X$ modules fitting into 
the following commutative diagram$:$
\begin{equation*} 
\begin{CD}
g^*(E')@>{\Phi}>> E\\
@V{g^*(\nabla')}VV @VV{\nabla}V \\
g^*(E'\otimes_{{\cal O}_{X'}}\Om^1_{X'/\os{\circ}{S}{}'}(\log D'))
@>{\Phi\otimes g^*}>> E\otimes_{{\cal O}_X}\Om^1_{X/\os{\circ}{S}{}}(\log D). 
\end{CD}
\end{equation*}
%If the two diagrams above are the same on $X$, 
Then there exist $\eps>0$ and $\eps'>0$ 
such that $u(B(\eps))  \subset B'(\eps')$ 
such that there exists the following canonical morphism 
\begin{align*}
g(\eps,\eps')^*\col &u(\eps,\eps')^*
(R^qf'(\eps')_*({\cal E}'\otimes_{{\cal O}_{{\cal X}'}}\Om^{\bul}_{{\cal X}'(\eps')/B'(\eps')}
(\log {\cal D}'(\eps'))))\\
&\lo 
R^qf(\eps)_*({\cal E}\otimes_{{\cal O}_{\cal X}}\Om^{\bul}_{{\cal X}(\eps)
/B(\eps)}(\log {\cal D}(\eps))). 
\end{align*} 
%for $i=1,2$ are equal. 
Here $u(\eps,\eps')\col B(\eps)\lo B'(\eps')$ is the induced morphism
by $u$. The morphism $g(\eps,\eps')$ satisfies the usual transitive relation. 
If the morphism $g$ is lifted to a morphism 
$\wt{g}\col {\cal X}(\eps)\lo {\cal X}'(\eps')$ and if 
the morphism $\Phi$ is lifted to a morphism 
$\wt{\Phi} \col g^*({\cal E}')\lo {\cal E}$ 
be a morphism of ${\cal O}_{\cal X}$ modules fitting into 
the following commutative diagram$:$
\begin{equation*} 
\begin{CD}
\wt{g}^*({\cal E}'(\eps'))@>{\wt{\Phi}}>> {\cal E}(\eps)\\
@V{\wt{g}{}^*(\nabla')}VV @VV{\nabla}V \\
\wt{g}^*({\cal E}'\otimes_{{\cal O}_{{\cal X}'}}\Om^1_{{\cal X}'(\eps')/\os{\circ}{S}{}'}(\log {\cal D}'(\eps')))
@>{\wt{\Phi}\otimes \wt{g}{}^*}>> {\cal E}\otimes_{{\cal O}_{\cal X}}\Om^1_{{\cal X}(\eps)/\os{\circ}{S}{}}(\log {\cal D}(\eps)), 
\end{CD}
\end{equation*}
where ${\cal E}^{\square}(\eps^{\square})=
{\cal E}^{\square}\vert_{{\cal X}^{\square}(\eps^{\square})}$, 
then 
\begin{align*}
g(\eps,\eps')^*=\wt{g}^*\col &u(\eps,\eps')^*
(R^qf'(\eps')_*({\cal E}'\otimes_{{\cal O}_{{\cal X}'}}\Om^{\bul}_{{\cal X}'(\eps')/B'(\eps')}
(\log {\cal D}'(\eps'))))\\
&\lo 
R^qf(\eps)_*({\cal E}\otimes_{{\cal O}_{\cal X}}\Om^{\bul}_{{\cal X}(\eps)
/B(\eps)}(\log {\cal D}(\eps))). 
\end{align*}
\end{coro}

The results  (\ref{theo:rnpi}) and (\ref{coro:urm}) suggest that 
there may exist a theory of crystals for an integrable connection on 
a proper log smooth morphism $Y\lo S$ of fine or 
fs(=fine and saturated) log analytic spaces over $B(\eps)$ 
as in the algebraic case (\cite{grc}, \cite{klog1}). 
The main theorem of \cite{no} (cf.~\cite{kjn})
which tells us that $f^{\log}\col Y^{\log}\lo S^{\log}$ is a fiber bundle
%for a local system $W$ on $X^{\log}$ in the setting in (\ref{theo:kikn}) 
(cf.~\cite{ur1}, \cite{ur2}) 
also suggests this. However the main theorem does not imply (\ref{theo:rnpi}): 
it only implies that there exists a canonical isomorphism 
\begin{align*}  
&{\cal O}_{B(\eps)^{\log}}\otimes_{p(\eps)^{-1}({\cal O}_S)}
p^{-1}R^qf_{S*}({\cal E}\otimes_{{\cal O}_{\cal X}}\Om^{\bul}_{X/S}(\log D))
\os{\sim}{\lo}  \tag{1.6.1}\label{ali:oeas}\\
&{\cal O}_{B(\eps)^{\log}}\otimes_{{\cal O}_{B(\eps)}}
R^qf(\eps)_*({\cal E}\otimes_{{\cal O}_{\cal X}}
\Om^{\bul}_{{\cal X}(\eps)/B(\eps)}(\log {\cal D}(\eps))). 
\end{align*}  
Moreover our theorem (\ref{theo:rnpi}) 
has an advantage in the concrete calculation of the isomorphism 
(\ref{ali:oebxas}) by using \v{C}ech coverings of ${\cal X}(\eps)$ 
as in the log rigid analytic case in \cite{ey}. 
%we can say that $R^qf_*({\cal E}\otimes_{{\cal O}_{\cal X}}
%\Om^{\bul}_{{\cal X}/B}(\log {\cal D}))$ has two parts of the nature of crystals 
%(as in the algebraic case) at least on a small tubular neighborhood of $S$ in $B$: 
%it grows rigidly. 
%As an immediate corollary to (\ref{coro:urm}) 
%we obtain the rigidity of the relative log de Rham complex 
%$R^qf(\eps)_*({\cal E}\otimes_{{\cal O}_{\cal X}}\Om^{\bul}_{{\cal X}(\eps)
%/B(\eps)}(\log {\cal D}(\eps))$ for a small $\eps>0$: this depends only on 
%$R^qf_{S*}({\cal E}\otimes_{{\cal O}_{\cal X}}\Om^{\bul}_{X
%/S}(\log D))$.  
%We can say that $R^qf_{S*}({\cal E}\otimes_{{\cal O}_{\cal X}}\Om^{\bul}_{X
%/S}(\log D))$ has every information of 
%$R^qf(\eps)_*({\cal E}\otimes_{{\cal O}_{\cal X}}\Om^{\bul}_{{\cal X}(\eps)
%/B(\eps)}(\log {\cal D}(\eps)))$. 

\par 
By (\ref{coro:rqe}) and (\ref{theo:rnpi}) we obtain the following result: 

\begin{coro}\label{coro:ps}
%Assume that $r_2=0$ $($hence $S=s)$. 
For a sufficiently close point $u\in B$ from a point of $S$, 
there exists a canonical contravariantly functorial isomorphism 
\begin{align*} 
{\cal O}_{\ol{S},u}\otimes_{{\cal O}_{S,p(u)}}
R^qf_{S*}({\cal E}\otimes_{{\cal O}_{\cal X}}\Om^{\bul}_{{\cal X}/B}(\log {\cal D})))_{p(u)}
\os{\sim}{\lo}  
R^qf_*({\cal E}\otimes_{{\cal O}_{\cal X}}\Om^{\bul}_{{\cal X}/B}(\log {\cal D})))_u. 
\tag{1.7.1}\label{ali:mxoox}
\end{align*} 
Consequently there exists a canonical contravariantly functorial isomorphism 
\begin{align*} 
H^q(X_s,{\cal E}\otimes_{{\cal O}_{\cal X}}\Om^{\bul}_{X_s/s}(\log D_s)))
\os{\sim}{\lo}  
H^q({\cal X}_u,{\cal E}\otimes_{{\cal O}_{\cal X}}\Om^{\bul}_{{\cal X}_u/u}
(\log {\cal D}_u))). 
\tag{1.7.2}\label{ali:muox}
\end{align*} 
%Consequently 
%$H^q({\cal X}_u,{\cal E}\otimes_{{\cal O}_{\cal X}}\Om^{\bul}_{{\cal X}_u/u}
%(\log {\cal D}_u)))$ is determined  by 
%$H^q(X_s,{\cal E}\otimes_{{\cal O}_{\cal X}}\Om^{\bul}_{X_s/s}(\log D_s)))$. 
\end{coro}

\par 
%Consider the case $r_2=0$. 
%Assume that $\os{\circ}{\cal X}$ is proper over $\os{\circ}{B}$. 
%Set ${\cal U}:={\cal X}\setminus {\cal D}$. 
%Let $u$ be any exact point of $B$ whose log structure is trivial. 
%As a corollary of the isomorphism (\ref{ali:mxqbox}), we obtain 
%a canonical contravariantly functorial isomorphism 
%\begin{align*} 
%H^q(X,{\cal E}\otimes_{{\cal O}_{\cal X}}\Om^{\bul}_{X/s}(\log D))
%\os{\sim}{\lo} 
%H^q({\cal U}_u,{\cal E}\otimes_{{\cal O}_{\cal X}}\Om^{\bul}_{{\cal U}_u/u}). 
%\tag{1.6.3}\label{ali:matbox}
%\end{align*} 
%if $u\in \Del^*$ is close from $O$.  
If $({\cal E}\vert_{{\cal U}_u},\nabla)
=({\cal O}_{{\cal U}_u}\otimes_{\mab C}V,d\otimes {\rm id}_V)$ 
for a local system $V$ on ${\cal U}_u$ of 
finite dimensional vector spaces over ${\mab C}$, 
then we obtain a canonical contravariantly functorial isomorphism 
\begin{align*} 
H^q(X_s,{\cal E}\otimes_{{\cal O}_{\cal X}}\Om^{\bul}_{X_s/s}(\log D_s))
\os{\sim}{\lo} H^q({\cal U}_u,V). 
\tag{1.7.3}\label{ali:matox}
\end{align*} 
%if $u$ is close from $O$. 
%Hence we obtain a canonical contravariantly functorial isomorphism 
%\begin{align*} 
%H^q(X,{\cal E}\otimes_{{\cal O}_{\cal X}}\Om^{\bul}_{X/s}(\log D))
%\os{\sim}{\lo} 
%H^q({\cal U}_u,V). 
%\tag{1.3.7}\label{ali:mabox}
%\end{align*} 
This is a generalization of the isomorphism (\ref{ali:dxa}). 
We call this generalization the {\it Steenbrink isomorphism}. 
This also tells us that $H^q({\cal U}_u,V)$ is determined by 
$(X_s,D_s)$ and ${\cal E}\vert_{X_s}$. 
%(if $u$ is close from $O$). 
%The reader should note 
%that the isomorphism (\ref{ali:matox}) is independent of the choice of 
%local parameters of $\Del^{r_1}$ (up to $({\mab C}^*)^{r_1}$). 

%As an application of (\ref{coro:urm}), 
%we obtain the following by the smoothing of log $K3$-surfaces 
%(\cite{frg}, \cite{kawn}): 

%\begin{coro}\label{coro:acc}
%Let $X$ be a log analytic $K3$ surface over the log point $s$ of ${\mab C}$. 
%Assume that $\os{\circ}{X}$ is K\"{a}hler. 
%Then the morphism 
%\begin{align*} 
%{\rm Aut}(X/{\mab C})\lo {\rm Aut}_{\mab C}(H^2_{\rm dR}(X/s))
%\tag{1.9.1}\label{ali:xca}
%\end{align*} 
%is injective. 
%\end{coro} 
%The corollary (\ref{coro:acc}) is an $\infty$-adic analytic analogue of 
%\cite[(6.12)]{nk3}, in which $H^2_{\rm dR}(X/s)$ has been replaced by 
%the log crystalline cohomology $H^2_{\rm crys}(Y/{\cal W})$, 
%where $Y$ is a log $K3$ surface over $s_{\kap}$. 
%(See \cite[\S7]{ma} for examples of 
%nontrivial elements of ${\rm Aut}(X/{\mab C})$ for 
%the algebraic degenerate log $K3$ surfaces.) 

\par 
We can also obtain the algebraic analogues of 
(\ref{theo:rpi}) and (\ref{coro:rqe}) by algebraic proofs 
(without using the Lefschetz principle, GAGA nor (\ref{theo:rpi})). 
This is also one of advantages of our algebraic method. 
In a future article we would like to discuss this.
\par 
As to the locally freeness in (\ref{coro:rqe}),  
%and (\ref{coro:qfe}), 
our method is also applicable for a more general family with worse reduction 
by virtue of the contravariant functoriality of the section (\ref{ali:mocx}). 
This general case is not contained in the case (\ref{theo:kikn}). 
We would also like to discuss this. 
\par 
Our results (\ref{theo:rnpi}) and (\ref{coro:urm}) 
should be applied for (the automorphism groups of) 
concrete degeneratation of analytic spaces, e.~g., 
(log) $K3$ surfaces, more generally (log) Calabi-Yau analytic spaces, 
degenerations of complex tori, degenerations of hyperk\"{a}hler manifolds 
and so on (cf.~\cite{frg}, \cite{kawn}, \cite{bea}, ...). 
\par 
The contents of this article are as follows. 
\par
In \S\ref{sec:da} we recall results in \cite{nf} 
which are necessary in this article. 
\par 
In \S\ref{sec:hie} we give additional results to those in \cite{nf}.   
\par 
In \S\ref{sec:itst} we give the definition of SNCL analytic spaces. 
\par 
In \S\ref{sec:lp} we consider the completed Hirsch extension of 
an integrable connection on a semistable family over a polydisc. 
\par 
In \S\ref{sec:m} we give the invariance theorems (\ref{theo:rnpi}) and (\ref{coro:urm}). 
%\par 
%In \S\ref{sec:ac} we give analogues of results in the algebraic case. 
%\par 
%In \S\ref{sec:k3} we give an application for an automorphism of 
%log analytic K\"{a}hler $K3$ surfaces and a morphism of degenerations of complex tori. 

%The key point of the proof of (\ref{theo:rpi}) is to construct 
%Hence we see that the morphism (\ref{ali:moox}) is surjective for any $q\in {\mab N}$. 
%By using Grauert's upper semi-continuity theorem, we see that 
%$R^qf_*({\cal E}\otimes_{{\cal O}_{\cal X}}\Om^{\bul}_{{\cal X}/B}(\log {\cal D}))$ 
%is locally free and commutes with base change. 
%Since the isomorphism (\ref{ali:mox}) 
%is independent of the choice of the parameters of $\Del^r$, 
%we have the canonical isomorphism between 
%$R^qf_*({\cal E}\otimes_{{\cal O}_{\cal X}}\Om^{\bul}_{X/s}(\log D))$ 
%and 
%$R^qf_*({\cal E}\otimes_{{\cal O}_{\cal X}}\Om^{\bul}_{{\cal X}_u/u})$. 
%In particular we can get rid of the dependence of the isomorphism 
%(\ref{ali:xtc}) on the parameter. 
%In fact, the isomorphism (\ref{ali:mox}) is contravariantly functorial. 
%Our proof of (\ref{theo:rpi}) is almost algebraic unlike \cite{kf} and \cite{ikn}. 

\par
\bigskip
\parno
{\bf Acknowledgment.} 
%I am grateful to T.~Yoshida for 
%having asked me whether the functorial compatibilities 
%with respect to $p$-adic weight spectral sequence 
%has been proved at the conference about Shimura varieties 
%at Kyoto in November 2008. 
%Though the question was only a word, 
%it was a strong happy power in my heart during this work. 
I would like to express my sincere gratitude to T.~Fujisawa 
for pointing out a non-minor mistake 
in a previous version of this article.

\section{Recall of results in \cite{nf}}\label{sec:da}  
%In \cite{fc} and \cite{fr} Fujisawa has considered the ``maximal'' logarithmic case 
%for a generalized semistable morphism: the case $r_2=0$ 
%in the notation (\ref{theo:rpi}) (without the horizontal log structure). 
%However, even in this case, the ``non-maximal'' case (the case $r_2>0)$ 
%arises at points which is different from the ``origin''. 
%Hence we consider a general log analytic space as 
%a base log analytic space from the beginning of this section. 
In this section we recall results in \cite{nf} which are necessary in this article. 
\par 
Let $r$ be a positive integer. 
Let $S$ be an analytic family over ${\mab C}$ of log points of virtual dimension $r$,  
that is, locally on $S$, the log structure $M_S$ of $S$ is isomorphic to  
${\mab N}^r\oplus {\cal O}_{S}^*\owns ((n_1,\ldots,n_r),u)\lom 0^{\sum_{i=1}^rn_i}u\in  
{\cal O}_{S}$ (cf.~\cite[\S2]{nb}), where $0^n=0$ when $n\not =0$ and $0^0:=1$. 
Let $g\col Y\lo S$ be a log smooth morphism of log analytic spaces over ${\mab C}$. 
\par
Locally on $S$,  there exists a family  
$\{t_i\}_{i=1}^r$ of local sections of $M_S$ 
giving a local basis of $\ol{M}_S:=M_S/{\cal O}_S^*$. 
Let $M_i$ be the submonoid sheaf of $M_S$ generated by $t_i$. 
For all $1\leq \forall i\leq r$, the submonoid sheaf 
$\bigoplus_{j=1}^iM_j$ of $M$ and $\os{\circ}{S}$ 
defines a family of log points of virtual dimension $i$. 
Let $S_i:=S(M_1,\ldots,M_i)=(\os{\circ}{S},(\bigoplus_{j=1}^iM_j\lo {\cal O}_S)^a)$ 
be the resulting local log analytic space. Set $S_0:=\os{\circ}{S}$. 
Then we have the following sequence of families of log points of virtual dimensions: 
\begin{align*} 
S=S_r\lo S_{r-1}\lo S_{r-2}\lo \cdots \lo S_1\lo S_0=\os{\circ}{S}.
\tag{2.0.1}\label{ali:ffoy}
\end{align*} 
The one-form $d\log t_i\in \Om^1_{S/\os{\circ}{S}}$ is 
independent of the choice of the generator 
$t_i$ of $M_i$. Denote also by $d\log t_i\in \Om^1_{Y/\os{\circ}{S}}$ 
the image of $d\log t_i\in \Om^1_{S/\os{\circ}{S}}$ in 
$\Om^1_{Y/\os{\circ}{S}}$.  
Let ${\cal F}$ be a 
(not necessarily coherent) locally free ${\cal O}_Y$-module and let 
\begin{align*} 
\nabla \col {\cal F}\lo {\cal F}\otimes_{{\cal O}_Y}\Om^1_{Y/\os{\circ}{S}}
\tag{2.0.2}\label{ali:ffoya}
\end{align*} 
be an integrable connection. 
(In later sections we assume that ${\cal F}$ is coherent.) 
Then we have the complex 
${\cal F}\otimes_{{\cal O}_Y}\Om^{\bul}_{Y/\os{\circ}{S}}$ 
of $g^{-1}({\cal O}_S)$-modules.

\begin{lemm}[{\rm \cite[(3.1)]{nf}}]\label{lemm:qseux}
$(1)$ The sheaf $\Om^i_{Y/\os{\circ}{S}}$ $(i\in {\mab N})$ is a locally free 
${\cal O}_Y$-module. 
\par 
$(2)$ Locally on $Y$, the following sequence 
\begin{align*} 
0  \lo {\cal F}
\otimes_{{\cal O}_Y}
{\Om}^{\bul}_{Y/S_i}[-1] 
\os{{\rm id}_{\cal F}\otimes (d\log t_i\wedge)}{\lo} 
{\cal F}\otimes_{{\cal O}_Y}
{\Om}^{\bul}_{Y/S_{i-1}}  
\lo {\cal F}\otimes_{{\cal O}_Y}{\Om}^{\bul}_{Y/S_i} \lo 0 
\quad (1\leq i\leq r)
\tag{2.1.1}\label{ali:agaxd}
\end{align*} 
is exact. 
$($Note that the morphism 
\begin{align*} 
{\rm id}_{\cal F}\otimes (d\log t_i\wedge) \col {\cal F}
\otimes_{{\cal O}_Y}
{\Om}^{\bul}_{Y/S_i}[-1] \lo 
{\cal F}\otimes_{{\cal O}_Y}
{\Om}^{\bul}_{Y/S_{i-1}}  
\end{align*} 
is indeed a morphism of complexes of $g^{-1}({\cal O}_S)$-modules.$)$ 
\end{lemm}

\par 
Let $\{t_1,\ldots,t_r\}$ be a set of local sections of $M_S$ whose images in 
$\ol{M}_S$ is a unique system of local generators of $\ol{M}_S$. 
Set $L_S(t_1,\ldots,t_r):=\bigoplus_{i=1}^r{\cal O}_St_i$ 
(a free ${\cal O}_S$-module with basis $t_1,\ldots,t_r$). 
This ${\cal O}_S$-module patches together to a locally free ${\cal O}_S$-module 
$L_S$ by \cite[(3.3)]{nf}. 
In the following we denote $t_i$ in $L_S$ by $u_i$. 
Because we prefer not to use non-intrinsic variables $u_1,\ldots,u_r$, 
we consider the following sheaf $U_S$ 
of sets of $r$-pieces on $\os{\circ}{S}$ defined by the following presheaf: 
\begin{align*} 
U_S\col \{{\rm open}~{\rm log}~{\rm analytic}~{\rm spaces}~{\rm of}~S\}\owns 
V\lom \{{\rm generators}~{\rm of}~\Gam(V,\ol{M}_S)\}\in 
({\rm Sets}).  
\tag{2.1.2}\label{eqn:bgff}
\end{align*}
Let ${\rm Sym}_{{\cal O}_S}(L_S)$ be the symmetric algebra 
of $L_S$ over ${\cal O}_S$ by $\Gam_{{\cal O}_S}(L_S)$. 
%The ${\cal O}_S$-module $L_S$ depends only on $U_S$. 
We consider the following Hirsch extension 
\begin{equation*}  
{\cal F}\otimes_{{\cal O}_Y}
\Om^{\bul}_{Y/\os{\circ}{S}}[U_S]
:= \Gam_{{\cal O}_S}(L_S)\otimes_{{\cal O}_S}
{\cal F}\otimes_{{\cal O}_Y}
\Om^{\bul}_{Y/\os{\circ}{S}} 
\tag{2.1.3}\label{eqn:blhff}
\end{equation*} 
of ${\cal F}\otimes_{{\cal O}_Y}
\Om^{\bul}_{Y/\os{\circ}{S}}$ 
by the morphism 
\begin{align*} 
d\log \col g^{-1}(L_S)\owns u_i\lom d\log t_i\in 
{\rm Ker}(\Om^1_{Y/\os{\circ}{S}}\lo \Om^2_{Y/\os{\circ}{S}})
\tag{2.1.4}\label{eqn:blff}
\end{align*}  
(\cite[\S3]{nhi}). 
(Here we have omitted to write $g^{-1}$ for  
$g^{-1}(\Gam_{{\cal O}_S}(L_S))\otimes_{g^{-1}({\cal O}_S)}$ 
in (\ref{eqn:blhff}). It may be better to denote ${\cal F}\otimes_{{\cal O}_Y}
\Om^{\bul}_{Y/\os{\circ}{S}}[U_S]$ by 
${\cal F}\otimes_{{\cal O}_Y}
\Om^{\bul}_{Y/\os{\circ}{S}}[L_S]$.)
It is easy to check that the morphism 
$g^{-1}(L_S)\lo {\rm Ker}(\Om^1_{Y/\os{\circ}{S}}\lo \Om^2_{Y/\os{\circ}{S}})$ 
is well-defined (cf.~\cite[\S2]{nb}). 
The boundary morphism 
\begin{equation*} 
\nabla \col  {\cal F}\otimes_{{\cal O}_Y}
\Om^i_{Y/\os{\circ}{S}}[U_S]
\lo  {\cal F}\otimes_{{\cal O}_Y}
\Om^{i+1}_{Y/\os{\circ}{S}}[U_S]
\quad (i\in {\mab Z}_{\geq 0})
\end{equation*}   
is, by definition,  the following ${\cal O}_S$-linear morphism: 
\begin{align*} 
&\nabla(m_1^{[e_1]}\cdots m_r^{[e_r]}\otimes \om)=
\sum_{j=1}^rm_1^{[e_1]}\cdots m_j^{[e_j-1]} \cdots m_r^{[e_r]}
d\log m_j\wedge \om +
m_1^{[e_1]}\cdots  m_r^{[e_r]}\otimes \nabla(\om)
\tag{2.1.5}\label{eqn:bdff}\\
&\quad (~\om \in {\cal F}
\otimes_{{\cal O}_Y}
\Om^j_{Y/\os{\circ}{S}},m_1, \ldots m_r\in L_S, 
e_1,\ldots,e_r\in {\mab Z}_{\geq 1}, m_i^{[e_i]}=((e_i)!)^{-1}m^{e_i}_i~). 
\end{align*}
It is easy to check an equality $\nabla^2=0$.
%Because $\os{\circ}{S}$ is of characteristic $0$, 
We have the following equality 
\begin{align*} 
%{\cal O}_S\langle L_S\rangle=
\Gam_{{\cal O}_S}(L_S)
%=:{\cal O}_S[L_S]
={\cal O}_S[\ol{M}_S]. 
\tag{2.1.6}\label{ali:uaubri} 
\end{align*}   
Let $\eps \col S^{\log}\lo S$ be the real blow up defined in \cite{kn}. 
Then $\Gam_{{\cal O}_S}(L_S)$ is nothing but 
$\eps_*({\cal O}_{S^{\log}})$. 
The sheaf ${\cal O}_S[\ol{M}_S]$ is locally isomorphic to 
a sheaf ${\cal O}_{S}[\ul{u}]:=
{\cal O}_{S}[u_1,\ldots,u_r]$ of polynomial algebras over ${\cal O}_S$ of 
independent $r$-variables. 
The projection $\Gam_{{\cal O}_S}(L_S)={\rm Sym}_{{\cal O}_S}(L_S)
\lo {\rm Sym}^0_{{\cal O}_S}(L_S)={\cal O}_S$
induces the following natural morphism of complexes: 
\begin{equation*} 
{\cal F}\otimes_{{\cal O}_Y}
\Om^{\bul}_{Y/\os{\circ}{S}}[U_S]
\lo {\cal F}\otimes_{{\cal O}_Y}
\Om^{\bul}_{Y/S}.
\tag{2.1.7}\label{ali:uafacui} 
\end{equation*}  
(\cite[\S3]{nhi}). 
Let $\Gam^{\geq n}_{{\cal O}_S}(L_S)$ $(n\in {\mab N})$ 
be the degree $\geq n$-part of $\Gam_{{\cal O}_S}(L_S)$. 
%and let ${\rm Sym}^{n}_{{\cal O}_S}(L_S)$
%be the degree $n$-part of ${\rm Sym}_{{\cal O}_S}(L_S)$ 
We set 
\begin{align*} 
\Gam_{{\cal O}_S}(L_S)^{\wedge}
%:={\cal O}_S[[u_1,\ldots,u_r]]
:=
%\prod_{e_1,\ldots,e_r\in {\mab N}}{\cal O}_Sm_1^{[e_1]}\cdots u_r^{[e_r]}. 
\vpl_n(\Gam_{{\cal O}_S}(L_S)/\Gam^{\geq n}_{{\cal O}_S}(L_S)). 
\tag{2.1.8}\label{ali:urui} 
\end{align*} 
Then we have the following ``inclusion morphism''
\begin{align*} 
\Gam_{{\cal O}_S}(L_S)  \os{\sus}{\lo} \Gam_{{\cal O}_S}(L_S)^{\wedge}. 
\tag{2.1.9}\label{ali:uauri} 
\end{align*} 
We also consider the following PD-Hirsch extension 
by completed powers:  
\begin{equation*}  
{\cal F}\otimes_{{\cal O}_Y}
\Om^{\bul}_{Y/\os{\circ}{S}}[[U_S]]:=
\Gam_{{\cal O}_S}(L_S)^{\wedge}\otimes_{{\cal O}_S}
{\cal F}\otimes_{{\cal O}_Y}
\Om^{\bul}_{Y/\os{\circ}{S}} 
\end{equation*} 
of ${\cal F}\otimes_{{\cal O}_Y}
\Om^{\bul}_{Y/\os{\circ}{S}}$; 
the boundary morphism 
\begin{equation*} 
\nabla \col  {\cal F}
\otimes_{{\cal O}_Y}
\Om^j_{Y/\os{\circ}{S}}[[U_S]]
\lo 
{\cal F}\otimes_{{\cal O}_Y}
\Om^{j+1}_{Y/\os{\circ}{S}}[[U_S]]
\quad (j\in {\mab Z}_{\geq 0})
\tag{2.1.10}\label{eqn:bbaadff}
\end{equation*}   
is an ${\cal O}_S$-linear morphism defined by the following 
\begin{align*} 
\nabla(m_1^{[e_1]}\cdots m_r^{[e_r]}\otimes \om)=&
\sum_{j=1}^rm_1^{[e_1]}\cdots m_j^{[e_j-1]} \cdots m_r^{[e_r]}d\log m_j\wedge \om 
\tag{2.1.11}\label{eqn:baadff}\\
&+
m_1^{[e_1]}\cdots  m_r^{[e_r]}\otimes \nabla(\om) 
%\quad (~\om \in {\cal F}
%\otimes_{{\cal O}_Y}
%\Om^j_{Y/\os{\circ}{S}},e_1,\ldots,e_r\in {\mab N}~).  
\end{align*}
%where $u_j^{[-1]}:=0$. 
as in (\ref{eqn:bdff}). 
It is easy to check an equality $\nabla^2=0$ for $\nabla$ in (\ref{eqn:baadff}).  
The projection 
$$\Gam_{{\cal O}_S}(L_S)/\Gam^{\geq n}_{{\cal O}_S}(L_S)(L_S)
\lo 
\Gam_{{\cal O}_S}(L_S)/\Gam^{\geq 1}_{{\cal O}_S}(L_S)
=\Gam^0_{{\cal O}_S}(L_S)={\cal O}_S \quad  (n\geq 1)$$
induces the following natural morphism of complexes: 
\begin{equation*}  
{\cal F}\otimes_{{\cal O}_Y}
\Om^{\bul}_{Y/\os{\circ}{S}}[[U_S]] \lo 
{\cal F}\otimes_{{\cal O}_Y}
\Om^{\bul}_{Y/S}.
\tag{2.1.12}\label{ali:uaafui} 
\end{equation*}   
The morphism (\ref{ali:uauri}) induces 
the following inclusion morphism  
\begin{align*}  
{\cal F}\otimes_{{\cal O}_Y}
\Om^{\bul}_{Y/\os{\circ}{S}}[U_S]
\os{\sus}{\lo} {\cal F}\otimes_{{\cal O}_Y}
\Om^{\bul}_{Y/\os{\circ}{S}}[[U_S]]. 
\tag{2.1.13}\label{ali:ufaui} 
\end{align*} 
We have the following commutative diagram 
\begin{equation*} 
\begin{CD} 
{\cal F}\otimes_{{\cal O}_Y}
\Om^{\bul}_{Y/\os{\circ}{S}}[U_S]@>{\subset}>>
{\cal F}\otimes_{{\cal O}_Y}
\Om^{\bul}_{Y/\os{\circ}{S}}[[U_S]]\\
@VVV @VVV \\
{\cal F}\otimes_{{\cal O}_Y}
\Om^{\bul}_{Y/S}@={\cal F}\otimes_{{\cal O}_Y}
\Om^{\bul}_{Y/S}. 
\end{CD}
\tag{2.1.14}\label{cd:ufyysi} 
\end{equation*}

The following is a (simpler) log analytic version of \cite[(4.3)]{nhi}:

\begin{defi}[{\rm \cite[(3.4)]{nf}}]\label{defi:wal}
(1) Take a local basis of $\Om^1_{Y/\os{\circ}{S}}$ containing 
$\{d\log t_1,\ldots,d\log t_r\}$. 
Let $(d\log t_i)^*\col \Om^1_{Y/\os{\circ}{S}}\lo {\cal O}_{Y}$ 
be the local morphism defined by the local dual basis of $d\log t_i$. 
We say that the connection $({\cal F},\nabla)$ has {\it no poles along} $S$ 
if the composite morphism 
${\cal F}\os{\nabla}{\lo} {\cal F}\otimes_{Y}\Om^1_{Y/\os{\circ}{S}}
\os{{\rm id}_{\cal F}\otimes(d\log t_i)^*}{\lo} {\cal F}$ vanishes for $1\leq \forall i\leq r$. 
\par 
(2) If there exists a finite increasing filtration 
$\{{\cal F}_i\}_{i\in {\mab Z}}$ on ${\cal F}$ locally on $Y$ 
such that ${\rm gr}_i{\cal F}:={\cal F}_i/{\cal F}_{i-1}$ 
is a locally free ${\cal O}_Y$-module 
and $\nabla$ induces a connection 
${\cal F}_i\lo {\cal F}_i\otimes_{{\cal O}_Y}\Om^1_{Y/\os{\circ}{S}}$ 
whose induced connection 
${\rm gr}_i{\cal F}\lo ({\rm gr}_i{\cal F})
\otimes_{{\cal O}_Y}\Om^1_{Y/\os{\circ}{S}}$  has no poles along $M_S$, 
then we say that $({\cal F},\nabla)$ 
is a {\it locally nilpotent integrable connection on $Y$ with respect to} $S$. 
\end{defi}

\par 
Let us recall the following result:  

\begin{theo}[{\rm \cite[(3.5)]{nf}}]\label{theo:afp}
Assume that $(F,\nabla)$ 
is a locally nilpotent integrable connection on $Y$ with respect to $S$. 
Then the morphism {\rm (\ref{ali:uafacui})} is a quasi-isomorphism. 
\end{theo} 

In the rest of this section we consider the functoriality of 
the quasi-isomorphism (\ref{ali:uafacui}). 
Let us recall a formalism of the functoriality in \cite{nf}. 
\par 
Let $S'$ be an analytic family over ${\mab C}$ of log points of virtual dimension $r'$ 
and let $v\col S\lo S'$ be a morphism of log analytic spaces. 
This morphism induces a morphism $v^*\col \ol{M}_{S'}\lo v_*(\ol{M}_{S})$.  
Locally on $S$, this morphism  is equal to a morphism 
$v^*\col {\mab N}^{r'}\lo {\mab N}^r$. 
Set 
$e_i=(0, \ldots,0,\os{i}{1},0,\ldots, 0)\in {\mab N}^s$ for $s=r$ or $r'$  
$(1\leq i\leq r)$. 
Let $A=(a_{ij})_{1\leq i\leq r, 1\leq j\leq r'} \in M_{rr'}({\mab N})$ 
be the representing matrix of $v^*$: 
$$v^*(e_1,\ldots,e_{r'})=(e_1,\ldots,e_r)A.$$ 
Let $t_1,\ldots,t_r$ and $t'_1,\ldots, t'_{r'}$ be local sections of 
$M_S$ and $M_{S'}$ 
whose images in $\ol{M}_S\os{\sim}{\lo} {\mab N}^r$ and 
$\ol{M}{}'_{S'}\os{\sim}{\lo}{\mab N}^{r'}$ are 
$e_1,\ldots,e_r$ and $e_1,\ldots,e_{r'}$, respectively. 
Then there exists a local section $b_j\in {\cal O}_{S}^*$ 
$(1\leq j\leq r')$ such that 
\begin{align*} 
v^*(t'_j)=b_jt^{a_{1j}}_1\cdots t^{a_{rj}}_r. 
\tag{2.3.1}\label{ali:vta} 
\end{align*}
Let $u_1,\ldots, u_r$ and $u'_1,\ldots,u'_{r'}$ be 
the corresponding local sections 
to $t_1,\ldots, t_r$ and $t'_1,\ldots,t'_{r'}$ of 
$L_S$ and $L_{S'}$, respectively. Then we define 
an ${\cal O}_{S'}$-linear morphism 
$v^*\col \Gam_{{\cal O}_{S'}}(L_{S'})\lo v_*(\Gam_{{\cal O}_S}(L_S))$
by the following formula: 
\begin{align*} 
v^*(u_j)=a_{1j}u_1+\cdots+a_{rj}u_r 
\tag{2.3.2}\label{ali:vua} 
\end{align*}
and by ${\rm Sym}(v^*)$ for the $v^*$ in (\ref{ali:vua}). 
Since ${\rm Aut}({\mab N}^s)={\mathfrak S}_s$ (\cite[p.~47]{nh3}), 
the morphism $v^*\col \Gam_{{\cal O}_{S'}}(L_{S'})\lo v_*(\Gam_{{\cal O}_S}(L_S))$
is independent of the choice of the local isomorphisms 
$\ol{M}_{S'}\simeq {\mab N}^{r'}$ and $\ol{M}_{S}\simeq {\mab N}^r$. 
In conclusion, 
we obtain the following well-defined morphism 
\begin{align*} 
v^*\col \Gam_{{\cal O}_{S'}}(L_{S'})\lo 
v_*(\Gam_{{\cal O}_S}(L_S))
\tag{2.3.3}\label{ali:kxvef}
\end{align*} 
of sheaves of commutative rings of unit elements on $S'$.  
This morphism satisfies the usual transitive relation 
\begin{align*} 
{\textrm ``}(v\circ v')^*=
v'{}^*v^*{\textrm '}{\textrm '}.
\tag{2.3.4}\label{ali:kxbvef}
\end{align*}

\par 
Let $g'\col Y'\lo S'$ be an analogous morphism of log analytic spaces to 
$g\col Y\lo S$. 
Assume that we are given a commutative diagram 
\begin{equation*} 
\begin{CD}
Y@>{h}>> Y'\\
@V{g}VV @VV{g'}V \\
S@>{v}>> S'. 
\end{CD}
\tag{2.3.5}\label{ali:ooxas}
\end{equation*}

\begin{prop}[Functoriality]\label{prop:rks}
Assume that we are given a commutative diagram 
\begin{equation*} 
\begin{CD}
{\cal F}'@>{\nabla'}>> {\cal F}'\otimes_{{\cal O}_{Y'}}\Om^1_{Y'/\os{\circ}{S}{}'}\\
@VVV @VVV \\
h_*({\cal F}) @>{h_*(\nabla)}>> h_*({\cal F}\otimes_{{\cal O}_X}\Om^1_{Y/\os{\circ}{S}}), 
\end{CD}
\tag{2.4.1}\label{cd:kgsx}
\end{equation*}
where $({\cal F}',{\nabla'})$ is an analogous connection to 
$({\cal F},{\nabla})$ on $Y'/\os{\circ}{S}{}'$. 
Then the morphism {\rm (\ref{ali:uafacui})} is 
contravariantly functorial for 
the commutative diagrams {\rm (\ref{ali:ooxas})} and {\rm (\ref{cd:kgsx})}. 
This contravariance satisfies the usual transitive relation ``$(i\circ h)^*
=h^*\circ i^*$''. 
\end{prop} 
\begin{proof} 
Obvious. 
\end{proof} 

\section{Results on completed Hirsch extensions}\label{sec:hie} 
In this section we give an easy result (\ref{prop:cc}) 
on completed Hirsch extensions. 
Using this result and the result (\ref{theo:afp}) in the previous section, we see that 
the morphism (\ref{ali:ufaui}) is a quasi-isomorphism.  
Though we do not use this result to obtain the main result in this article, 
it is of independent interest. 
%we use the technique of the proof for this result in the next section, which is necessary 
%for the main result. 
%Hence, even 
If one wishes to know the proof of the main result in this article, 
one can skip this section and 
(\ref{lemm:qslux}) and (\ref{prop:ys}) in the next section. 
\par 

\begin{lemm}\label{lemm:uqse}
Let the notations be as in {\rm (\ref{lemm:qseux})}. 
Let $\{t_1,\ldots,t_r\}$ be a set of local sections of $M_S$ whose images in 
$\ol{M}_S$ is a unique system of local generators of $\ol{M}_S$. 
Denote $t_i$ in $L_S$ by $u_i$. 
For $1\leq i\leq r$, let us consider the completed Hirsch extensions 
${\cal F}\otimes_{{\cal O}_Y}
{\Om}^{\bul}_{Y/S_{i-1}}[[u_{i+1},\ldots,u_r]]$ 
and 
${\cal F}
\otimes_{{\cal O}_Y}
{\Om}^{\bul}_{Y/S_i}[[u_{i+1},\ldots,u_r]]$ of 
${\cal F}\otimes_{{\cal O}_Y}
{\Om}^{\bul}_{Y/S_{i-1}}$ and 
${\cal F}
\otimes_{{\cal O}_Y}
{\Om}^{\bul}_{Y/S_i}$
by the morphisms 
$\bigoplus_{j=i+1}^r{\cal O}_Su_j\owns u_j\lom d\log t_j\in {\Om}^1_{Y/S_{i-1}}$ 
and $\bigoplus_{j=i+1}^r{\cal O}_Su_j\owns u_j\lom d\log t_j\in{\Om}^1_{Y/S_i}$, respectively. 
Then the following sequence 
\begin{align*} 
0  &\lo {\cal F}
\otimes_{{\cal O}_Y}
{\Om}^{\bul}_{Y/S_i}[[u_{i+1},\ldots,u_r]][-1] 
\os{{\rm id}_{\cal F}\otimes (d\log t_i\wedge)}{\lo} 
{\cal F}\otimes_{{\cal O}_Y}
{\Om}^{\bul}_{Y/S_{i-1}}[[u_{i+1},\ldots,u_r]]
\tag{3.1.1}\label{ali:bagaxd}\\
&\lo {\cal F}\otimes_{{\cal O}_Y}{\Om}^{\bul}_{Y/S_i}[[u_{i+1},\ldots,u_r]] \lo 0 
\quad (1\leq i\leq r)
\end{align*} 
is exact. 
\end{lemm} 
\begin{proof} 
This immediately follows from (\ref{lemm:qseux}) by considering 
``coefficient forms'' of each $u_{i+1}^{e_{i+1}} \cdots u_r^{e_r}$ 
$(e_{i+1}, \ldots e_r\in {\mab N})$.  
\end{proof}

%Let the notations be as in the previous section. 
%We need the following result in this article: 

\begin{prop}\label{prop:cc}
The morphism {\rm (\ref{ali:uaafui})} is a quasi-isomorphism. 
\end{prop}
\begin{proof} 
The problem is local. 
We may assume that there exists a sequence (\ref{ali:ffoy}) and 
that $L_S=\oplus_{i=1}^r{\cal O}_Su_i$. 
The complex 
${\cal F}\otimes_{{\cal O}_Y}
{\Om}^{\bul}_{Y/S_{r-1}}[[u_r]]$ is, by definition,  
the following  complex: 
\begin{equation*} 
\begin{CD}  
\cdots @>>> \cdots @>>> \cdots \\
@A{{\rm id}\otimes \nabla}AA @A{{\rm id}\otimes \nabla}AA 
@A{{\rm id}\otimes \nabla}AA \\
{\cal O}_Su^2_r
\otimes_{{\cal O}_S}{\cal F}
@>{(?)'d\log t_r\wedge}>> {\cal O}_Su_r\otimes_{{\cal O}_S}
{\cal F}\otimes_{{\cal O}_Y}
\Om^1_{Y/S_{r-1}}  
@>{(?)'d\log t_r\wedge}>> {\cal F}
\otimes_{{\cal O}_Y}\Om^2_{Y/S_{r-1}} \\
@. @A{{\rm id}\otimes \nabla}AA @A{{\rm id}\otimes \nabla}AA \\
@. {\cal O}_Su_r
\otimes_{{\cal O}_S}{\cal F}@>{(?)'d\log t_r\wedge}>>{\cal F}
\otimes_{{\cal O}_Y}\Om^1_{Y/S_{r-1}} 
\\
@. @. @A{\nabla}AA  \\
@. @. {\cal F} ,
\end{CD} 
\tag{3.2.1}\label{cd:hirfd}
\end{equation*} 
where the horizontal arrow $(?)'d\log t_r\wedge$ is defined by 
$u^{[i]}_r\otimes f\otimes \om \lom 
u^{[i-1]}_r\otimes f\otimes (d\log t_r\wedge \om)$ 
$(f\in {\cal F},\om\in \Om^{\bul}_{Y/S_{r-1}})$. 
This is augmented to the complex 
${\cal F}\otimes_{{\cal O}_Y}\Om^{\bul}_{Y/S}$. 
By (\ref{ali:agaxd}) we see that this augmentation 
\begin{align*} 
{\cal F}\otimes_{{\cal O}_Y}\Om^{\bul}_{Y/S_{r-1}}[[u_r]]
\lo {\cal F}\otimes_{{\cal O}_Y}\Om^{\bul}_{Y/S}
\tag{3.2.2}\label{ali:hiysd}
\end{align*} 
is an isomorphism. 
The complex 
${\cal F}\otimes_{{\cal O}_Y}
{\Om}^{\bul}_{Y/S_{r-2}}[[u_r,u_{r-1}]]=
({\cal F}\otimes_{{\cal O}_Y}
{\Om}^{\bul}_{Y/S_{r-2}}[[u_r]])[[u_{r-1}]]$ is, by definition,  
the following  complex: 
\begin{equation*} 
{\footnotesize{
\begin{CD}  
\cdots @>>> \cdots @>>> \cdots \\
@A{{\rm id}\otimes \nabla}AA @A{{\rm id}\otimes \nabla}AA 
@A{{\rm id}\otimes \nabla}AA \\
{\cal O}_Su^2_{r-1}
\otimes_{{\cal O}_S}{\cal F}[[u_r]]
@>{(?)'d\log t_{r-1}\wedge}>> 
{\cal O}_Su_{r-1}\otimes_{{\cal O}_S}{\cal F}\otimes_{{\cal O}_Y}
\Om^1_{Y/S_{r-2}}[[u_r]]  
@>{(?)'d\log t_{r-1}\wedge}>> {\cal F}
\otimes_{{\cal O}_Y}\Om^2_{Y/S_{r-2}}[[u_r]] \\
@. @A{{\rm id}\otimes \nabla}AA @A{{\rm id}\otimes \nabla}AA \\
@. {\cal O}_Su_{r-1}
\otimes_{{\cal O}_S}{\cal F}[[u_r]]@>{(?)'d\log t_{r-1}\wedge}>>{\cal F}
\otimes_{{\cal O}_Y}\Om^1_{Y/S_{r-2}}[[u_r]] 
\\
@. @. @A{\nabla}AA  \\
@. @. {\cal F}. 
\end{CD}}}
\tag{3.2.3}\label{cd:hirsfd}
\end{equation*} 
The augmentation 
\begin{align*} 
({\cal F}\otimes_{{\cal O}_Y}
{\Om}^{\bul}_{Y/S_{r-2}}[[u_r]])[[u_{r-1}]]\lo {\cal F}\otimes_{{\cal O}_Y}
{\Om}^{\bul}_{Y/S_{r-1}}[[u_r]]
\tag{3.2.4}\label{ali:hfd}
\end{align*} 
is an isomorphism by (\ref{lemm:uqse}). 
Continuing this argument repeatedly, we see that 
the morphism {\rm (\ref{ali:uaafui})} is a quasi-isomorphism. 
\end{proof} 

Though we do not use the following in this article, this is of independent interest 
as in \cite[(3.17) (4)]{ey}: 

\begin{coro}\label{coro:qii}
The morphism {\rm (\ref{ali:ufaui})} is a quasi-isomorphism.  
\end{coro}
\begin{proof} 
This immediately follows from (\ref{cd:ufyysi}), (\ref{theo:afp}) and (\ref{prop:cc}). 
\end{proof}

\begin{prop}[Functoriality]\label{prop:raks}
Let the notations be as in {\rm (\ref{prop:rks})}. 
Then the morphism {\rm (\ref{ali:ufaui})} is 
contravariantly functorial for 
the commutative diagrams {\rm (\ref{ali:ooxas})} and {\rm (\ref{cd:kgsx})}. 
This contravariance satisfies the usual transitive relation ``$(i\circ h)^*
=h^*\circ i^*$''. 
\end{prop} 
\begin{proof} 
Obvious. 
\end{proof} 

\section{Log analytic families}\label{sec:itst}
In this section we consider a ``thicker'' family 
than the family in the previous two sections. 
\par 
Let the notations be as in the previous section. 
Let $r$ be a positive integer. 
Set $\os{\circ}{\Del}{}^r_{\os{\circ}{S}}:=\os{\circ}{\Del}{}^r\times \os{\circ}{S}$. 
The zero section $\os{\circ}{S}\os{\sus}{\lo} \os{\circ}{\Del}_{\os{\circ}{S}}$ 
defines an fs log structure on $\os{\circ}{\Del}_{\os{\circ}{S}}$. 
Let $\Del_{\os{\circ}{S}}$ be the resulting log analytic space and let 
$\Del^r_{\os{\circ}{S}}$ be the $r$-times product of $\Del_{\os{\circ}{S}}$ over 
$\os{\circ}{S}$. 
Locally on $S$, we have the following exact closed immersion
\begin{align*} 
S\os{\sus}{\lo} \Del^r_{\os{\circ}{S}}
\end{align*} 
of log analytic spaces. 
This immersion fits into the following commutative diagram
\begin{equation*} 
\begin{CD} 
S@>{\sus}>> \Del^r_{\os{\circ}{S}}\\
@VVV @VVV \\
\os{\circ}{S}@=\os{\circ}{S},
\end{CD} 
\end{equation*} 
where the vertical morphisms are natural morphisms. 
\par
By replacing ${\mab A}^r_{\os{\circ}{S}}$ in \cite[(1.1)]{nb} by $\Del^r_{\os{\circ}{S}}$, 
we have a well-defined log analytic space $\ol{S}$ over $\os{\circ}{S}$
with an exact closed immersion $S\os{\sus}{\lo} \ol{S}$ 
fitting into the following commutative diagram over $\os{\circ}{S}$: 
\begin{equation*} 
\begin{CD} 
S@>{\sus}>> \ol{S}\\
@VVV @VV{p}V \\
\os{\circ}{S}@=\os{\circ}{S}. 
\end{CD} 
\tag{4.0.1}\label{ali:fsoy}
\end{equation*} 
Here $p\col \ol{S} \lo \os{\circ}{S}$ is the structural morphism. 
Locally on $S$, $\ol{S}$ is isomorphic to 
a unit polydisc $\Del^r_{\os{\circ}{S}}=\{(t_1,\ldots,t_r)\in {\mab C}^r~\vert~
\vert t_1\vert<1,\ldots,\vert t_r\vert<1\}$ with log structures 
${\mab N}^r\owns e_i\lom t_i\in {\cal O}_{\ol{S}}$ over $\os{\circ}{S}$; 
$\ol{S}$ is obtained by gluing $\{(t_1,\ldots,t_r)\in {\mab C}^r~\vert~
\vert t_1\vert<1,\ldots,\vert t_r\vert<1\}$ by 
the log structure $M_S$. It is important to 
note that, by the construction of $\ol{S}$, we obtain well-defined 
closed 1-forms $d\log t_1,\ldots, d\log t_r\in \Om^1_{\ol{S}/\os{\circ}{S}}$ 
for local sections  $t_1,\ldots,t_r$ of $M_S$ in \S\ref{sec:da}; 
we do not permit to take other differential forms 
$d\log (v_1t_1),\ldots, d\log (v_rt_r)$ for $v_1,\ldots v_r\in {\cal O}_{\ol{S}}^*$; 
we permit to take only $d\log (v_it_i)=d\log t_i\in \Om^1_{\ol{S}/\os{\circ}{S}}$ 
for $v_i\in {\cal O}_{S}^*$. 
This convention is similar to the convention in \cite{nb} in which 
to construct the isomorphism (\ref{ali:ckrxw}), we use 
the prime number $p\in {\cal V}$ (note that $p$ is not necessarily a uniformizer of 
${\cal V}$!) and the isomorphism ``$\rho_p$'' in \cite[(6.3.14)]{nb}. 
\par 
The local morphism 
\begin{align*} 
L_S \owns \sum_{i=1}^rc_i u_i\lom \sum_{i=1}^rc_i t_i\in {\cal O}_{\ol{S}}
\tag{4.0.2}\label{ali:osils} 
\end{align*} 
induces the following well-defined global morphism 
\begin{align*} 
\Gam_{{\cal O}_S}(L_S) \lo  {\cal O}_{\ol{S}}. 
\tag{4.0.3}\label{ali:osls} 
\end{align*} 

\begin{rema}\label{rema:sc}
Let $\eps>0$ be a real number. 
Set $\os{\circ}{\Del}(\eps):=\{z\in {\mab C}~\vert~\vert z\vert <\eps \}$. 
Replacing $\os{\circ}{\Del}$ by $\os{\circ}{\Del}(\eps)$, we obtain a log anaytic space 
$\ol{S}(\eps)$ (a tubular neighborhood of $S$ in $\ol{S}$). 
If $\os{\circ}{S}$ is quasi-compact and if $\eps$ is small enough, 
then we have a natural immersion 
\begin{align*} 
\ol{S}(\eps)\os{\sus}{\lo} \ol{S}. 
\tag{4.1.1}\label{ali:ollss} 
\end{align*} 
In this case, for a morphism ${\cal Z}\lo \ol{S}$, we denote 
$f\times_{\ol{S}}\ol{S}(\eps)\col 
{\cal Z}\times_{\ol{S}}\ol{S}(\eps)\lo S(\eps)$ by 
\begin{align*} 
f(\eps)\col {\cal Z}(\eps)\lo S(\eps).
\tag{4.1.2}\label{ali:olzlss} 
\end{align*}  
\end{rema}

\par 
In this section we assume that  the morphism 
$g\col Y\lo S$ in the previous section fits into 
the following cartesian diagram 
\begin{equation*} 
\begin{CD} 
Y@>{\sus}>> {\cal Y}\\
@V{g}VV @VV{h}V \\
S@>{\sus}>> {\ol{S}}, 
\end{CD} 
\tag{4.1.2}\label{ali:ffloy}
\end{equation*} 
where the morphism $h\col {\cal Y}\lo \ol{S}$ is log smooth.  
%Let $M_{\ol{S}}$ be the log structure of $\ol{S}$. 
%Locally on $S$,  there exists a family  
%$\{t_i\}_{i=1}^r$ of local sections of $M_{\ol{S}}$ 
%giving a local basis of $M_{\ol{S}}/{\cal O}_{\ol{S}}^*$. 
%Set $\wt{M}_i:=(t_i) \subset M_{\ol{S}}$. 
%For any $1\leq \forall i\leq r$, the ideal sheaf $\oplus_{j=1}^i\wt{M}_j$ and 
%$\os{\circ}{\ol{S}}$ 
%define an fs log analytic family. 
%Let $\wt{S}_i:=\ol{S}(\wt{M}_1,\ldots,\wt{M}_i)=
%(\os{\circ}{\ol{S}},(\oplus_{j=1}^i\wt{M}_j\lo {\cal O}_{\ol{S}})^a)$ 
%be the resulting local log analytic space. Set $\ol{S}_0:=\os{\circ}{S}$. 
%Locally on $\ol{S}$, we can define a sequence of fs log analytic spaces:  
%\begin{align*} 
%\ol{S}=\ol{S}_r\lo \ol{S}_{r-1}\lo \ol{S}_{r-2}\lo \cdots \lo 
%\ol{S}_1\lo \ol{S}_0=\os{\circ}{S}.
%\tag{4.0.2}\label{ali:fflsoy}
%\end{align*} 
%as in (\ref{ali:ffoy}). 
Let $U_S$ and $L_S$ be as in \S\ref{sec:da}. 
%As in \S\ref{sec:da}, we can define a sheaf $U_S$ on $\ol{S}$ 
%of $r$-pieces and a locally free ${\cal O}_{\ol{S}}$-module 
%$L_{\ol{S}}$ of rank $r$. 
By abuse of notation, 
denote by $d\log t_i\in \Om^1_{{\cal Y}/\os{\circ}{S}}$ the image of 
$d\log t_i$ by the morphism $h^*(\Om^1_{\ol{S}/\os{\circ}{S}})\lo 
\Om^1_{{\cal Y}/\os{\circ}{S}}$. 
Let $\ol{\cal F}$ be a locally free ${\cal O}_{\cal Y}$-module.
Let 
\begin{align*} 
\ol{\nabla}\col  \ol{\cal F}\lo 
\ol{\cal F}\otimes_{{\cal O}_{\cal Y}}\Om^1_{{\cal Y}/\os{\circ}{S}}
\tag{4.1.3}\label{ali:fffoy}
\end{align*} 
be an integrable connection.
As in the previous section, we can define the Hirsch extension 
$\ol{\cal F}\otimes_{{\cal O}_{\cal Y}}
\Om^{\bul}_{{\cal Y}/\os{\circ}{S}}[[U_S]]$ of 
the log de Rham complex
$\ol{\cal F}\otimes_{{\cal O}_{\cal Y}}\Om^{\bul}_{{\cal Y}/\os{\circ}{S}}$ 
by using the well-defined morphism 
$L_S\owns u_j:=t_j\lom d\log t_j\in \Om^1_{{\cal Y}/\os{\circ}{S}}$: 
\begin{align*} 
\ol{\cal F}\otimes_{{\cal O}_{\cal Y}}
\Om^{\bul}_{{\cal Y}/\os{\circ}{S}}[[U_S]]
&:=\Gam_{{\cal O}_S}(L_S)\otimes_{{\cal O}_S}
\ol{\cal F}\otimes_{{\cal O}_{\cal Y}}
\Om^{\bul}_{{\cal Y}/\os{\circ}{S}}
\tag{4.1.4}\label{ali:ffaloy}\\
&:=
(p\circ h)^{-1}(\Gam_{{\cal O}_S}(L_S)\otimes_{(p\circ h)^{-1}({\cal O}_S)}
\ol{\cal F}\otimes_{{\cal O}_{\cal Y}}
\Om^{\bul}_{{\cal Y}/\os{\circ}{S}}. 
\end{align*} 
We also have the log de Rham complex 
$\ol{\cal F}\otimes_{{\cal O}_{\cal Y}}\Om^{\bul}_{{\cal Y}/S}$. 
\par 
The following remark is important: 
\begin{rema}
In the definition (\ref{ali:ffaloy}), 
we consider $L_{S}$ as a sheaf on $\os{\circ}{S}$ {\it not} on $S$. 
The reader should also note that we do not consider 
$\ol{M}_{\ol{S}}:=M_{\ol{S}}/{\cal O}_{\ol{S}}^*$, $U_{\ol{S}}$ 
nor $L_{\ol{S}}$ because the sheaf $\ol{M}_{\ol{S}}$ is {\it not} locally constant.  
\end{rema}
%the log de Rham complex
%$\ol{\cal F}\otimes_{{\cal O}_{\cal Y}}\Om^{\bul}_{{\cal Y}/\os{\circ}{S}}$ 
%by using the well-defined morphism 
%$L_{\ol{S}}\owns u_j:=t_j\lom d\log t_j\in \Om^1_{{\cal Y}/\os{\circ}{S}}$: 

\begin{lemm}\label{lemm:qslux}
Let the notations be as in {\rm (\ref{lemm:qseux})}. 
Let $\{t_1,\ldots,t_r\}$ be a set of local sections of $M_S$ whose images in 
$\ol{M}_S$ is a unique system of local generators of $\ol{M}_S$. 
Denote $t_i$ in $L_S$ by $u_i$. 
Then the following hold$:$ 
\par 
$(1)$ The sheaf $\Om^i_{\ol{\cal Y}/\os{\circ}{S}}$ $(i\in {\mab N})$ is a locally free 
${\cal O}_{\ol{\cal Y}}$-module. 
\par 
$(2)$ Locally on $\ol{S}$, the following sequence 
\begin{align*} 
0  \lo \ol{\cal F}
\otimes_{{\cal O}_{\cal Y}}
{\Om}^{\bul}_{{\cal Y}/\ol{S}_i}[-1] 
\os{{\rm id}_{\cal F}\otimes (d\log t_i\wedge)}{\lo} 
\ol{\cal F}\otimes_{{\cal O}_{\cal Y}}
{\Om}^{\bul}_{{\cal Y}/\ol{S}_{i-1}}  
\lo \ol{\cal F}\otimes_{{\cal O}_{\cal Y}}{\Om}^{\bul}_{{\cal Y}/\ol{S}_i} \lo 0 
\quad (1\leq i\leq r). 
\tag{4.3.1}\label{ali:cagaxd}
\end{align*} 
is exact. 
\par 
$(3)$ 
For $1\leq i\leq r$, let us consider the completed Hirsch extensions 
$\ol{\cal F}\otimes_{{\cal O}_{\cal Y}}
{\Om}^{\bul}_{{\cal Y}/\ol{S}_{i-1}}[[u_{i+1},\ldots,u_r]]$ 
and 
$\ol{\cal F}\otimes_{{\cal O}_{\cal Y}}
\Om^{\bul}_{{\cal Y}/\ol{S}_i}[[u_{i+1},\ldots,u_r]]$ of 
$\ol{\cal F}\otimes_{{\cal O}_{\cal Y}}
\Om^{\bul}_{{\cal Y}/\ol{S}_{i-1}}$ and 
$\ol{\cal F}\otimes_{{\cal O}_{\cal Y}}\Om^{\bul}_{{\cal Y}/\ol{S}_i}$
by the morphisms 
$\bigoplus_{j=i+1}^r{\cal O}_{\ol{S}}u_j\owns u_j\lom d\log t_j\in 
\Om^1_{{\cal Y}/\ol{S}_{i-1}}$ 
and $\bigoplus_{j=i+1}^r{\cal O}_{\ol{S}}
u_j\owns u_j\lom d\log t_j\in \Om^1_{{\cal Y}/\ol{S}_i}$, 
respectively. 
Then the following sequence 
\begin{align*} 
0  &\lo \ol{\cal F}
\otimes_{{\cal O}_{\cal Y}}
\Om^{\bul}_{{\cal Y}/\ol{S}_i}[[u_{i+1},\ldots,u_r]][-1] 
\os{{\rm id}_{\cal F}\otimes (d\log t_i\wedge)}{\lo} 
\ol{\cal F}\otimes_{{\cal O}_{\cal Y}}
{\Om}^{\bul}_{{\cal Y}/\ol{S}_{i-1}}[[u_{i+1},\ldots,u_r]]
\tag{4.3.2}\label{ali:afxd}\\
&\lo 
\ol{\cal F}\otimes_{{\cal O}_{\cal Y}}
\Om^{\bul}_{{\cal Y}/\ol{S}_i}[[u_{i+1},\ldots,u_r]] \lo 0 
\quad (1\leq i\leq r)
\end{align*} 
is exact.
\end{lemm}
\begin{proof} 
The proof is the same as that of \cite[(3.1)]{nf} and (\ref{lemm:uqse}); 
the proof is easy. 
\end{proof}

\begin{prop}\label{prop:ys}
The natural morphism 
\begin{align*} 
\ol{\cal F}\otimes_{{\cal O}_{\cal Y}}
\Om^{\bul}_{{\cal Y}/\os{\circ}{S}}[[U_S]]
\lo \ol{\cal F}\otimes_{{\cal O}_{\cal Y}}
\Om^{\bul}_{{\cal Y}/\ol{S}}
\tag{4.4.1}\label{ali:fouu}
\end{align*} 
is a quasi-isomorphism. 
This quasi-isomorphism is contravariantly functorial for a morphism from 
the commutative diagram {\rm (\ref{ali:ffloy})} to a similar commutative diagram 
and a morphism to the integrable connection 
{\rm (\ref{ali:fffoy})} from a similar integrable connection on the similar diagram. 
\end{prop}
\begin{proof}
By using (\ref{lemm:qslux}), 
the proof is the same as that of (\ref{prop:cc}). 
\end{proof}

\par
Set ${\cal F}:=\ol{\cal F}\otimes_{{\cal O}_{\cal Y}}{\cal O}_Y$. 
Then the connection (\ref{ali:fffoy}) induces the following connection: 
\begin{align*} 
\nabla\col {\cal F}\lo 
{\cal F}\otimes_{{\cal O}_{\cal Y}}\Om^1_{Y/\os{\circ}{S}}. 
\tag{4.4.2}\label{ali:fffaoy}
\end{align*} 
Indeed, this is a local question. 
We may assume that $\ol{S}=\Del^r_{\os{\circ}{S}}$. 
Then $\ol{\nabla}(t_if)=dt_if+t_i\ol{\nabla}(f)=t_i(fd\log t_i+\ol{\nabla}(f))$ 
$(f\in \ol{\cal F})$. 
Hence $\ol{\nabla}$ induces the connection $\nabla$ in (\ref{ali:fffaoy}).

\section{SNCL analytic spaces}\label{sec:lp}
Let $r$ be a positive integer. 
Let $S$ be a log analytic family of log points of virtual dimension $r$. 
In this section we give a definition of 
an (S)NCL(=(simple) normal crossing log) analytic space 
over $S$. 
%Here LP stands for ``locally product''. 
To define it, we have only to mimic the definition of 
an (S)NCL(=(simple) normal crossing log) analytic scheme defined 
%over $S$ suitably which is a log analytic version of 
%the definition of an (S)NCL scheme 
in \cite{nb} suitably. 
An (S)NCL analytic space over $S$ 
is locally the finitely many product of (S)NCL analytic spaces in the usual sense. 
The definition is an (obvious) generalization of the definition of 
a (generalized) semistable analytic space over a log point of ${\mab C}$ 
virtual dimension $r$ in \cite{fc} and \cite{fr}.
%As in \cite{fr}, we call a generalized semistable analytic space 
%a semistable analytic space simply. 

\par 
Let $B$ be an analytic space over ${\mab C}$. 
Let $a_1<\ldots <a_r$ be nonnegative integers. 
Set ${\bf a}:=(a_1,\ldots a_r)$ 
and let $d$ be a nonnegative integer such that 
$a_r\leq d+r$. Let $z_1,\ldots, z_{d+r}$ be a system of 
standard coordinates of $\os{\circ}{\Del}{}^{d+r}$. 
Let 
$\os{\circ}{\Del}_B({\bf a},d+r)$ be a closed analytic space of 
$\os{\circ}{\Del}{}^{d+r}\times B$ defined by an ideal sheaf 
$$(z_1\cdots z_{a_1},z_{a_1+1}\cdots z_{a_2}, \ldots,z_{a_{r-1}+1}\cdots z_{a_r}).$$ 

\begin{defi}\label{defi:zbsnc}
Let $Z$ be an analytic space over $B$ 
with structural morphism $g \col Z \lo B$. 
We call $Z$ an
{\it $($S$)$NC$($=$($simple$)$ normal crossing$)$ analytic space} over 
$B$ if $Z$ is a union of (smooth) analytic space 
$\{Z_{\lam}\}_{\lam \in \Lam}$ over $B$ 
($\Lam$ is a set) 
and if, for any point of $x \in Z$, 
there exist an open neighborhood $V$ of 
$x$ and an open neighborhood $W$ of 
$g(x)$ such that there exists an \'{e}tale morphism  
$\pi \col V \lo \os{\circ}{\Del}_B({\bf a},d+r)$ 
such that 
\begin{equation*} 
\{Z_{\lam}\vert_{V}\}_{\lam \in \Lam} 
=\{\pi^*(z_{b_1})=\cdots=\pi^*(z_{b_r})=0\}_{1\leq b_1\leq a_1,\ldots, 
a_{r-1}+1\leq b_r\leq a_r}
\tag{5.1.1}\label{eqn:defilsnc}
\end{equation*}   
where $a_1,\ldots,a_r$ and $d$ are nonnegative integers 
such that $\sum_{i=1}^ra_i\leq d+r$,  which depend on 
a local neighborhood of $z$ in $Z$ and 
$Z_{\lam}\vert _V:=Z_{\lam}\cap V$.  
The strict meaning of the equality (\ref{eqn:defilsnc})
is as follows.  
There exists a subset $\Lam(V)$ of $\Lam$ such that
there exists a bijection $i\col \Lam(V) \lo \{1,\ldots,a_1\}\times 
\cdots \times\{a_{r-1}+1,\ldots,a_r\}$ 
such that, if 
$\lam \not\in \Lam(V)$, then 
$Z_{\lam}\vert_{V}=\emptyset$ 
and if $\lam \in \Lam(V)$, then 
$Z_{\lam}$ is a closed analytic space defined by 
the ideal sheaf $(\pi^*(z_{b_1}),\ldots,\pi^*(z_{b_r}))$.  
We call the set $\{Z_{\lam}\}_{\lam \in \Lam}$ 
a {\it decomposition of $Z$ by $($smooth$)$ components of} 
$Z$ over $B$. In this case, we call $Z_{\lam}$ a 
{\it $($smooth$)$ component} of $Z$ over $B$. 
\end{defi} 

\parno
Let $Z$ be an SNC analytic space over $B$. 
For a nonnegative integer $k$, 
set 
\begin{equation}
Z_{\{\lam_0, \lam_1,\ldots \lam_k\}} 
:=Z_{\lam_0}\cap \cdots \cap Z_{\lam_k} \quad 
(\lam_i \not= \lam_j~{\rm if}~i\not= j) 
\tag{5.1.2}\label{eqn:parlm}
\end{equation}
and set
\begin{equation}
Z^{(k)} =  
\us{\{\lam_0, \ldots,  \lam_{k}~\vert~\lam_i 
\not= \lam_j~(i\not=j)\}}{\coprod}
Z_{\{\lam_0, \lam_1, \ldots, \lam_k\}}.   
\tag{5.1.3}\label{eqn:kfdintd}
\end{equation} 
For a negative integer $k$, set 
$Z^{(k)}=\emptyset$.  
Set $\Gam:=\{Z_{\lam}\}_{\lam \in \Lam}$ and 
$\Gam_{V}:=\{Z_{\lam}\vert_{V}\}_{\lam \in \Lam}$
for an open analytic space $V$ of $Z$.

First assume that $M_S$ is the free hollow log structure (\cite{nb}). 
We fix an isomorphism 
\begin{align*} 
(M_S,\al_S)\simeq ({\mab N}^r\oplus {\cal O}_S^*\lo {\cal O}_S) 
\tag{5.1.4}\label{ali:nsos} 
\end{align*} 
globally on $S$. 
Let $M_S({\bf a},d+r)$ be the log structure on 
$\os{\circ}{\Del}_{\os{\circ}{S}}({\bf a},d+r)$ associated to the following morphism 
\begin{equation*} 
{\mab N}^{{\oplus}a_r}\owns 
%e_i:=
(0, \ldots,0,\os{i}{1},0,\ldots, 0)\lom z_i \in 
{\cal O}_{\os{\circ}{\Del}_{\os{\circ}{S}}({\bf a},d+r)}.
\tag{5.1.5}  
\end{equation*} 
Let ${\Del}_B({\bf a},d+r)$ be the resulting log analytic space over $S$.
We have the ``multi-diagonal'' morphism 
${\mab N}^r\lo {\mab N}^{a_1}\oplus {\mab N}^{a_2-a_1}
\oplus \cdots \oplus {\mab N}^{a_r-a_{r-1}}
={\mab N}^{a_r}$ 
induces a morphism ${\Del}_S({\bf a},d+r) \lo S$ of log analytic spaces.  
We call ${\Del}_S({\bf a},d+r)$ the {\it standard NCL analytic space}.

\par 
Let $S$ be a family of log points of virtual dimension $r$  
($M_S$ is not necessarily free). 
%Let $S'\lo S$ be a family of log points. 
%($M_S$ and $M_{S'}$ are not necessarily free). 
%{\it locally free log structure of rank} $1$ on $\os{\circ}{S}$ 
%and we call $S$ {\it a family of log points}. 

\begin{defi}\label{defi:lfac}  
Let $f \col X(=(\os{\circ}{X},M_X)) \lo S$ be a morphism of 
log analytic spaces such that $\os{\circ}{X}$ is 
an (S)NC analytic space over $\os{\circ}{S}$ with a decomposition 
$\Del:=\{\os{\circ}{X}_{\lam}\}_{\lam \in \Lam}$ 
of $\os{\circ}{X}/\os{\circ}{S}$ by its smooth components. 
We call $f$ (or $X/S$) an {\it NCL$($=normal crossing log$)$ analytic space} if, 
for any point of $x \in \os{\circ}{X}$, 
there exist an open neighborhood $\os{\circ}{V}$ of 
$x$ and an open neighborhood $\os{\circ}{W}$ of 
$\os{\circ}{f}(x)$ such that $M_W$ is 
the free hollow log structure of rank $r$ 
and such that  $f\vert_V$ factors through 
a strict \'{e}tale morphism 
$\pi \col V {\lo} {\Del}_S({\bf a},d+r)$ over $W$
for some ${\bf a},d$ depending on $\os{\circ}{V}$
such that $\Del_{\os{\circ}{V}}
=\{\pi^*(z_{b_1})=\cdots=\pi^*(z_{b_r})=0\}_{1\leq b_1\leq a_1,\ldots, 
a_{r-1}+1\leq b_r\leq a_r}$ in $\os{\circ}{V}$.  
If $\os{\circ}{X}/\os{\circ}{S}$ is an SNC analytic space, then 
we call $f$ or $X/S$ an {\it SNCL analytic space}. 
\end{defi}

In the case $r=1$, an SNCL analytic space is an SNCL analytic space 
in the usual sense (\cite{fup}, \cite{nhi}).

%\begin{prop}\label{prop:cls}
%Let $f \col X \lo S$ be a GSNCL analytic space 
%with a decomposition 
%$\Del:=\{\os{\circ}{X}_{\lam}\}_{\lam \in \Lam}$ 
%of $\os{\circ}{X}/\os{\circ}{S}$ by its smooth components. 
%Then there exists canonical subsheaf $M_1(f),\ldots,M_r(f)$ of commutative monoids 
%of $M_X$ such that $\bigoplus_{i=1}^nM_i(f)=M_X$. 
%\end{prop}
%\begin{proof}
%First we construct $M_i(f)$ locally. 
%We assume that $f$ is the natural morphism 
%$\Del_S({\bf a},d+r)\lo S$. 
%Let $M_i(f)$ be the associated log structure of the morphism 
%${\mab N}^{\oplus (a_i-a_{i-1})}\owns e_j\lom 
%z_{a_{i-1}+j}\in {\cal O}_{\Del_S({\bf a},d+r)}$. 
%It is easy to check $M_i(f)$ is independent of 
%the local isomorphism $X\os{\sim}{\lo} \Del_S({\bf a},d+r)$ over $S$ 
%and the coordinates $z_i$'s. 
%\end{proof}

\begin{coro}\label{coro:lps}
Let $f \col X \lo S$ be an SNCL analytic space 
with a decomposition 
$\Del:=\{\os{\circ}{X}_{\lam}\}_{\lam \in \Lam}$ 
of $\os{\circ}{X}/\os{\circ}{S}$ by its smooth components. 
Then, locally on $X$, $X$ is a finitely many product of SNCL analytic spaces 
in the usual sence with some decomposition of their smooth components. 
Consequently $f$ is log smooth. 
\end{coro}
\begin{proof} 
Obvious. 
\end{proof} 

%\begin{defi} 
%Let $D=\{D_{\mu}\}_{\mu}$ be a closed analytic space of $\os{\circ}{X}$. 
%\end{defi} 

The following is an obvious imitation of \cite[(2.1.7)]{nh2}. 

\begin{defi}\label{defsncd}
Let $f \col X \lo S$ be an SNCL analytic space 
with a decomposition 
$\Del:=\{\os{\circ}{X}_{\lam}\}_{\lam \in \Lam}$ 
of $\os{\circ}{X}/\os{\circ}{S}$ by its smooth components. 
We call an effective Cartier divisor $D$ on $X/S$ 
is a {\it relative simple normal crossing divisor 
$(=:$relative SNCD$)$} 
\index{relative simple normal crossing divisor(=:relative SNCD)}
on $X/S$ if there exists a set $\Gam:=\{D_{\mu}\}_{\mu}$ of 
non-zero effective Cartier divisors on $X/S$ of locally finite intersection 
such that 
\begin{equation*}  
D = \sum_{\mu}D_{\mu}
\quad \text{in} \quad {\rm Div}(X/S)_{\geq 0}, 
\tag{5.4.1}
\label{eqn:dcsncd}
\end{equation*} 
such that $D_{\mu}\vert_{Z_{\lam}}$ is a smooth divisor on $Z_{\lam}$ for any $\lam$ 
and, for any point $x$ of $D$, there exist 
an open neighborhood $V\simeq \Del_S({\bf a},d+r)$ of $x$ in $X$ and 
$D\vert_V=\bigcup_{c=a_r+1}^b\{z_c=0\}$ for some $a_r+1\leq b\leq d+r$. 
%\begin{equation*}
%\begin{CD}
%D\vert_V @>{\subset}>> V\\ 
%@V{}VV  @VV{g}V \\
%\ul{\rm Spec}_{S_0}({\cal O}_{S_0}
%[y_1, \ldots, y_d]/(y_1\cdots y_s)) 
%@>{}>> \ul{\rm Spec}_{S_0}({\cal O}_{S_0}[y_1, \ldots, y_d]) 
%\end{CD}
%\tag{5.4.2}
%\label{cd:dvs}
%\end{equation*}
%(for some positive integers $s$ and $d$ 
%such that $s\leq d$), where the morphism $g$ is etale. 
\end{defi}
%Note that we do not require a relation a priori 
%between $\{D_{\lam}\vert_V\}_{\lam \in \Lam_V}$ 
%and the family $\{y_i=0\}_{i=1}^s$ of closed subschemes 
%in $V$ in the diagram  (\ref{cd:dvs}). 
%However, by (\ref{prop:mainsncd}) below, 
%we obtain $\{D_{\lam}\vert_V\}_{\lam \in \Lam_V}=
%\{\{y_i=0\}\}_{i=1}^s$ in the diagram (\ref{cd:dvs}) 
%if $V$ is small. 
%We call a smooth divisor on $X/S$ contained in $D$ 
%a {\it smooth component}\index{smooth component} of $D$. 
We call $\Gam=\{D_{\mu}\}$ a {\it decomposition of} $D$
{\it by smooth components of $D$ over} $S$.  
Note that $D_{\mu}$ itself is not necessarily smooth over $\os{\circ}{S}$; 
$D_{\mu}$ is a union of smooth analytic space over $\os{\circ}{S}$. 
\par 
Let ${\rm Div}_D(X/S)_{\geq 0}$ be a submonoid of 
${\rm Div}(X/S)_{\geq 0}$ consisting 
of effective Cartier divisors $E$'s on $X/S$
such that there exists an open covering 
$X = \bigcup_{i \in I}V_i$ (depending on $E$) 
of $X$ such that 
$E \vert_{V_i}$ is contained 
in the submonoid of ${\rm Div}(V_i/S)_{\geq 0}$ 
generated by $D_{\lam}\vert_{V_i}$ $(\lam \in \Lam)$. 
As in \cite[(A.0.1)]{nh2},  we see that the definition of 
${\rm Div}_D(X/S)_{\geq 0}$ is 
independent of the choice of $\Del$. 
\par 
Let $M(D)'$ be a presheaf of monoids on $X$ defined as follows: 
for an open sub analytic space $V$ of $X$, 
\begin{align*}
\Gamma(V, M(D)'):=
\{(E,a)\in & {\rm Div}_{D\vert_V} (V/S)_{\geq 0} 
\times \Gamma(V, {\cal O}_X)\vert \tag{5.4.2}\label{ali:gm} \\
& a\text{ is a generator of } \Gamma(V, {\cal O}_X(-E))\} 
%\nonumber 
\end{align*} 
with a monoid structure defined by 
$(E,a) \cdot (E',a') := 
(E + E', aa')$. 
The  natural  morphism 
$M(D)' \lo {\cal O}_X$ defined by 
the second projection 
$(E,a) \mapsto a$
induces a morphism 
$M(D)' \lo ({\cal O}_X,*)$
of presheaves of monoids on $X$.
The log structure $M(D)$ is, 
by definition, the associated log 
structure to the sheafification of $M(D)'$. 
Because ${\rm Div}_{D\vert_V}(V/S)_{\geq 0}$ 
is independent of the choice of the decomposition 
of $D\vert_V$ by smooth components, 
$M(D)$ is independent of 
the choice of the decomposition of $D$ 
by smooth components of $D$. 

%\begin{rema}\label{rema:grhc}
%Let $\pi \col Y\lo Z$ be a proper morphism of complex analytic spaces. 
%Assume that $Z$ is reduced. Usually Grauert's theorem has been stated 
%for a $\pi$-flat coherent sheaf. However the proof for this theorem 
%is valid for a bounded complex whose any component 
%is a coherent ${\cal O}_Y$-module with 
%$\pi^{-1}({\cal O}_Y)$-linear boundary morphism. 
%\end{rema}

\section{Main results}\label{sec:m} 
In \cite{fc} Fujisawa has introduced the notion of a generalized semistable family over 
a unit polydisc. In this article we give a generalization of this notion and we call this 
a semistable family. In this section we give a main result in this article. 
\par 
Let the notations be as in the previous section. 
Let $f\col {\cal X}\lo \ol{S}$ be a log smooth morphism. 
We assume that $\os{\circ}{f} \col \os{\circ}{\cal X}\lo \os{\circ}{S}$ 
is smooth and that, locally on $\os{\circ}{\cal X}$, 
this is a product of semistable families of analytic spaces in the usual sense. 
That is, locally on $\os{\circ}{\cal X}$, there exist a system of local coordinates 
$t_1,\ldots, t_r$ of $\ol{S}$ over $\os{\circ}{S}$
and a system of local coordinates 
$x_1, \ldots, x_{d+r}$ such that 
$x_1\cdots x_{a_1}=t_1, x_{a_1+1}\cdots x_{a_2}=t_2, 
\ldots, x_{a_{r-1}+1}\cdots x_{a_r}=t_r$ $(0\leq i_1<i_2<\cdots <i_r\leq d+r)$ 
over $\os{\circ}{S}$ (\cite[(6.2), (6.5)]{fc}; see also \cite[(0.3)]{ak}).  
We also assume that the morphism $f^{-1}(M_{\ol{S}})\lo M_{\cal X}$ of log structures 
of $f$ is locally isomorphic to the associated log structure to the multi-diagonal morphism 
${\mab N}^r\os{\subset}{\lo} {\mab N}^{a_1}\oplus {\mab N}^{a_2-a_1}\oplus \cdots \oplus 
{\mab N}^{a_r-a_{r-1}}={\mab N}^{a_r}$, 
where $M_S$ and $M_{\cal X}$ are the associated log structures to 
the morphism ${\mab N}^r\owns e_i\lom t_i\in {\cal O}_{\ol{S}}$ $(1\leq i\leq r)$ and 
${\mab N}^{a_{i+1}-a_i}\owns e_j\lom x_{j+a_i+1}\in {\cal O}_{\cal X}$ $(1\leq i\leq r)$
in the local situation above. 
We call ${\cal X}/\ol{S}$ a {\it semistable family} over $\ol{S}$. 
Let ${\cal D}$ be a horizontal NCD(=normal crossing divisor) on ${\cal X}/\ol{S}$. 
That is, it is a relative NCD on $\os{\circ}{\cal X}/\os{\circ}{\ol{S}}$ 
which is locally defined by an equation 
$x_{a_r+1}\cdots x_b=0$ for some $a_r+1\leq b\leq d+r$ in the local situation above. 
Set $X:={\cal X}\times_{\ol{S}}S$ and $D:={\cal D}\times_{\ol{S}}S$. 
Let $f_S\col X\lo S$ be the structural morphism. 

\par 

%More generally, we prove that the analogous morphism to 
%(\ref{ali:his}) obtained by the vector bundle $V$ on ${\cal X}$ 
%is a quasi-isomorphism. More precisely, we prove the following: 

\begin{lemm}\label{lemm:cn}
Let ${\cal E}$ be a coherent locally free ${\cal O}_{\cal X}$-module and let 
$\nabla \col {\cal E}\lo {\cal E}\otimes_{{\cal O}_{\cal X}}
\Om^1_{{\cal X}/{\mab C}}(\log D)$ be a locally nilpotent integrable connection 
on ${\cal X}$ 
and let ${\cal E}\otimes_{{\cal O}_{\cal X}}\Om^{\bul}_{{\cal X}/{\mab C}}(\log D)$
be the associated log de Rham complex to $\nabla$. 
Let $\iota \col X\os{\sus}{\lo} {\cal X}$ be the natural exact closed immersion. 
Then the following natural morphism 
\begin{align*}
%\otimes_{\mab C}
\iota^{-1}({\cal E}\otimes_{{\cal O}_{\cal X}}\Om^{\bul}_{{\cal X}/\os{\circ}{S}}(\log {\cal D}))
\lo 
%{\mab C}[\log \tau]
%\otimes_{\mab C}
{\cal E}\otimes_{{\cal O}_{\cal X}}\Om^{\bul}_{X/\os{\circ}{S}}(\log D)
\tag{6.1.1}\label{ali:hixs}
\end{align*} 
is a quasi-isomorphism of complexes of $f^{-1}({\cal O}_S)$-modules.  
This quasi-isomorphism is contravariantly functorial. 
\end{lemm} 
\begin{proof} 
(cf.~\cite[(2.10)]{sti}, \cite[(3.17) (2)]{ey})
Since the problem is local, we may assume that $\nabla$ is nilpotent by the five lemma.  
In this case, by \cite[(4.6) (2)]{kn}, there exists a local system $V$ of 
finite dimensional ${\mab C}$-vector spaces on $\os{\circ}{\cal X}$ 
such that ${\cal E}:={\cal O}_{\cal X}\otimes_{\mab C}V$ and 
$\nabla:=d\otimes {\rm id} \col {\cal E}\lo {\cal E}\otimes_{{\cal O}_{\cal X}}
\Om^1_{{\cal X}/{\mab C}}(\log D)$. 
Indeed, we have only to set $V:={\rm Ker}(\nabla)$. 
Let ${\cal E}\otimes_{{\cal O}_{\cal X}}\Om^{\bul}_{{\cal X}/{\mab C}}(\log D)$
be the associated log de Rham complex to $\nabla$. 
We may assume that $V={\mab C}$. 
%Because a generalized semistable family of 
%log analytic spaces over $\ol{S}$ is locally a product of 
%semistable families of log analytic spaces over a one-parameter family over 
%an analytic space, 
%we may assume that $r=1$ 
%(cf.~the following argument on the tensor product of complexes). 
Moreover we may assume that $\os{\circ}{S}=\{O\}$ 
since the family ${\cal X}/\ol{S}$ is locally 
obtained by the base change of a semistable family in the case $\os{\circ}{S}=\{O\}$. 
%We denote $t_1$ simply by $t$.
Since $\Om^{\bul}_{X/{\mab C}}(\log D)=
\Om^{\bul}_{{\cal X}/{\mab C}}(\log {\cal D})\otimes_{{\cal O}_{\cal X}}{\cal O}_X$, 
it suffices to prove that 
the natural morphism 
\begin{align*} 
\iota^{-1}(\Om^{\bul}_{{\cal X}/{\mab C}}(\log {\cal D}))
\lo \Om^{\bul}_{{\cal X}/{\mab C}}(\log {\cal D})
\otimes_{{\cal O}_{\cal X}}{\cal O}_X
\tag{6.1.2}\label{ali:hinxs}
\end{align*} 
is a quasi-isomorphism.  
\par 
Let $(z_1,\ldots,z_{m+1})$ $(\vert z_i\vert <1)$ 
be the local coordinate of a point of ${\cal X}$. 
We may assume that ${\cal X}=\Del^{m+1}$ and that 
the morphism $f\col {\cal X}\lo \Del$ is defined by the following equation 
$$f(z_1,\ldots,z_m)=
(z_1\cdots z_{a_1},z_{a_1+1}\cdots z_{a_2}, \ldots,z_{a_{r-1}+1}\cdots z_{a_r})
=(t_1,\ldots,t_r)\in \Del^{r}.$$  
We may also assume that ${\cal D}$ is a closed analytic space of 
$\os{\circ}{\cal X}$ defined by an ideal sheaf $(z_{a_r+1}\cdots z_b)$ 
$(a_r+1\leq b\leq m)$. 
\par 
To prove that the morphism (\ref{ali:hinxs}) is a quasi-isomorphism, 
it suffices to prove that 
the complex $\iota^{-1}(t_1\cdots t_r\Om^{\bul}_{{\cal X}/{\mab C}}(\log {\cal D}))$ is acyclic. 
(This is indeed a complex.)
Indeed, the stalk of $\iota^{-1}(t_1\cdots t_r\Om^{\bul}_{{\cal X}/{\mab C}}(\log {\cal D}))$ 
at $S$ is the Koszul complex on ${\mab C}\{z_0,\ldots,z_m\}$ with respect to 
the operators $z_i\dfrac{d}{dz_i}+1$ $(1\leq i\leq a_r)$, 
$z_i\dfrac{d}{dz_i}$ $(a_r+1\leq i\leq b)$ and $\dfrac{d}{dz_i}$ $(b+1\leq i\leq m)$.  
Since the operator $z_i\dfrac{d}{dz_i}+1$ $(1\leq i\leq a_r)$ is bijective, 
we see that 
$\iota^{-1}(t_1\cdots t_r\Om^{\bul}_{{\cal X}/{\mab C}}(\log {\cal D}))$ is acyclic by 
\cite[(1.12)]{sti}. 
\par  
We can complete the proof of (\ref{lemm:cn}). 
\end{proof}

\begin{coro}\label{coro:sqi}
The natural morphism 
\begin{align*} 
\iota^{-1}({\cal E}\otimes_{{\cal O}_{\cal X}}
\Om^{\bul}_{{\cal X}/\os{\circ}{S}}(\log {\cal D}))[U_S]
\lo 
{\cal E}\otimes_{{\cal O}_{\cal X}}
\Om^{\bul}_{X/\os{\circ}{S}}(\log {\cal D})[U_S]
\tag{6.2.1}\label{ali:munox}
\end{align*} 
is a quasi-isomorphism of complexes of $f^{-1}({\cal O}_S)$-modules. 
This quasi-isomorphism is contravariantly functorial. 
\end{coro}
\begin{proof}
In \cite[(3.14)]{nhi} we have proved that the (PD-)Hirsch extension preserves 
the quasi-isomorphism. Hence (\ref{coro:sqi}) follows from (\ref{lemm:cn}). 
\end{proof}

\begin{theo}\label{theo:tesc}
Let the notations be as in {\rm (\ref{lemm:cn})}. 
Then there exists a canonical section 
\begin{align*} 
\rho \col {\cal E}\otimes_{{\cal O}_{\cal X}}\Om^{\bul}_{X/S}(\log D)
\lo 
\iota^{-1}({\cal E}\otimes_{{\cal O}_{\cal X}}\Om^{\bul}_{{\cal X}/\ol{S}}(\log {\cal D}))
\tag{6.3.1}\label{ali:mnox}
\end{align*}
of complexes of $f^{-1}({\cal O}_S)$-modules
of the natural morphism 
\begin{align*}
\iota^{-1}({\cal E}\otimes_{{\cal O}_{\cal X}}\Om^{\bul}_{{\cal X}/\ol{S}}(\log {\cal D}))
\lo 
{\cal E}\otimes_{{\cal O}_{\cal X}}\Om^{\bul}_{X/S}(\log D)
\tag{6.3.2}\label{ali:mmoox}
\end{align*} 
in the derived category $D^+(\iota^{-1}f^{-1}({\cal O}_{\ol{S}}))$ of bounded 
below complexes of $\iota^{-1}f^{-1}({\cal O}_{\ol{S}})$-modules. 
The section {\rm (\ref{ali:mnox})} is contravariantly functorial with respect to 
a morphism 
$({\cal X},{\cal D})\lo ({\cal X}',{\cal D}')$ over $\ol{S}\lo \ol{S}{}'$, 
where $({\cal X}',{\cal D}')/\ol{S}{}'$ is an analogous log analytic space over 
${\mab C}$. 
\end{theo}
\begin{proof} 
By (\ref{theo:afp}), (\ref{prop:ys}) and (\ref{coro:sqi})
we obtain the following commutative diagram 
(cf.~\cite[(3.17)]{ey}): 
\begin{equation*} 
\begin{CD}
\iota^{-1}({\cal E}\otimes_{{\cal O}_{\cal X}}\Om^{\bul}_{{\cal X}/\os{\circ}{S}}(\log {\cal D}))[[U_S]]
@>>>\iota^{-1}
({\cal E}\otimes_{{\cal O}_{\cal X}}\Om^{\bul}_{{\cal X}/\ol{S}}(\log {\cal D}))\\
@A{\bigcup}AA @|\\
\iota^{-1}({\cal E}\otimes_{{\cal O}_{\cal X}}\Om^{\bul}_{{\cal X}/\os{\circ}{S}}
(\log {\cal D}))[U_S]
@.\iota^{-1}({\cal E}\otimes_{{\cal O}_{\cal X}}\Om^{\bul}_{{\cal X}/\ol{S}}(\log {\cal D}))\\
@V{\simeq}VV @|\\
{\cal E}\otimes_{{\cal O}_{\cal X}}\Om^{\bul}_{X/\os{\circ}{S}}(\log D)[U_S]
@.\iota^{-1}
({\cal E}\otimes_{{\cal O}_{\cal X}}\Om^{\bul}_{{\cal X}/\ol{S}}(\log {\cal D}))\\
@V{\simeq}VV @VVV\\
{\cal E}\otimes_{{\cal O}_{\cal X}}\Om^{\bul}_{X/S}(\log D)
@={\cal E}\otimes_{{\cal O}_{\cal X}}\Om^{\bul}_{X/S}(\log D). 
\end{CD}
\tag{6.3.3}\label{cd:eoxd}
\end{equation*} 
(By (\ref{prop:ys}) the horizontal morphism in (\ref{cd:eoxd}) 
is indeed an isomorphism.)
\par 
The contravariantly functoriality follows from (\ref{prop:rks}), 
(\ref{prop:ys}) and (\ref{coro:sqi}). 
\end{proof} 

\begin{rema}\label{rema:ptc}
(1) In \cite[(2.9)]{nf}  we have proved that 
the morphism 
\begin{align*}
{\cal E}\otimes_{{\cal O}_{\cal X}}\Om^{\bul}_{X/\os{\circ}{S}}(\log D)[U_S]
\lo 
{\cal E}\otimes_{{\cal O}_{\cal X}}\Om^{\bul}_{X/S}(\log D)
\end{align*} 
is an underlying isomorphism of filtered isomorphism. 
\par 
(2) Because the following natural morphism 
\begin{align*} 
\iota^{-1}({\cal E}\otimes_{{\cal O}_{\cal X}}\Om^{\bul \geq i}_{{\cal X}/\ol{S}}(\log {\cal D}))
\lo 
{\cal E}\otimes_{{\cal O}_{\cal X}}\Om^{\bul \geq i}_{X/S}(\log D)
\quad (i\in {\mab N})
\end{align*} 
is not a quasi-isomorphism in general,  
the morphism (\ref{ali:mnox}) cannot be an underlying morphism of 
a filtered morphism with respect to the Hodge filtration in general. 
\end{rema}

\begin{coro}\label{coro:tesc}
Let $\iota_S\col S\os{\sus}{\lo} \ol{S}$ be the natural inclusion. 
Set $f_S:=f\times_{\ol{S}}S$. 
Then there exists a canonical section 
\begin{align*} 
\rho \col \iota_{S*}Rf_{S*}({\cal E}\otimes_{{\cal O}_{\cal X}}\Om^{\bul}_{X/S}(\log D))
\lo 
Rf_*\iota_{*}\iota^{-1}
({\cal E}\otimes_{{\cal O}_{\cal X}}\Om^{\bul}_{{\cal X}/\ol{S}}(\log {\cal D}))
\tag{6.5.1}\label{ali:amnox}
\end{align*}
of complexes of $f^{-1}({\cal O}_S)$-modules 
of the natural morphism 
\begin{align*}
Rf_*\iota_{*}\iota^{-1}
({\cal E}\otimes_{{\cal O}_{\cal X}}\Om^{\bul}_{{\cal X}/\ol{S}}(\log {\cal D}))
\lo 
\iota_{S*}Rf_{S*}({\cal E}\otimes_{{\cal O}_{\cal X}}\Om^{\bul}_{X/S}(\log D)). 
\tag{6.5.2}\label{ali:ox}
\end{align*} 
The section {\rm (\ref{ali:amnox})} is contravariantly functorial with respect to 
a morphism 
$({\cal X},{\cal D})\lo ({\cal X}',{\cal D}')$ over $\ol{S}\lo \ol{S}{}'$, 
where $({\cal X}',{\cal D}')/\ol{S}{}'$ is an analogous log analytic space over 
${\mab C}$. 
\end{coro}

%For $({\cal X},{\cal D})$, we can reprove \cite[(6.4)]{ikn} when 
%$\os{\circ}{S}$ is a point. 

\begin{theo}\label{theo:lff}
Assume that $\os{\circ}{S}$ is smooth and that 
$\os{\circ}{f}\col \os{\circ}{\cal X}\lo \os{\circ}{\ol{S}}$ is proper. 
Then the higher direct image 
$R^qf_*({\cal E}\otimes_{{\cal O}_{\cal X}}\Om^{\bul}_{{\cal X}/\ol{S}}(\log {\cal D}))$ 
is a locally free ${\cal O}_{\ol{S}}$-module and commutes with base change. 
\end{theo}
\begin{proof}
By Grauert's theorem (\cite[p.~62 Satz 3]{grh}), 
we have only to prove that, for any point $y\in \ol{S}$, 
the natural morphism 
\begin{align*} 
R^qf_*({\cal E}\otimes_{{\cal O}_{\cal X}}\Om^{\bul}_{{\cal X}/\ol{S}}(\log {\cal D}))
\otimes_{{\cal O}_{\ol{S}}}\kap(y)
\lo 
H^q(X_y,{\cal E}\otimes_{{\cal O}_{\cal X}}\Om^{\bul}_{{\cal X}_y/y}(\log {\cal D}_y))
\tag{6.6.1}\label{ali:eexs}
\end{align*} 
is surjective. 
%Because we have the Gauss-Manin connection on 
%$R^qf_*({\cal E}\otimes_{{\cal O}_{\cal X}}\Om^{\bul}_{{\cal X}/\ol{S}}(\log {\cal D}))$ and 
%because $\os{\circ}{\ol{S}}$ is smooth, 
%we may assume that $M_y\not={\cal O}_y^*$ by \cite[(2.17)]{de}. 
%Furthermore we may assume that $y\in S$ because, if $y\notin S$, then 
%we have only to consider the smaller $r$. 
%In the following we assume that $y\in S$. 
We may assume that $\os{\circ}{y}$ is the origin $\os{\circ}{S}$. 
Indeed, if $\os{\circ}{y}$ is not the origin, 
we have only to consider the semistable family around $y$. 
%In this case, the morphism (\ref{ali:eexs}) factors through the natural morphism 
%\begin{align*} 
%R^qf_*({\cal E}\otimes_{{\cal O}_{\cal X}}\Om^{\bul}_{{\cal X}/\ol{S}}(\log {\cal D}))
%\otimes_{{\cal O}_{\ol{S}}}\kap(y)
%\lo 
%R^qf_{S*}({\cal E}\otimes_{{\cal O}_{\cal X}}\Om^{\bul}_{X/S}(\log D))
%\otimes_{{\cal O}_{S}}\kap(y). 
%\tag{6.6.2}\label{ali:eeoxs}
%\end{align*}  
%We claim that this morphism is surjective at $y\in S$. 
By (\ref{ali:mnox}) we see that the following natural morphism 
\begin{align*} 
R^qf_*
(\iota_*\iota^{-1}({\cal E}\otimes_{{\cal O}_{\cal X}}\Om^{\bul}_{{\cal X}/\ol{S}}(\log {\cal D})))
&\lo 
R^qf_*(\iota_*({\cal E}\otimes_{{\cal O}_{\cal X}}\Om^{\bul}_{X/S}(\log D)))\tag{6.6.2}\label{ali:mnmax}\\
=\iota_{S*}R^qf_{S*}({\cal E}\otimes_{{\cal O}_{\cal X}}\Om^{\bul}_{X/S}(\log D))
\end{align*} 
is surjective. 
%Hence we have the following morphism 
%\begin{align*} 
%&\iota_{S*}Rf_{S*}({\cal E}\otimes_{{\cal O}_{\cal X}}\Om^{\bul}_{X/S}(\log D))=
%Rf_*\iota_*({\cal E}\otimes_{{\cal O}_{\cal X}}\Om^{\bul}_{X/S}(\log D))
%\tag{6.5.5}\label{ali:manobx}\\
%&\lo 
%Rf_*\iota_*(\iota^{-1}({\cal E}\otimes_{{\cal O}_{\cal X}}\Om^{\bul}_{{\cal X}/\ol{S}}
%(\log {\cal D}))). 
%\end{align*} 
Set 
$${\cal K}^{\bul}:={\rm Ker}({\cal E}\otimes_{{\cal O}_{\cal X}}\Om^{\bul}_{{\cal X}/\ol{S}}(\log {\cal D})\lo 
\iota_*\iota^{-1}({\cal E}\otimes_{{\cal O}_{\cal X}}\Om^{\bul}_{{\cal X}/\ol{S}}(\log {\cal D})))
$$
(${\cal K}^{\bul}$ is equal to 
$j_{!}j^*({\cal E}\otimes_{{\cal O}_{\cal X}}\Om^{\bul}_{{\cal X}/\ol{S}}(\log {\cal D}))$, 
where $j\col {\cal X}\setminus X\os{\sus}{\lo}{\cal X}$ is the natural open immersion.)
Then we have the following exact sequence 
\begin{align*}
0\lo {\cal K}^{\bul}\lo {\cal E}\otimes_{{\cal O}_{\cal X}}\Om^{\bul}_{{\cal X}/\ol{S}}(\log {\cal D})\lo 
\iota_*\iota^{-1}({\cal E}\otimes_{{\cal O}_{\cal X}}\Om^{\bul}_{{\cal X}/\ol{S}}(\log {\cal D})))\lo 0
\end{align*} 
of $f^{-1}({\cal O}_{\ol{S}})$-modules. 
Hence we have the following exact sequence 
\begin{align*} 
\cdots &\lo R^qf_*({\cal K}^{\bul})\lo 
R^qf_*({\cal E}\otimes_{{\cal O}_{\cal X}}\Om^{\bul}_{{\cal X}/\ol{S}}(\log {\cal D}))\\
&\lo 
R^qf_*\iota_*\iota^{-1}({\cal E}\otimes_{{\cal O}_{\cal X}}\Om^{\bul}_{{\cal X}/\ol{S}}(\log {\cal D}))\lo R^{q+1}f_*({\cal K}^{\bul})\lo \cdots
\end{align*}
of $f^{-1}({\cal O}_{\ol{S}})$-modules. 
By the proper base change theorem and by the following cartesian diagrams: 
\begin{equation*} 
\begin{CD}
X_s@=X@>>>{\cal X}\\
@VVV @VVV @VVV\\
s@= S@>{\subset}>> \ol{S}, 
\end{CD}
\end{equation*} 
$R^qf_*({\cal K}^{\bul})_y=R^qf_{S*}(\iota^{-1}({\cal K}^{\bul}))_y$. 
Since $\iota^{-1}({\cal K}^{\bul})={\cal K}^{\bul}\vert_X=0$, 
$R^qf_*({\cal K}^{\bul})_y=0$. 
Hence 
\begin{align*} 
R^qf_*({\cal E}\otimes_{{\cal O}_{\cal X}}\Om^{\bul}_{{\cal X}/\ol{S}}(\log {\cal D}))_y
=
R^qf_*\iota_*\iota^{-1}
({\cal E}\otimes_{{\cal O}_{\cal X}}\Om^{\bul}_{{\cal X}/\ol{S}}(\log {\cal D}))_y.
\tag{6.6.3}\label{ali:eosexs}
\end{align*} 
%Since we have a natural morphism 
%\begin{align*} 
%Rf_*\iota_*(\iota^{-1}
%({\cal E}\otimes_{{\cal O}_{\cal X}}\Om^{\bul}_{{\cal X}/\ol{S}}(\log {\cal D})))\lo 
%Rf_*({\cal E}\otimes_{{\cal O}_{\cal X}}\Om^{\bul}_{{\cal X}/\ol{S}}(\log {\cal D})), 
%\end{align*} 
%this is a local problem. Moreover 
%\begin{align*} 
%R^qf_*({\cal E}\otimes_{{\cal O}_{\cal X}}\Om^{\bul}_{{\cal X}/\ol{S}}(\log {\cal D}))
%=R^qf_{*}\iota_*(\iota^{-1}
%({\cal E}\otimes_{{\cal O}_{\cal X}}\Om^{\bul}_{{\cal X}/\ol{S}}(\log {\cal D}))) 
%\quad (q\in {\mab N}). 
%\end{align*} 
%We may assume that $M_S/{\cal O}_S\simeq {\mab N}^r$ and 
%that there exists the immersion $S\os{\sus}{\lo} \Del_S=\Del^r\times \os{\circ}{S}$.  
%Because the homotopy type of ${\cal X}\setminus {\cal D}$ is invariant even if 
%we shrink ${\cal X}\setminus {\cal D}$ with respect to the direction of $\Del^r$, 
%we have the following equality as in the proof of \cite[(2.5)]{sti}: 
%\begin{align*} 
%R^qf_*({\cal E}\otimes_{{\cal O}_{\cal X}}\Om^{\bul}_{{\cal X}/\ol{S}}(\log {\cal D}))
%&=
%\vil_{S\subset U\subset \ol{S}}
%R^qf_*j_U(
%{\cal E}\otimes_{{\cal O}_{\cal X}}\Om^{\bul}_{{\cal X}/\ol{S}}(\log {\cal D})\vert_{f^{-1}(U)})\\
%&=
%\vil_{X\subset V\subset {\cal X}}
%R^qf_{*}j_{V*}
%({\cal E}\otimes_{{\cal O}_{\cal X}}\Om^{\bul}_{{\cal X}/\ol{S}}(\log {\cal D})\vert_{V})\\
%&= 
%R^qf_{*}\iota_*
%\iota^{-1}
%({\cal E}\otimes_{{\cal O}_{\cal X}}\Om^{\bul}_{{\cal X}/\ol{S}}(\log {\cal D})). 
%\end{align*} 
%Here $j_U \col f^{-1}(U)\os{\sus}{\lo} {\cal X}$ (resp.~$j_V\col V\os{\sus}{\lo}{\cal X}$) 
%is the open immersion.  
Consequently the morphism 
\begin{align*}
R^qf_*({\cal E}\otimes_{{\cal O}_{\cal X}}\Om^{\bul}_{{\cal X}/\ol{S}}(\log {\cal D}))_y 
\lo 
\iota_{S*}Rf_{S*}({\cal E}\otimes_{{\cal O}_{\cal X}}\Om^{\bul}_{X/S}(\log D))_y
=Rf_{S*}({\cal E}\otimes_{{\cal O}_{\cal X}}\Om^{\bul}_{X/S}(\log D))_y
\end{align*} 
is surjective. 
%Now we have only to prove that the morphism 
%\begin{align*} 
%R^qf_{S*}({\cal E}\otimes_{{\cal O}_{\cal X}}\Om^{\bul}_{X/S}(\log D))
%\otimes_{{\cal O}_S}\kap(y) \lo 
%H^q(X_y,{\cal E}\otimes_{{\cal O}_{\cal X}}\Om^{\bul}_{X_y/y}(\log D_y))
%\tag{6.6.5}\label{ali:eoexs}
%\end{align*} 
%is surjective. 
%In (\ref{coro:lfs}) we have already proved this surjectivity.  
\par
We complete the proof of this theorem. 
\end{proof} 

\begin{rema}
(1) Though (\ref{theo:lff}) is a special case of (\ref{theo:kikn}), 
our proof of (\ref{theo:lff}) is purely algebraic. 
\par 
(2) 
Let $\pi \col Y\lo Z$ be a proper morphism of complex analytic spaces. 
Assume that $Z$ is reduced. Usually Grauert's theorem has been stated 
for a $\pi$-flat coherent sheaf. However the proof for this theorem 
is valid for a bounded complex whose any component 
is a coherent ${\cal O}_Y$-module with 
$\pi^{-1}({\cal O}_Y)$-linear boundary morphism. 
\end{rema}

%\begin{coro}\label{coro:amm}
%{\rm (\ref{theo:rpi})} holds.
%\end{coro}
%\begin{proof}
%\end{proof} 

We also consider the case where $\os{\circ}{S}$ is not necessarily a point. 

\begin{coro}\label{coro:nsii}
Assume that $\os{\circ}{S}$ is smooth. 
Then the higher direct image 
$R^qf_*({\cal E}\otimes_{{\cal O}_{\cal X}}\Om^{\bul}_{{\cal X}/\ol{S}}(\log {\cal D}))$ 
is a locally free ${\cal O}_{\ol{S}}$-module and commutes with base change. 
\end{coro}
\begin{proof} 
Take any point $y\in \ol{S}$. Then we have only to consider the semistable family around 
$y$ and we may assume that $y$ is the origin of this neighborhood of $y$. 
\end{proof} 

\begin{coro}\label{coro:ls}
Assume that $\os{\circ}{S}$ is smooth. 
Then $R^qf_*({\cal E}\otimes_{{\cal O}_{\cal X}}\Om^{\bul}_{X/S}(\log {\cal D}))$ 
is a locally free ${\cal O}_{\ol{S}}$-module and commutes with base change. 
\end{coro} 

\begin{rema}\label{rema:nnta} 
In (\ref{coro:nsii}) and (\ref{coro:ls}) 
it is not necessary to assume that $\os{\circ}{S}$ is smooth 
by \cite[(6.4)]{ikn}.  
\end{rema}

In the following we do not assume that $\os{\circ}{S}$ is smooth. 

Let 
\begin{align*} 
\sig_{({\cal X},{\cal D})/\ol{S}} \col 
{\cal E}\otimes_{{\cal O}_{\cal X}}\Om^{\bul}_{{\cal X}/\os{\circ}{S}}(\log {\cal D})[U_S]
\lo 
{\cal E}\otimes_{{\cal O}_{\cal X}}\Om^{\bul}_{{\cal X}/\ol{S}}(\log {\cal D})
\tag{6.10.1}\label{ali:xxd} 
\end{align*} 
be the induced morphism by the morphism 
\begin{align*} 
{\cal E}\otimes_{{\cal O}_{\cal X}}\Om^{\bul}_{{\cal X}/\os{\circ}{S}}(\log {\cal D})
\os{{\rm proj}.}{\lo}  
{\cal E}\otimes_{{\cal O}_{\cal X}}\Om^{\bul}_{{\cal X}/\ol{S}}(\log {\cal D})
\end{align*} 
and the morphism (\ref{ali:osls}).  
The morphism $\sig_{({\cal X},{\cal D})/\ol{S}}$ is 
a morphism of complexes since ``$d\log t_i=0=dt_i$'' in  
$\Om^1_{{\cal X}/\ol{S}}(\log {\cal D})$. 

\begin{defi} 
We call the following morphism 
\begin{align*} 
{\cal E}\otimes_{{\cal O}_{\cal X}}\Om^{\bul}_{X/S}(\log D)
&\os{\sim}{\longleftarrow} 
{\cal E}\otimes_{{\cal O}_{\cal X}}\Om^{\bul}_{X/\os{\circ}{S}}(\log D)[U_S]
\os{\sim}{\longleftarrow}  
\iota^{-1}({\cal E}\otimes_{{\cal O}_{\cal X}}\Om^{\bul}_{{\cal X}/\os{\circ}{S}}(\log {\cal D}))[U_S]\tag{6.11.1}\label{cd:eaoxd}\\
&
\os{\iota^{-1}(\sig_{({\cal X},{\cal D})/\ol{S}})}{\lo} 
\iota^{-1}({\cal E}\otimes_{{\cal O}_{\cal X}}\Om^{\bul}_{{\cal X}/\ol{S}}(\log {\cal D}))
\end{align*} 
the $\infty${\it -adic Yamada section}. 
We denote this section by $\iota^{-1}(\sig_{({\cal X},{\cal D})/\ol{S}})$ by abuse of notation. 
\end{defi} 
\parno
The $\infty$-adic Yamada section is indeed a section of the following natural morphism 
\begin{align*}
\iota^{-1}({\cal E}\otimes_{{\cal O}_{\cal X}}\Om^{\bul}_{{\cal X}/\ol{S}}(\log {\cal D}))
\lo 
{\cal E}\otimes_{{\cal O}_{\cal X}}\Om^{\bul}_{X/S}(\log D) 
\end{align*}
because the following diagram is commutative: 
\begin{equation*} 
\begin{CD}
\iota^{-1}({\cal E}\otimes_{{\cal O}_{\cal X}}\Om^{\bul}_{{\cal X}/\os{\circ}{S}}
(\log {\cal D}))[U_S]
@>{\iota^{-1}(\sig_{({\cal X},{\cal D})/\ol{S}})}>>\iota^{-1}
({\cal E}\otimes_{{\cal O}_{\cal X}}\Om^{\bul}_{{\cal X}/\ol{S}}(\log {\cal D}))\\
@V{\simeq}VV @VVV\\
{\cal E}\otimes_{{\cal O}_{\cal X}}\Om^{\bul}_{X/\os{\circ}{S}}(\log D)[U_S]
@.{\cal E}\otimes_{{\cal O}_{\cal X}}\Om^{\bul}_{X/S}(\log D).\\
@V{\simeq}VV @|\\
{\cal E}\otimes_{{\cal O}_{\cal X}}\Om^{\bul}_{X/S}(\log D)
@={\cal E}\otimes_{{\cal O}_{\cal X}}\Om^{\bul}_{X/S}(\log D). 
\end{CD}
\tag{6.11.2}\label{cd:eoaxd}
\end{equation*} 
This commutative diagram is the crux in this article.
The $\infty$-adic Yamada section induces the following morphism 
\begin{align*} 
&\iota_{S*}Rf_{S*}({\cal E}\otimes_{{\cal O}_{\cal X}}\Om^{\bul}_{X/S}(\log D))
\os{\sim}{\longleftarrow} 
\iota_{S*}Rf_{S*}({\cal E}\otimes_{{\cal O}_{\cal X}}\Om^{\bul}_{X/\os{\circ}{S}}(\log D)[U_S])\tag{6.11.3}\label{cd:eboxd}\\
&\os{\sim}{\longleftarrow}  
Rf_*(\iota_*\iota^{-1}
({\cal E}\otimes_{{\cal O}_{\cal X}}\Om^{\bul}_{{\cal X}/\os{\circ}{S}}(\log {\cal D}))[U_S])
\os{\iota^{-1}(\sig_{({\cal X},{\cal D})/\ol{S}})}{\lo} 
Rf_*(\iota_*\iota^{-1}
({\cal E}\otimes_{{\cal O}_{\cal X}}\Om^{\bul}_{{\cal X}/\ol{S}}(\log {\cal D}))). 
\end{align*} 
Let $q$ be a nonnegative integer. 
Obviously the morphism (\ref{cd:eboxd}) induces the following isomorphism of 
${\cal O}_S$-modues: 
\begin{align*} 
&\iota_{S*}R^qf_{S*}({\cal E}\otimes_{{\cal O}_{\cal X}}\Om^{\bul}_{X/S}(\log D))
\os{\sim}{\longleftarrow} 
R^qf_*(\iota_*\iota^{-1}
({\cal E}\otimes_{{\cal O}_{\cal X}}\Om^{\bul}_{{\cal X}/\ol{S}}(\log {\cal D}))). \tag{6.11.4}\label{cd:ekoxd}\\
\end{align*} 
Let $y$ be any exact point of $S$. 
Because $R^qf_{S*}({\cal E}\otimes_{{\cal O}_{\cal X}}\Om^{\bul}_{X/S}(\log D))$ is 
a coherent ${\cal O}_S$-module, there exists an open neighbourhood $V(y)$ of 
$y$ in $\ol{S}$ such that the restriction of the isomorphism (\ref{cd:ekoxd}) to $V(y)$ 
factors through the natural morphism  
\begin{align*} 
R^qf_*({\cal E}\otimes_{{\cal O}_{\cal X}}\Om^{\bul}_{{\cal X}/\ol{S}}
(\log {\cal D}))\vert_{V(y)}
\lo R^qf_*\iota_*\iota^{-1}
({\cal E}\otimes_{{\cal O}_{\cal X}}\Om^{\bul}_{{\cal X}/\ol{S}}(\log {\cal D}))\vert_{V(y)}
\tag{6.11.5}\label{ali:eosmexs}
\end{align*} 
of ${\cal O}_S$-modules by (\ref{ali:eosexs}). 
Consequently we have the following injective morphism 
\begin{align*} 
&\iota_{S*}R^qf_{S*}({\cal E}\otimes_{{\cal O}_{\cal X}}\Om^{\bul}_{X/S}(\log D))
\vert_{V(y)}
\os{\subset}{\lo} 
R^qf_*({\cal E}\otimes_{{\cal O}_{\cal X}}\Om^{\bul}_{{\cal X}/\ol{S}}
(\log {\cal D}))\vert_{V(y)}. 
\tag{6.11.6}\label{cd:ekomxd}
\end{align*} 
Assume that $\os{\circ}{S}$ is quasi-compact. 
If $\eps>0$ is small enough, then we have a natural immersion 
$\ol{S}(\eps)\os{\sus}{\lo} \ol{S}$ ((\ref{ali:ollss})). 
%\begin{align*} 
%R^qf_*({\cal E}\otimes_{{\cal O}_{\cal X}}\Om^{\bul}_{{\cal X}/\ol{S}}(\log {\cal D}))\vert_V
%=R^qf_*\iota_*\iota^{-1}
%({\cal E}\otimes_{{\cal O}_{\cal X}}\Om^{\bul}_{{\cal X}/\ol{S}}(\log {\cal D}))\vert_V.
%\tag{6.11.5}\label{ali:eoxs}
%\end{align*} 
Because $\os{\circ}{S}$ is quasi-compact, there exists an 
open neighbourhood $V$ of $S$ in $\ol{S}$ 
such that the morphism restricted to $V\cap V(y)$ 
extends to the following morphism 
\begin{align*} 
\iota_{V*}R^qf_{S*}({\cal E}\otimes_{{\cal O}_{\cal X}}\Om^{\bul}_{X/S}(\log D))
\lo 
R^qf_*({\cal E}\otimes_{{\cal O}_{\cal X}}\Om^{\bul}_{{\cal X}/\ol{S}}(\log {\cal D}))\vert_V
\tag{6.11.7}\label{ali:amnomx}
\end{align*}
of ${\cal O}_S$-modules. Here $\iota_V\col S\os{\sus}{\lo} V$ is the inclusion morphism. 
Let $p_V \col V\lo \os{\circ}{S}$ be the projection. 
Applying $p_{V*}$ to (\ref{ali:amnomx}), we have the following morphism 
\begin{align*} 
R^qf_{S*}({\cal E}\otimes_{{\cal O}_{\cal X}}\Om^{\bul}_{X/S}(\log D)) 
\lo 
p_{V*}(R^qf_*({\cal E}\otimes_{{\cal O}_{\cal X}}\Om^{\bul}_{{\cal X}/\ol{S}}(\log {\cal D}))\vert_V)
\tag{6.11.8}\label{ali:amoomx}
\end{align*}
(Because $\ol{S}(\eps)\lo S$ for $\eps \lo 0$, we may assume that 
$\ol{S}(\eps)\subset V$.)  
Hence we have the following morphism 
\begin{align*} 
p^{-1}_{V}R^qf_{S*}({\cal E}\otimes_{{\cal O}_{\cal X}}\Om^{\bul}_{X/S}(\log D))
\lo 
R^qf_*({\cal E}\otimes_{{\cal O}_{\cal X}}\Om^{\bul}_{{\cal X}/\ol{S}}(\log {\cal D}))_V. 
\end{align*} 
of $p_V^{-1}({\cal O}_S)$-modules. 
Consequently we have the following morphism 
\begin{align*}  
\sig_V:=\sig(({\cal X},{\cal D})/\ol{S})_V \col 
p^*_{V}R^qf_{S*}({\cal E}\otimes_{{\cal O}_{\cal X}}\Om^{\bul}_{X/S}(\log D))
\lo 
R^qf_*({\cal E}\otimes_{{\cal O}_{\cal X}}\Om^{\bul}_{{\cal X}/\ol{S}}(\log {\cal D}))_V 
\tag{6.11.9}\label{cd:aazxd}
\end{align*} 
of ${\cal O}_V$-modules.

The following is the main result in this article:

\begin{theo}[\bf Invariance theorem]\label{theo:nm}
If $V$ is small enough $($containing $S)$, 
then the morphism $\sig_V$ is an isomorphism.  
%For any small $\eps>0$, 
%there exists a canonical contravariantly functorial morphism 
%\begin{align*} 
%\sig(\eps) \col  
%p(\eps)^*(R^qf_{S*}({\cal E}\otimes_{{\cal O}_{\cal X}}\Om^{\bul}_{X/S}(\log D)))
%\os{\sim}{\lo}  
%R^qf(\eps)_*({\cal E}\otimes_{{\cal O}_{\cal X}}\Om^{\bul}_{{\cal X}(\eps)/\ol{S}(\eps)}
%(\log {\cal D}(\eps))).
%\tag{6.11.1}\label{ali:oexas}
%\end{align*} 
%The isomorphisms $\sig(\eps)$'s for small $\eps$'s are compatible in an obvious sense.  
\end{theo}
\begin{proof} 
By (\ref{rema:nnta}) and (\ref{theo:lff}) 
the source and the target of the morphism (\ref{cd:aazxd}) 
are locally free ${\cal O}_{\ol{S}}$-modules and commutes with base change. 
Hence the morphism (\ref{cd:aazxd})
induces the following isomorphism 
\begin{align*} 
{\mab C}\otimes_{{\mab C}}
R^qf_{s*}({\cal E}\otimes_{{\cal O}_{\cal X}}\Om^{\bul}_{{\cal X}_s/s}(\log {\cal D}_s))
=
R^qf_{s*}({\cal E}\otimes_{{\cal O}_{\cal X}}\Om^{\bul}_{{\cal X}_s/s}(\log {\cal D}_s))
\tag{6.12.1}\label{ali:oexsa}
\end{align*}
at any point $s\in S$, where $f_s\col {\cal X}_s\lo s$ is the structural morphism. 
%Take an open log analytic space $W$ of $\ol{S}$ such that $W\cap S\not=\emptyset$. 
Take an open neighborhood $W$ of $s$ in $\ol{S}$ such that 
$R^qf_*({\cal E}\otimes_{{\cal O}_{\cal X}}\Om^{\bul}_{{\cal X}/\ol{S}}(\log {\cal D}))\vert_W$ 
and 
$R^qf_{S*}({\cal E}\otimes_{{\cal O}_{\cal X}}\Om^{\bul}_{X/S}(\log D))\vert_{W\cap S}$ 
are free. Let $\{v_i\}_{i=1}^d$ and $\{w_{i}\}_{i=1}^d$ be bases of 
$R^qf_*({\cal E}\otimes_{{\cal O}_{\cal X}}\Om^{\bul}_{{\cal X}/\ol{S}}(\log {\cal D}))\vert_W$ 
and   
$R^qf_{S*}({\cal E}\otimes_{{\cal O}_{\cal X}}\Om^{\bul}_{X/S}(\log D))\vert_{W\cap S}$, 
respectively.  
Let  $\ol{\mathfrak m}_s$ be the maximal ideal of ${\cal O}_{\ol{S}}$ at $s$. 
%Set $\sig(\eps):=\sig_{\ol{S}(\eps)}$. 
Then 
\begin{align*} 
v_{j}~{\rm mod}~\ol{\mathfrak m}_s
=\sum_{i=1}^d(f_{ij}~{\rm mod}~\ol{\mathfrak m}_s)
(\sig\vert_{V\cap \ol{S}(\eps)}(p^{-1}(w_i))~{\rm mod}~\ol{\mathfrak m}_s)
\end{align*}
for some local sections $f_{ij}$'s of ${\cal O}_{\ol{S}}$. 
By Nakayama's lemma,  
$v_{j}=\sum_{i=1}^df_{ij}\sig(\eps)(p^{-1}(w_i))$ 
for some local sections $f_{ij}$'s of ${\cal O}_{\ol{S}}$.  
Hence there exist an open log analytic space $V_1$ and a set 
$\{f_{ij}\}_{i,j=1}^r$ of local sections of $\Gam(V_1,{\cal O}_{\ol{S}})$ 
such that $v_j=\sum_{i=1}^rf_{ij}\sig(p^{-1}(w_i))$ on $V_1$. 
This means that the morphism 
\begin{align*} 
\Gam(V_1,{\cal O}_{\ol{S}}\otimes_{p^{-1}({\cal O}_S)}
p^{-1}R^qf_*({\cal E}\otimes_{{\cal O}_{\cal X}}\Om^{\bul}_{X/S}(\log D)))
\lo 
\Gam(V_1,
R^qf_*({\cal E}\otimes_{{\cal O}_{\cal X}}\Om^{\bul}_{{\cal X}(\eps)/\ol{S}(\eps)}
(\log {\cal D}(\eps))))
\tag{6.12.2}\label{ali:omexs}
\end{align*} 
is surjective. 
\par 
Because the source and the target of the morphism (\ref{ali:omexs}) 
are free $\Gam(V_1,{\cal O}_{\ol{S}})$-modules of 
the same rank, the morphism (\ref{ali:omexs}) is injective by  
\cite[II \S3 Corollaire to Th\'{e}oreme 1]{bou} and 
[loc.~cit., II \S3 Corollaire to Proposition 6]. 
Because $\os{\circ}{S}$ is quasi-compact, 
we have only to consider a finitely many open covering $W$'s. 
%\par 
%(2): (2) follows from (1). 
%Let $w$ be an element  of $\Gam(V_1\cap S,
%p^{-1}R^qf_*({\cal E}\otimes_{{\cal O}_{\cal X}}\Om^{\bul}_{X/S}(\log D)))$ 
%such that $\sig(w)=0$. Then $\sig(w)~{\rm mod}~\ol{\mathfrak m}^n_s=0$ 
%for any $n\in{\mab Z}_{\geq 1}$. 
%Because $p^{-1}({\cal O}_S)\lo {\cal O}_{\ol{S}}$ is faithfully flat and 
%because 
%$R^qf_*({\cal E}\otimes_{{\cal O}_{\cal X}}\Om^{\bul}_{X/S}(\log D)$ 
%commutes with base change, 
%this vanishing means that $w~{\rm mod}~{\mathfrak m}^n_s=0$, where 
%${\mathfrak m}_s$ is the maximal ideal of ${\cal O}_{S,s}$. 
%Since $s$ is any point of $V_1\cap S$, 
%we see that $w=0$ as in the proof of (\ref{theo:qq}). 
%This means that the morphism (\ref{ali:omexs}) is injective. 
\end{proof}

%We have proved the following in the proof of the theorem above. 
%\begin{coro}\label{coro:cb}
%If $\eps>0$ is sufficiently small, then there exists the following canonical isomorphism
%\begin{align*} 
%{\cal O}_{\ol{S}}\otimes_{p^{-1}({\cal O}_S)}
%p^{-1}R^qf_*({\cal E}\otimes_{{\cal O}_{\cal X}}\Om^{\bul}_{X/S}(\log D))
%\os{\sim}{\lo} 
%R^qf_*({\cal E}\otimes_{{\cal O}_{\cal X}}\Om^{\bul}_{{\cal X}(\eps)/\ol{S}(\eps)}
%(\log {\cal D}(\eps)))) \quad (q\in {\mab N}). 
%\end{align*} 
%s\end{coro}

\begin{exem}
Let $f\col {\cal E} \lo \Del$ be a log elliptic curve (Tate curve) 
over $\Del$ described in \cite[0.2.10]{ku} 
defined as follows. (Strictly speaking, the definition in [loc.~cit.] is not perfect.) 
\par  
Set $X=\{(t_1,t_2)\in {\mab C}^2~\vert~\vert t_1t_2\vert <1\}$ 
and let the equivalence relation $\sim$ be as follows: 
$(t_1,t_2)\sim (t'_1,t'_2) \us{\rm def}{\Longleftrightarrow} 
t_1t_2=t'_1t'_2$. Let $\pi \col X\lo X/\sim$ be the projection 
and let $g\col X/\sim \owns [(t_1,t_2)]\lom t_1t_2\in \Del$.  
We define ${\cal E}$ as the quotient of $X/\sim$ by 
the following equivalence relation 
$\approx$: 
for $q=t_1t_2=t_1't_2'\in \Del^*$,  
$[(t_1,t_2)]\approx [(t'_1,t'_2)]\us{\rm def}{\Longleftrightarrow} 
t_1'=q^nt_1$ and $t_2'=q^{-n}t_2$ for some $n\in {\mab Z}$; 
for $q=0$, $[(t_1,t_2)]\approx [(t'_1,t'_2)]
\us{\rm def}{\Longleftrightarrow}$ 
if $t_1=0$, take $t_2\not=0$. 
Then $(t'_1,t'_2)=(0,t_2)$ or $(t_2^{-1},0)$. 
If $t_2=0$, take $t_1\not=0$. Then $(t'_1,t'_2)=(t_1,0)$ or $(0,t_1^{-1})$. 
Let $f\col {\cal E}\lo \Del$ be the induced morphism by $g$. 
Let $O$ be the origin of $\Del$. 
Then $f^{-1}(q)={\mab C}^*/q^{\mab Z}$ and 
$f^{-1}(O)={\mab P}^1_{\mab C}({\mab C})/(0\sim {\inf})$. 
Let $E$ be the log special fiber of ${\cal E}$ at $O$.
\par 
%for ${\cal O}_{\Del}\owns \al \not=0$. 
%However we have not proved this.  
%a section 
%$$R^1f_*(\Om^{\bul}_{E_s/s})\os{\sus}{\lo} R^1f_*(\Om^{\bul}_{E/\Del})$$ 
%of the projection $R^1f_*(\Om^{\bul}_{E/\Del})\lo 
%R^1f_*(\Om^{\bul}_{E_s/s})$.
%Obviously 
%an isomorphism 
%$$R^1f_*(\Om^{\bul}_{E/\Del})={\cal O}_{\Del}\otimes_{\mab C}
%H^1_{\rm dR}(E_s/s).$$
Consider the map $h\col X\owns (t_1,t_2)\lom \vert t_1t_2\vert \in {\mab R}$.
Set $W_1(q):=(\vert q\vert^{\frac{3}{2}}-\eps, \vert q\vert +\eps)$ 
and $W_2(q):=(\vert q^2\vert-\eps,\vert q\vert^{\frac{3}{2}}+\eps)$ be 
intervals in ${\mab R}$ for a small $\eps>0$. 
Set $W_1:=\cup_{\vert q\vert <1}h^{-1}(W_1(q))$ and 
$W_2:=\cup_{\vert q\vert <1}h^{-1}(W_2(q))$. 
Let $U_i$ $(i=1,2)$ be the image of $W_i$ in $E$. 
Then we obtain an open covering $E=U_1\cup U_2$. 
If $\eps >0$ is small, then $U_1\cap U_2=V_1\coprod V_2$ for 
two nonempty open sets $V_1$ and $V_2$. 
%Consider the affine analytic spaces
%$$U_1:={\rm Spec}({\mab C}[t,\dfrac{z}{t},\dfrac{q}{z}]):=
%{\rm Spec}({\mab C}[t,x_1,y_1]/(x_1y_1-t))$$ 
%and 
%$$U_2={\rm Spec}({\mab C}[t,z,\dfrac{t}{z}])
%:={\rm Spec}({\mab C}[t,x_2,y_2]/(x_2y_2-t)).$$ 
%Consider the affine scheme $\{\vert z \vert=\vert t \vert\}$ in $U$ and $V$, respectively:
%$$V_1:={\rm Spec}({\mab C}[t,(\dfrac{z}{t})^{\pm 1}]):=
%{\rm Spec}({\mab C}[t,x_1^{\pm 1}]),$$  
%$$V'_1:={\rm Spec}({\mab C}[t,(\dfrac{t}{z})^{\pm 1}]):=
%{\rm Spec}({\mab C}[t,y_2^{\pm 1}]).$$  
%Patch 
%$V_1$ and  $V'_1$ by $z\lom z$. 
%Furthermore, 
%consider also the affine scheme $\{\vert z \vert=\vert q \vert\}\simeq 
%\{\vert z \vert=\vert 1 \vert\}$ in $U$ and $V$, respectively:
%$$V_2:={\rm Spec}({\mab C}[t,(\dfrac{t}{z})^{\pm 1}]):=
%{\rm Spec}({\mab C}[t,y_1^{\pm 1}]),$$  
%$$V'_2:={\rm Spec}({\mab C}[t,z^{\pm 1}]):=
%{\rm Spec}({\mab C}[t,x_2^{\pm 1}]),$$  
%Patch 
%$V_2$ and  $V'_2$ by $z\lom qz$. 
The notation $(0)$ stands for the specialization $q=0$. 
\par 
The cohomology $H^1_{\rm dR}(E/s)$ is calculated by 
the $H^1$ of the following double complex:  
\begin{equation*} 
\begin{CD}
\bigoplus_{j=1}^2\bigoplus_{i=0}^{\infty}\Gam(U_j(0),\Om^2_{U_j(0)/{\mab C}}u^{[i]})
@>{\partial}>>
\bigoplus_{j=1}^2\bigoplus_{i=0}^{\infty}\Gam(V_j(0),\Om^2_{V_j(0)/{\mab C}}u^{[i]})\\
@A{d}AA @AA{-d}A \\
\bigoplus_{j=1}^2\bigoplus_{i=0}^{\infty}\Gam(U_j(0),\Om^1_{U_j(0)/{\mab C}}u^{[i]})@>{\partial}>>
\bigoplus_{j=1}^2\bigoplus_{i=0}^{\infty}\Gam(V_j(0),\Om^1_{V_j(0)/{\mab C}}u^{[i]})\\
@A{d}AA @AA{-d}A \\
\bigoplus_{j=1}^2\bigoplus_{i=0}^{\infty}\Gam(U_j(0),{\cal O}_{U_j(0)}u^{[i]})
@>{\partial}>>\bigoplus_{j=1}^2\bigoplus_{i=0}^{\infty}\Gam(V_j(0),{\cal O}_{V_j(0)}u^{[i]}). 
\end{CD}
\end{equation*} 
Here $\partial$'s are the usual boundary morphsims in \v{C}ech cohomologies. 
Let $\ol{\om}(0)$ be the cohomology class of $(0,1)
\in \Gam(V_1(0),{\cal O}_{V_1(0)})\oplus 
\Gam(V_2(0),{\cal O}_{V_2(0)})$. 
Let $\om(0)$ be the cohomology class of $(d\log t_1,d\log t_2)\oplus(-u,u)
\in \bigoplus_{j=1}^2\Gam(U_j(0),\Om^1_{U_j(0)/s})\oplus 
\bigoplus_{j=1}^2\Gam(V_j(0)u,{\cal O}_{V_j(0)}u)$.
\par 
The cohomology sheaf $H^1_{\rm dR}({\cal E}/\Del)$ 
is calculated by the $H^1$ of the following double complex:  
\begin{equation*} 
\begin{CD}
%\bigoplus_{j=1}^2\bigoplus_{i=0}^{\infty}\Gam(U_j,\Om^2_{U_j/\Del})
%@>{\partial}>>
%\bigoplus_{j=1}^2\bigoplus_{i=0}^{\infty}\Gam(V_j,\Om^2_{V_j/\Del})\\
%@A{d}AA @AA{-d}A \\
\bigoplus_{j=1}^2\Gam(U_j,\Om^1_{U_j/\Del})@>{\partial}>>
\bigoplus_{j=1}^2\Gam(V_j,\Om^1_{V_j/\Del})\\
@A{d}AA @AA{-d}A \\
\bigoplus_{j=1}^2\bigoplus_{i=0}^{\infty}\Gam(U_j,{\cal O}_{U_j})
@>{\partial}>>\bigoplus_{j=1}^2\Gam(V_j,{\cal O}_{V_j}). 
\end{CD}
\end{equation*} 
Let $\ol{\om}$ be the cohomology class of 
$(0,1)\in \Gam(V_1,{\cal O}_{V_1})\oplus \Gam(V_2,{\cal O}_{V_2})$ 
and let $\om$ be the cohomology class of $(d\log t_1,d\log t_2)
\in \bigoplus_{j=1}^2\Gam(V_j,\Om^1_{V_j/\Del})$.  
Then the morphism 
\begin{align*} 
\sig_{\Del(\eps)} \col H^1_{\rm dR}(E/s)\lo H^1_{\rm dR}({\cal E}/\Del)
\end{align*} is given by the following: 
\begin{align*} 
\sig_{\Del(\eps)}(\ol{\om}(0))=\ol{\om}, \quad \sig_{\Del(\eps)}(\om(0))=\om+q\ol{\om}.
\end{align*}   
In particular, $\sig_{\Del(\eps)}$ is not a filtered morphism on 
$\Del^*=\{q\in {\mab C}~\vert~0<\vert q\vert <1\}$. 
However the following inclusion holds:
\begin{align*} 
\sig_{\Del(\eps)}(F^i_HH^1_{\rm dR}(E/s))\subset 
\sum_{j=0}^iq^{i-j}F^i_HH^1_{\rm dR}({\cal E}/\Del) \quad (i,j\in {\mab N}),  
\end{align*}
where $F_H$ means the Hodge filtration. 
\end{exem}

\begin{rema}
%(1) 
See \cite[\S5]{ey} for the Hyodo-Kato section between the rigid cohomology of 
a Tate curve over a complete discrete valuation ring of mixed characteristics. 
%\par 
%(2) 
\end{rema}

As in \cite[(10.2.4)]{nhi} consider the morphism 
\begin{equation*} 
d/du_k: {\cal O}_S[\ul{u}] \owns u^{[i]}_k \lom u^{[i-1]}_k\in {\cal O}_S[\ul{u}] 
\quad (1\leq k\leq r). 
\tag{6.14.1}\label{ali:eyufnt}
\end{equation*}
Because this morphism is independent of the choice of $\ul{u}$, we have the following morphism 
\begin{equation*}
D_k \col \Gam_{{\cal O}_T}(L_S)\lo \Gam_{{\cal O}_T}(L_S). 
\tag{6.14.2}\label{eqn:bufnt}
\end{equation*}
This morphism induces the following morphism 
\begin{equation*} 
D_k: {\cal E}\otimes_{{\cal O}_{\cal X}}\Om^{\bul}_{X/\os{\circ}{S}}(\log D)[U_S]
\lo {\cal E}\otimes_{{\cal O}_{\cal X}}\Om^{\bul}_{X/\os{\circ}{S}}(\log D)[U_S]. 
\tag{6.14.3}\label{ali:eufant}
\end{equation*}
\par 
Let the notations be as in the beginning of {\rm \S\ref{sec:da}}. 
Set $S^k:=(\os{\circ}{S},(\oplus_{j\not =k}M_j\lo {\cal O}_S)^a)$. 
Then we have the following exact sequence 
\begin{align*} 
0\lo {\cal E}\otimes_{{\cal O}_{\cal X}}\Om^{\bul}_{X/S}(\log D)[-1]
\os{d\log t_k\wedge }{\lo} 
{\cal E}\otimes_{{\cal O}_{\cal X}}\Om^{\bul}_{X/S^k}(\log D)
\lo 
{\cal E}\otimes_{{\cal O}_{\cal X}}\Om^{\bul}_{X/S}(\log D)\lo 0. 
\tag{6.14.4}\label{ali:eufbant}
\end{align*} 
Hence we have the following boundary morphism 
\begin{align*} 
N_k\col {\cal E}\otimes_{{\cal O}_{\cal X}}\Om^{\bul}_{X/S}(\log D) \lo
{\cal E}\otimes_{{\cal O}_{\cal X}}\Om^{\bul}_{X/S}(\log D). 
\tag{6.14.5}\label{ali:eufnt}
\end{align*} 

\begin{prop}\label{prop:monf}
The following diagram is commutative: 
\begin{equation*} 
\begin{CD}
{\cal E}\otimes_{{\cal O}_{\cal X}}\Om^{\bul}_{X/S}(\log D)
@>{N_k}>> 
{\cal E}\otimes_{{\cal O}_{\cal X}}\Om^{\bul}_{X/S}(\log D)\\
@A{\simeq}AA @AA{\simeq}A \\ 
{\cal E}\otimes_{{\cal O}_{\cal X}}\Om^{\bul}_{X/\os{\circ}{S}}(\log D)[U_S]
@>{-D_k}>> {\cal E}\otimes_{{\cal O}_{\cal X}}\Om^{\bul}_{X/\os{\circ}{S}}(\log D)[U_S].  
\end{CD} 
\tag{6.15.1}\label{cd:efnt}
\end{equation*}
\end{prop}
\begin{proof} 
Because the proof is standard, we omit the proof 
(cf.~\cite[(10.3)]{nhi}).  
\end{proof} 

Consider the local situation in the beginning of \S\ref{sec:da}. 
and the following morphism 
\begin{align*} 
L_S \owns \sum_{i=1}^rc_i u_i\lom \sum_{i=1}^rc_i (t_i+\al_i(\ul{t})) \in {\cal O}_{\ol{S}}
\tag{6.15.2}\label{ali:osials} 
\end{align*} 
instead of the morphism (\ref{ali:osils}), where $\al(\ul{t})\in \Gam(\ol{S},{\cal O}_{\ol{S}})$. 
Here $\ul{t}=(t_1,\ldots,t_r)$. 
Let 
\begin{align*}  
\sig_{V,\{\al_i(\ul{t})\}_{i=1}^r}\col p^*_{V}R^qf_{S*}({\cal E}\otimes_{{\cal O}_{\cal X}}\Om^{\bul}_{X/S}(\log D))
\os{\sim}{\lo}  
R^qf_*({\cal E}\otimes_{{\cal O}_{\cal X}}\Om^{\bul}_{{\cal X}/S}(\log {\cal D}))_V 
\tag{6.15.3}\label{cd:aabzxd}
\end{align*} 
be the isomorphism by the use of the morphism (\ref{ali:osials}) instead of 
the morphism (\ref{ali:osils}).

\begin{prop}\label{prop:ison}
The isomorphism 
\begin{align*}  
\sig_{V,\{\al_i(\ul{t})\}_{i=1}^r}\circ \sig_V^{-1} \col 
p^*_{V}R^qf_{S*}({\cal E}\otimes_{{\cal O}_{\cal X}}\Om^{\bul}_{X/S}(\log D))
\lo 
p^*_{V}R^qf_{S*}({\cal E}\otimes_{{\cal O}_{\cal X}}\Om^{\bul}_{X/S}(\log D)) 
\tag{6.16.1}\label{cd:amzxd}
\end{align*} 
is equal to 
${\rm id}_V\otimes{\rm exp}(-\sum_{l=1}^r\al_l(\ul{0})N_l)$. 
\end{prop} 
\begin{proof}
We have only to give a generalization of the argument in \cite[p.~248]{sti}. 
\par 
We use the multi-index notation $\ul{u}^{\ul{k}}:=u_1^{k_1}\cdots u_r^{k_r}$ 
$(k_1,\ldots, k_r\in {\mab N})$. Let ${\mathfrak U}$ be a Stein covering of $X$.
Set $\vert \ul{k}\vert =\sum_{j=1}^rk_j$. 
Let $C^{\bul}({\mathfrak U},
\iota^{-1}({\cal E}\otimes_{{\cal O}_{\cal X}}\Om^{\bul}_{X/\os{\circ}{S}}(\log D)))$ 
be the \v{C}ech complex.  
Let $\sum_{\vert \ul{k}\vert=0}^{\vert \ul{s}\vert}\ul{u}^{\ul{k}}\om_{\ul{k}}\in 
C^q({\mathfrak U},
\iota^{-1}({\cal E}\otimes_{{\cal O}_{\cal X}}
\Om^{\bul}_{{\cal X}/\os{\circ}{S}}(\log {\cal D})))$ 
$(q\in {\mab N})$ be a cocycle. Set $e_l:=(0,\ldots,0,\us{l}{1},0,\ldots,0)$. 
Then $d\om_{\ul{s}}=0$ and 
$d\om_{\ul{k}}+\sum_{l=1}^r(k_l+1)d\log t_l\wedge \om_{\ul{k}+e_l}=0$ $(\ul{k}<\ul{s})$. 
The last equation implies that 
$d\om_{\ul{k}}+(k_l+1)d\log t_l\wedge \om_{\ul{k}+e_l}=0$ in 
$C^q({\mathfrak U},
\iota^{-1}({\cal E}\otimes_{{\cal O}_{\cal X}}
\Om^{\bul}_{{\cal X}/S^l}(\log {\cal D})))$.  
Hence the image $\ol{\om}_{\ul{k}}$ of $\om_{\ul{k}}$ in 
$C^q({\mathfrak U},
\iota^{-1}({\cal E}\otimes_{{\cal O}_{\cal X}}\Om^{\bul}_{X/S}(\log D)))$ 
satisfies the following relations: 
$N_l^{s_l+1}(\ol{\om}_{\ul{0}})=0$ and 
$N_l(\ol{\om}_{\ul{k}})=-(k_l+1)\ol{\om}_{\ul{k}+e_l}$ $(\ul{k}<\ul{s})$.  
Hence $\ol{\om}_{\ul{k}}=-\sum_{l=1}^r\dfrac{N^{k_1}_1\cdots 
N^{k_r}_r(\ol{\om}_{\ul{0}})}{k_1!\cdots k_r!}$. 
Then 
\begin{align*}  
&\sig_{V,\{\al_i(\ul{t})\}_{i=1}^r}\circ \sig_V^{-1}
(\ol{\om}_{\ul{0}})=
\sig_{V,\{\al_i(\ul{t})\}_{i=1}^r}
(\sum_{\vert \ul{k}\vert=0}^{\vert \ul{s}\vert}\ul{u}^{\ul{k}}\ol{\om}_{\ul{k}})=
\sum_{\vert \ul{k}\vert=0}^{\vert \ul{s}\vert}(\ul{t}+\ul{\al}(\ul{t}))^{\ul{k}}\ol{\om}_{\ul{k}}
\vert_{\ul{t}=0}\\
&=-\sum_{l=1}^r\sum_{k_1,\ldots,k_r\geq 0}\al_1(\ul{0})^{k_1}\cdots \al_r(\ul{0})^{k_r}
\dfrac{N^{k_1}_1\cdots N^{k_r}_r(\ol{\om}_{\ul{0}})}{k_1!\cdots k_l!}\\
&={\rm exp}(-\sum_{l=1}^r\al_l(\ul{0})N_l)(\ol{\om}_{\ul{0}}). 
\end{align*} 
\end{proof}

\begin{coro}[{\bf Invariance theorem of the pull-backs of morphisms}]\label{coro:ucrm}
Let $B'$ and $S'$ be an analogous log analytic spaces to $B$ and $S$, respectively. 
Let $({\cal X}',{\cal D}')/B'$ 
be an analogous log analytic space to $({\cal X},{\cal D})/B$ 
and let $(X',D')$ be an analogous log analytic space to 
$(X,D)$ for $({\cal X}',{\cal D}')/B'$. 
%For $i=1,2$, let 
%\begin{equation*} 
%\begin{CD}
%({\cal X},{\cal D})@>{g_i}>> ({\cal X}',{\cal D}')\\
%@V{f}VV @VV{f'}V \\
%B @>{u}>> B'
%\end{CD}
%\end{equation*}
%be two commutative diagrams of log analytic spaces. 
Assume that we are given the following diagram 
\begin{equation*} 
\begin{CD}
(X,D)@>{g}>> (X',D')\\
@VVV @VVV \\
S@>{u\vert_S}>> S'\\
@V{\bigcap}VV @VV{\bigcap}V \\
B @>{u}>> B'. 
\end{CD}
\end{equation*}
%Let $t$ be a point of $S$. 
Let $({\cal E}',\nabla')$ be an analogous integrable connection to $({\cal E},\nabla)$ 
for ${\cal X}'/\os{\circ}{S}{}'$. 
Set $E^{\square}:={\cal E}^{\square}$, where $\square=$nothing or $'$. 
Let $\Phi \col g^*(E')\lo E$ 
be a morphism of ${\cal O}_X$ modules fitting into 
the following commutative diagram$:$
\begin{equation*} 
\begin{CD}
g^*(E')@>{}>> E\\
@V{g^*(\nabla')}VV @VV{\nabla}V \\
g^*(E'\otimes_{{\cal O}_{X'}}\Om^1_{X'/\os{\circ}{S}{}'}(\log D'))
@>>> E\otimes_{{\cal O}_X}\Om^1_{X/\os{\circ}{S}{}}(\log D). 
\end{CD}
\end{equation*}
%If the two diagrams above are the same on $X$, 
Then there exist $\eps>0$ and $\eps'>0$ 
such that $u(B(\eps))  \subset B'(\eps')$ 
such that there exists the following canonical morphism 
\begin{align*}
g(\eps,\eps')^*\col &u(\eps,\eps')^*
(R^qf'(\eps')_*({\cal E}'\otimes_{{\cal O}_{{\cal X}'}}\Om^{\bul}_{{\cal X}'(\eps')/B'(\eps')}
(\log {\cal D}'(\eps'))))\\
&\lo 
R^qf(\eps)_*({\cal E}\otimes_{{\cal O}_{\cal X}}\Om^{\bul}_{{\cal X}(\eps)
/B(\eps)}(\log {\cal D}(\eps))). 
\end{align*} 
%for $i=1,2$ are equal. 
Here $u(\eps,\eps')\col B(\eps)\lo B'(\eps')$ is the induced morphism
by $u$. The morphism $g(\eps,\eps')$ satisfies the usual transitive relation. 
If the morphism $g$ is lifted to a morphism 
$\wt{g}\col {\cal X}(\eps)\lo {\cal X}'(\eps')$ and if 
the morphism $\Phi$ is lifted to a morphism 
$\wt{\Phi} \col g^*({\cal E}')\lo {\cal E}$ 
be a morphism of ${\cal O}_{\cal X}$ modules fitting into 
the following commutative diagram$:$
\begin{equation*} 
\begin{CD}
\wt{g}^*({\cal E}'(\eps'))@>{}>> {\cal E}(\eps)\\
@V{g^*(\nabla')}VV @VV{\nabla}V \\
\wt{g}^*({\cal E}'\otimes_{{\cal O}_{{\cal X}'}}\Om^1_{{\cal X}'(\eps')/\os{\circ}{S}{}'}(\log {\cal D}'(\eps')))
@>>> {\cal E}\otimes_{{\cal O}_{\cal X}}\Om^1_{{\cal X}(\eps)/\os{\circ}{S}{}}(\log {\cal D}(\eps)), 
\end{CD}
\end{equation*}
where ${\cal E}^{\square}(\eps^{\square})=
{\cal E}^{\square}\vert_{{\cal X}^{\square}(\eps^{\square})}$, 
then 
\begin{align*}
g(\eps,\eps')^*=\wt{g}^*\col &u(\eps,\eps')^*
(R^qf'(\eps')_*({\cal E}'\otimes_{{\cal O}_{{\cal X}'}}\Om^{\bul}_{{\cal X}'(\eps')/B'(\eps')}
(\log {\cal D}'(\eps'))))\\
&\lo 
R^qf(\eps)_*({\cal E}\otimes_{{\cal O}_{\cal X}}\Om^{\bul}_{{\cal X}(\eps)
/B(\eps)}(\log {\cal D}(\eps))). 
\end{align*}
\end{coro}

\begin{coro}[{\bf Invariance of relative log de Rham complexes}]\label{coro:a}
Let the notations be as in {\rm (\ref{coro:ucrm})}. 
Assume that $g_1=g_2$ and $\Phi_1=\Phi_2$. 
Assume also that $(X',D')/S'=(X,D)/S$, 
that $u_S={\rm id}_S$, $g_1\vert_{(X,D)}={\rm id}_{(X,D)}$,
${\cal E}\vert_{X_S}={\cal E}'_{X_S}$  
and $\Phi_1\vert_{{\cal E}\vert_{X_S}}={\rm id}\vert_{{\cal E}\vert_{X_S}}$. 
Then there exists a positive integer $\eps$ such that 
\begin{align*} 
R^qf'(\eps)_*({\cal E}'\otimes_{{\cal O}_{{\cal X}'}}\Om^{\bul}_{{\cal X}'(\eps)/\ol{S}'(\eps)}
(\log {\cal D}'(\eps))))
=
R^qf(\eps)_*({\cal E}\otimes_{{\cal O}_{\cal X}}\Om^{\bul}_{{\cal X}(\eps)
/\ol{S}(\eps)}(\log {\cal D}(\eps))). 
\tag{6.18.1}\label{ali:mosuox}
\end{align*} 
\end{coro} 

By the proof of (\ref{theo:nm}) we obtain the following: 

\begin{coro}\label{coro:pis}
Let $s$ be an exact point of $S$.  
%Consider the case $\os{\circ}{S}=\{O\}$. 
For a sufficiently close point $u\in \ol{S}=\Del^r$ from $s$ 
such that $p(u)=\os{\circ}{s}$, 
there exists a canonical contravariantly functorial isomorphism 
\begin{align*} 
{\cal O}_{\ol{S},u}\otimes_{{\cal O}_{{\cal O}_{S,s}}}
R^qf_{S*}({\cal E}\otimes_{{\cal O}_{\cal X}}\Om^{\bul}_{X/S}(\log D))_s
\os{\sim}{\lo} 
R^qf_*({\cal E}\otimes_{{\cal O}_{\cal X}}\Om^{\bul}_{{\cal X}/B}(\log {\cal D})))_u. 
\tag{6.19.1}\label{ali:mxocox}
\end{align*} 
Consequently there exists a canonical contravariantly functorial isomorphism 
\begin{align*} 
R^qf_{S*}({\cal E}\otimes_{{\cal O}_{\cal X}}\Om^{\bul}_{X/S}(\log D))\otimes
_{{\cal O}_{S}}\kap(s)
\os{\sim}{\lo}  
R^qf_{u*}({\cal E}\otimes_{{\cal O}_{\cal X}}\Om^{\bul}_{{\cal X}_u/u}
(\log {\cal D}_u))). 
\tag{6.19.2}\label{ali:mcuox}
\end{align*} 
Here $f_u\col ({\cal X}_u,{\cal D}_u)\lo \{u\}$ is the structural morphism. 
\end{coro}

\begin{coro}\label{coro:sc}
Let the notations be as in {\rm (\ref{coro:pis})}. 
Assume that $u\in \ol{S}\setminus S$. 
%(1) 
%Consider the case $\os{\circ}{S}=\{O\}$. 
%Then $\ol{S}$ is the polydisc $\Del^r$ $(r\in {\mab Z}_{\geq 1})$. 
%In this case 
%$R^qf_*\iota_*({\cal E}\otimes_{{\cal O}_{\cal X}}\Om^{\bul}_{X/S}(\log D))$ 
%is a finite dimensional ${\mab C}$-vector space. 
%Set $(\Del^r)^*:=\{u \in \Del^r~\vert~M_u={\cal O}_u^*\}$ 
%and let $u$ be a point of $(\Del^r)^*$. 
%Let $O$ be the origin of $\Del^r$. Then $S$ is the log point of virtual dimension $r$ 
%whose underlying analytic space is $\{O\}$. 
Set ${\cal U}:={\cal X}\setminus {\cal D}$ and 
$V:={\rm Ker}(\nabla \col {\cal E}\vert_{{\cal U}_u}\lo 
{\cal E}\vert_{{\cal U}_u}\otimes_{{\cal O}_{{\cal U}_u}}\Om^1_{{{\cal U}_u}/{\mab C}})$.
If $u$ is sufficiently close from $s$, 
then there exists the following canonical isomorphisms 
\begin{align*} 
R^qf_{S*}({\cal E}\otimes_{{\cal O}_{\cal X}}\Om^{\bul}_{X/S}(\log D))\otimes
_{{\cal O}_{S}}\kap(s)
\os{\sim}{\lo}  R^qf_{u*}(\Om^{\bul}_{{\cal U}_u/u})
\os{\sim}{\longleftarrow} R^qf_{u*}(V). 
\tag{6.20.1}\label{ali:mcauox}
\end{align*} 
\end{coro}
\begin{proof} 
We obtain the first isomorphism in (\ref{ali:mcauox}) by \cite[II (3.13) (ii)]{de} 
%by (\ref{ali:oexs}),  
and 
the second isomorphism in (\ref{ali:mcauox}) by [loc.~cit., I (2.27.2)].
%and the third isomorphism in (\ref{ali:xsc}) .
%This is nothing but the generalization of the isomorphism (\ref{ali:dxa}). 
\end{proof}

\begin{coro}\label{coro:n}
{\rm (\ref{coro:rqe})} holds.
\end{coro}
\begin{proof}
The problem is local. Hence (\ref{coro:rqe}) follows from (\ref{theo:lff}). 
\end{proof}

\bigskip
\bigskip
\parno
Yukiyoshi Nakkajima 
\parno
Department of Mathematics,
Tokyo Denki University,
5 Asahi-cho Senju Adachi-ku,
Tokyo 120--8551, Japan. 
\parno
{\it E-mail address\/}: 
nakayuki@cck.dendai.ac.jp

\end{document}